\documentclass[12pt]{amsart}

\usepackage{graphicx}
\usepackage{amsmath}
\usepackage{amsfonts}
\usepackage{amssymb}
\usepackage{pstricks}
\usepackage{url}
\usepackage{nicefrac}
\usepackage{color}
\usepackage{subfigure}
\input xy
\xyoption{all}

\newtheorem{theorem}{Theorem}

\newtheorem{corollary}[theorem]{Corollary}
\newtheorem{definition}[theorem]{Definition}

\newtheorem{lemma}[theorem]{Lemma}

\newtheorem{proposition}[theorem]{Proposition}
\newtheorem{oldtheorem}{Theorem}

\theoremstyle{definition}
\newtheorem{remark}[theorem]{Remark}

\newcommand{\trace}{\mathrm{trace}}
\renewcommand{\ker}{\mathrm{ker}}
\newcommand{\im}{\mathrm{im}}
\newcommand{\Fix}{\mathrm{Fix}}
\newcommand{\Per}{\mathrm{Per}}
\renewcommand{\int}{\mathrm{int}}
\newcommand{\Inv}{\mathrm{Inv}}
\newcommand{\id}{\mathrm{id}}

\title{Infinite series in cohomology: Attractors and Conley index}

\author[L. Hern\'andez-Corbato]{Luis Hern\'andez-Corbato}
\author[F. R. Ruiz del Portal]{Francisco R. Ruiz del Portal}
\author[J. J. S\'anchez-Gabites]{J. J. S\'anchez-Gabites}

\thanks{The authors have been supported by MINECO grant MTM2015-63612-P}

\address{L. Hern\'andez Corbato \\ Departamento de Matem\'atica Aplicada a las TIC \\ Universidad Polit\'ecnica de Madrid \\  28031 Madrid \\ Spain}
\email{luis.hcorbato@upm.es}

\address{F. R. Ruiz del Portal \\ Departamento de \'Algebra, Geometr\'{\i}a y Topolog\'{\i}a\\ Universidad Complutense de Madrid \\ Plaza de Ciencias 3 \\ 28040 Madrid \\ Spain}
\email{rrportal@ucm.es}

\address{J. J. S{\'{a}}nchez-Gabites \\ Departamento de An\'alisis Econ\'omico (M\'etodos cuantitativos) \\ Facultad de Ciencias Econ\'omicas y Empresariales \\ Universidad Aut\'onoma de Madrid \\ 28049 Madrid \\ Spain}
\email{JaimeJ.Sanchez@uam.es}

\begin{document}

\begin{abstract} In this paper we study the cohomological Conley index of arbitrary isolated invariant continua for continuous maps $f \colon U \subseteq \mathbb{R}^d \to \mathbb{R}^d$ by analyzing the topological structure of their unstable manifold. We provide a simple dynamical interpretation for the first cohomological Conley index, describing it completely, and relate it to the cohomological Conley index in higher degrees. A number of consequences are derived, including new computations of the fixed point indices of isolated invariant continua in dimensions 2 and 3.

Our approach exploits certain attractor-repeller decomposition of the unstable manifold, reducing the study of the cohomological Conley index to the relation between the cohomology of an attractor and its basin of attraction. This is a classical problem that, in the present case, is particularly difficult because the dynamics is discrete and the topology of the unstable manifold can be very complicated. To address it we develop a new method that may be of independent interest and involves the summation of power series in cohomology: if $Z$ is a metric space and $K \subseteq Z$ is a compact, global attractor for a continuous map $g \colon Z \to Z$, we show how to interpret series of the form $\sum_{j \ge 0} a_j (g^*)^j$ as endomorphisms of the cohomology group of the pair $(Z,K)$.
\end{abstract}

\maketitle

\section{Introduction}

The Conley index is a powerful topological invariant of dynamical systems that has proved to be very useful. It can be thought of as a far-reaching generalization of the fixed point index to (almost) arbitrary compact invariant sets $X$, and in fact both indices satisfy formally similar properties. The only requirement on $X$ is that it be \emph{isolated} or \emph{locally maximal}, which means that it has a compact neighbourhood $N$ such that $X$ is the maximal invariant subset of $N$ (in that case \(N\) is called an \emph{isolating neighborhood} of \(X\)). The Conley index was first introduced by Conley \cite{conley1} in the context of continuous dynamical systems under the name ``Morse index'' and extended later on by several authors (see for instance \cite{franksricheson}, \cite{handbookconley}, \cite{Mro1}, \cite{robinsalamon1} or \cite{Szy2}),  to discrete dynamical systems. In either case the index is homotopical in nature, but it has homological and cohomological versions that are often more convenient to work with. For a complete introduction to the theory of Conley index we refer the reader to \cite{handbookconley}.

In this paper we shall be concerned with the (co)homological Conley index for a discrete dynamical system generated by the iteration of some continuous map $f \colon M \to M$. The precise definition is somewhat involved and will be recalled later on, but for the purposes of this Introduction the following rough idea will suffice: for each dimension $q = 0,1,2,\ldots$ the $q$th homological Conley index of an isolated invariant set $X$ is an endomorphism $h_q(f,X)$ of a finite dimensional vector space over a field $\mathbb{K}$; intuitively, $h_q(f,X)$ captures the action of $f$ on a neighbourhood of $X$ at the level of $q$--dimensional homology. The homology is computed with coefficients in $\mathbb{K}$, which we will take to be $\mathbb{Q}$ or $\mathbb{C}$. By considering the eigenvalues or the trace of $h_q(f,X)$ one obtains even simpler, yet still powerful, numerical invariants. Similar considerations apply to the $q$th cohomological Conley index $h^q(f,X)$.

The following two results from \cite{HCR} are our starting point (recall that a compact set $X$ is called \emph{acyclic} if its \v{C}ech homology groups of degree $q \ge 1$ vanish):

\begin{oldtheorem} \label{teo:A}
Let $f$ be a homeomorphism defined on an open subset $U$ of $\mathbb{R}^d$ and let $X \subseteq U$ be an isolated invariant acyclic continuum.
Then there exist a finite set $J$ and a map $\varphi\colon J \to J$ such that the trace of the first homological Conley index satisfies
$$\trace(h_1(f^n, X)) = -1 + \#\Fix(\varphi^n).$$
In particular, $\trace(h_1(f, X)) \ge -1$.
\end{oldtheorem}

\begin{oldtheorem} \label{teo:B}
Let $f$ be a homeomorphism defined on an open subset $U$ of $\mathbb{R}^d$ and let $X \subseteq U$ be an isolated invariant acyclic continuum.
If $\trace(h_{1}(f, X)) = -1$ then $\trace(h_{q}(f, X)) = 0$ for every $q > 1$.
\end{oldtheorem}

The rather difficult proofs of these theorems involve a combinatorial description of the first homological Conley index that is ultimately related to how $f$ acts on the components of the exit set of an isolating neighbourhood $N$ of $X$ (the exit set consists of the points in $N$ whose image under $f$ lies outside $\int(N)$). Several questions arise naturally:

\begin{itemize}
	\item[(i)] Is it possible to provide a dynamical interpretation of the map $\varphi$ and the set $J$?
	\item[(ii)] Is it possible to generalize Theorems \ref{teo:A} and \ref{teo:B} to non acyclic continua $X$? This suggests that we turn to the cohomological Conley index $h^q(f,X)$ and use \v{C}ech cohomology $\check{H}^*$, which is the preferred theory for sets with a complicated local structure, and indeed henceforth we shall do so. (By the universal coefficient theorem, acyclic sets also have vanishing \v{C}ech cohomology in every degree $q \geq 1$).
	\item[(iii)] Is it possible to extend the results for arbitrary continuous maps $f$, rather than homeomorphisms, and possibly more general phase spaces?
\end{itemize}

In this paper was to provide answers to all these questions, obtaining suitable generalizations of Theorems \ref{teo:A} and \ref{teo:B} and deriving from them a number of corollaries concerning the fixed point index of isolated invariant continua (these are discussed in Section \ref{sec:corollaries}). The original proofs in \cite{HCR} involved a very delicate ``woodworking'' with homology classes in an isolating neighbourhood of $X$, but here we shall adopt a radically different approach which is much more algebraic in nature. It accommodates naturally the non acyclic case and involves only dynamically meaningful objects that are canonically associated to $X$ (namely, its unstable manifold), providing a straightforward interpretation of $\varphi$ and $J$. In order to be able to carry out this approach we develop a notion of summation of infinite series in cohomology that may be of independent interest.

\subsection*{Main results} In order to state our main results we need to recall a construction due to Robbin and Salamon \cite{robinsalamon1}. We will be rather informal, postponing precise definitions to Section \ref{sec:background}. Given an invariant set $X$ for a homeomorphism, its \emph{unstable manifold} $W^u$ is defined as the subset of phase space consisting of those points whose negative semitrajectory approaches $X$, in the sense that it eventually enters any neighbourhood of $X$. The unstable manifold (which, in spite of its name, need not be a manifold) is an invariant set that contains $X$. Robbin and Salamon introduced the so-called \emph{intrinsic topology} on $W^u$, which is generally finer than the one that it inherits from the phase space and, roughly speaking, makes two points close if and only if their backward orbits remain close to each other for a large number of iterates. The unstable manifold endowed with the intrinsic topology has the following properties:
\begin{itemize}
	\item[(i)] $W^u$ is a locally compact, Hausdorff space. In particular it has an Alexandroff compactification $W^u_{\infty}$, whose point at infinity we shall denote by $\infty$.
	\item[(ii)] $f|_{W^u}$ generates a discrete dynamical system on $W^u$ which can be trivially extended to all of $W^u_{\infty}$ declaring $\infty$ to be a fixed point.
	\item[(iii)] The pair $(\infty,X)$ is an attractor-repeller Morse decomposition of $W^u_{\infty}$.
\end{itemize}

We will show in Section \ref{sec:background} that, under very general circumstances, the induced homomorphism $f^* \colon \check{H}^*(W^u_{\infty},\infty) \to \check{H}^*(W^u_{\infty},\infty)$ is equivalent to the cohomological Conley index of $X$. Notice that this expresses the cohomological Conley index in terms of a canonical object associated to $X$ (its unstable manifold), as mentioned earlier. This justifies our interest in the cohomology of $W^u_{\infty}$ and the action of $f^*$ on it.

Now consider the set $W^u_{\infty} \setminus (X \cup \infty)$. In general this set may have infinitely many connected components, or rather quasicomponents, as it will turn that it is advantageous to work with the latter. This owes to the fact that the unstable manifold may have very complicated topological structure, as illustrated by the construction given in \cite{Art}. These quasicomponents are typically called \emph{branches} of the unstable manifold. We say that a quasicomponent $F$ is \emph{essential} if its closure in $W^u_{\infty}$ has a nonempty intersection with $X$ and $\infty$; intuitively, it ``joins'' $X$ and $\infty$. Clearly, $f$ permutes the essential quasicomponents of $W^u_{\infty} \setminus (X \cup \infty)$ among themselves.

Since we want to consider the general case when $f$ is not necessarily a homeomorphism but merely a continuous map, the preceding discussion will need some revision. The construction of Robbin and Salamon carries over with appropriate modifications and leads naturally to the replacement of $f|_{W^u_{\infty}}$ and $X \subseteq W^u_{\infty}$ with a homeomorphism $\tilde{f}$ on $W^u_{\infty}$ and a compact invariant set $\tilde{X} \subseteq W^u_{\infty}$ in such a way that properties (i) to (iii) above remain true. These new objects $\tilde{f}$ and $\tilde{X}$ are closely related but not exactly equal to $f$ and $X$ but, since the differences are largely inconsequential, we suggest that the reader assumes $f$ to be a homeomorphism and simply ignores the tildes for the time being.

Now we can state two of the two main results in this paper, which describe the cohomology of $(W^u_{\infty},\tilde{X})$ in terms of the essential quasicomponents. In degree one it is actually convenient to use $\infty$ as a basepoint and work with the relative cohomology group $\check{H}^1(W^u_{\infty},\tilde{X} \cup \infty)$ because it has a more symmetric description than $\check{H}^1(W^u_{\infty},\tilde{X})$. We have the following:

\begin{theorem} \label{teo:main1_intro} Let $f$ be a continuous map and assume $X$ has finitely many connected components. Then there are only finitely many essential quasicomponents $F_1, \ldots, F_m$ in $W^u_{\infty} \setminus (\tilde{X} \cup \infty)$ and the cohomology group $\check{H}^1(W^u_{\infty},\tilde{X} \cup \infty)$ can be identified with the $\mathbb{K}$--vector space having the set $\{F_j\}_{j=1}^m$ as a basis. Moreover, under this identification the map $\tilde{f}^*$ simply becomes the permutation $F_j \longmapsto \tilde{f}^{-1}F_j$.
\end{theorem}

This result suggests that we introduce the following notation: $J$ will denote the set of essential quasicomponents of $W^u_{\infty} \setminus (\tilde{X} \cup \infty)$ and $\varphi \colon J \to J$ will denote the permutation induced by $\tilde{f}^{-1}$ on $J$.

For higher degrees, our second main theorem roughly states that the cohomology of $(W^u_{\infty},\tilde{X})$ is entirely determined by that of the essential quasicomponents alone:

\begin{theorem} \label{teo:main2_intro} Let $f$ be a continuous map and assume $X$ has finitely many connected components. In any degree $q \geq 1$ such that $\check{H}^q(X)$ is finite dimensional, there is an isomorphism $\check{H}^{q+1}(W^u_{\infty},\tilde{X}) = \oplus_{j=1}^m \check{H}^{q+1}(\overline{F}_j)$ that commutes with $\tilde{f}^*$. Here $\overline{F}_j$ denotes the closure of $F_j$ in the quotient space $W^u_{\infty}/\tilde{X}$ that results from $W^u_{\infty}$ by collapsing $\tilde{X}$ to a single point.
\end{theorem}

The description given by the theorems above holds for the iterates $f^n$ ($n \geq 1$) of $f$ as well. Upon replacing $f$ with $f^n$ the set $X$ is still isolated for this new dynamical system, the compactified unstable manifold $W^u_{\infty}$ and its subset $\tilde{X}$ remain essentially the same, and the homeomorphism $\widetilde{f^n}$ induced on $W^u_{\infty}$ can be identified with $\tilde{f}^n$. (This is almost immediate when $f$ is a homeomorphism but needs a more careful proof for continuous maps. Statements can be found in Subsection \ref{subsec:convention}). As a consequence, the set of essential quasicomponents for $f^n$ can be identified with $J$ and the permutation induced by $\widetilde{f^n}^{-1}$ on $J$ is then $\varphi^n$. With these changes, Theorems \ref{teo:main1_intro} and \ref{teo:main2_intro} extend to $f^n$.

It is clear from Theorem \ref{teo:main1_intro} that the trace of $\tilde{f}^*$ on $\check{H}^1(W^u_{\infty},\tilde{X} \cup \infty)$ is the number of essential quasicomponents that are fixed by $\tilde{f}^{-1}$ or, equivalently, the number $\#\Fix(\varphi)$ of fixed points of $\varphi$. Since the trace of the $1$--cohomological Conley index of $X$ coincides with that of $f^* \colon \check{H}^1(W^u_{\infty},\infty) \to \check{H}^1(W^u_{\infty},\infty)$, it is just a matter of using the long exact sequence of the triple $(W^u_{\infty},\tilde{X} \cup \infty,\infty)$ and the additivity of the trace to obtain the following generalization of Theorem \ref{teo:A}:

\begin{corollary} \label{coro:genA_intro} Let $f$ be a continuous map and let $X$ be an arbitrary isolated invariant continuum. Denote by $i^*$ the map induced by the inclusion $i\colon \tilde{X} \subseteq W^u_{\infty}$ in one-dimensional cohomology. Then the trace of the first cohomological Conley index of $X$ is given by $$\trace(h^1(f^n, X)) = -1 + \#\Fix(\varphi^n) + \trace((\tilde{f}^*)^n|_{  {\rm im}\ i^* })$$ unless $X$ is an asymptotically stable attractor, in which case only the third summand should be retained.
\end{corollary}

If $\check{H}^1(X) = 0$ (in particular, if $X$ is acyclic) then the third term vanishes, yielding an expression which is formally identical to that of Theorem \ref{teo:A}. In our context, however, we have a clear dynamical interpretation for both $\varphi$ and $J$.

The eigenvalues of $h^1(f,X)$ can also be easily described in similar terms. Since the set $J$ of essential quasicomponents is finite, the action of $\varphi$ partitions it into disjoint, irreducible cycles of the form $F_{i_1} \rightarrow F_{i_2} \rightarrow \ldots \rightarrow F_{i_r} \rightarrow F_{i_1} \rightarrow \ldots$ It is straightforward to check from Theorem \ref{teo:main1_intro} that each of these contributes a $\lambda^r - 1$ factor to the characteristic polynomial of $\tilde{f}^*$ on $\check{H}^1(W^u_{\infty},\tilde{X} \cup \infty)$. Thus, when coefficients are taken in $\mathbb{K} = \mathbb{C}$ each cycle of quasicomponents of minimal period $r$ contributes $r$ eigenvalues of $\tilde{f}^*$ that constitute precisely a complete set of $r$th roots of unity. Again through the exact sequence for the pair $(W^u_{\infty},\tilde{X} \cup \infty)$ it follows that every eigenvalue of $h^1(f,X)$ is either one of $\tilde{f}^*|_{{\rm im}\ i^*}$ or one of the roots of unity just described.

The generalization of Theorem \ref{teo:B} reads as follows:

\begin{corollary} \label{coro:genB_intro} Let $f$ be a continuous map and let $X$ be an arbitrary isolated invariant continuum. Denote by $(i^*)^q$ and $(i^*)^{q-1}$ the homomorphisms induced by the inclusion $\tilde{X} \subseteq W^u_{\infty}$ in $q$-- and $(q-1)$--dimensional cohomology. Assume that $\Fix(\varphi^n) = \emptyset$; that is, $\tilde{f}^n$ sends every essential quasicomponent onto a different one. Then \[\trace(h^q(f^n,X)) = \trace((\tilde{f}^*)^n|_{{\rm im}\ (i^*)^q}) + \trace((\tilde{f}^*)^n|_{{\rm im}\ (i^*)^{q-1}}) - \trace((f^*)^n|_{\check{H}^{q-1}(X)})\] in any degree $q > 1$ such that $\check{H}^q(X)$ and $\check{H}^{q-1}(X)$ are finitely generated.
\end{corollary}

The proof of Corollary \ref{coro:genB_intro} is a simple consequence of Theorem \ref{teo:main2_intro} but involves a minor technical point concerning the relation between $X$ and $\tilde{X}$. We shall discuss this in detail in Section \ref{sec:background} and prove Corollaries \ref{coro:genA_intro} and \ref{coro:genB_intro} in Section \ref{sec:proofs34}.

Corollary \ref{coro:genB_intro} may not be very appealing aesthetically, but it has a neat consequence concerning the fixed point index of $f$ on $X$. This very classical invariant is an integer $i(g,Z)$ that provides an algebraic measure of the amount of fixed points that a continuous map $g \colon Z \to Z$ has. In particular, a non--zero index implies the existence of fixed points. The celebrated Lefschetz--Hopf theorem states that for compact triangulable spaces $Z$ the index $i(g,Z)$ is equal to the so-called \emph{Lefschetz number} of $g$, which is defined as the alternated sum of the traces of the maps induced by $g$ on the singular (co)homology groups $H_q(Z)$ and usually denoted by $\Lambda(g)$. This result was later on extended by Lefschetz himself to the case where $Z$ is a compact absolute neighborhood retract (see \cite{Brown} or \cite{JM}). It is known, however, that it does not generally hold when $Z$ has bad local topological features. The interest of the following corollary is that it provides another instance where the Lefschetz--Hopf theorem holds even though no direct assumption about the topology of $X$ is made; in fact, the latter may well be very complicated:

\begin{corollary} \label{coro:Lef} Let $f$ be a continuous map defined in a manifold and let $X$ be an isolated invariant continuum having finitely generated \v{C}ech cohomology.
Then, the (\v{C}ech) Lefschetz numbers $\Lambda(f^n|_X) = \sum_{q \ge 0} (-1)^q \trace((f^*_q)^n)$, where $f^*_q$ denotes the map induced by $f$ in the $q$th \v{C}ech cohomology groups of $X$, are well defined.
Assume that $\Fix(\varphi) = \emptyset$, so that $\tilde{f}$ sends every essential quasicomponent onto a different one, then $i(f,X) = \Lambda(f|_X)$.

More generally, if we denote by $d$ the greatest common divisor of the periods of the action of $\varphi$ in the set of essential quasicomponents, then $i(f^n,X) = \Lambda(f^n|_X)$ if $n$ is not multiple of any of those periods and $i(f^n, X) \equiv \Lambda(f^n|_X) \mod d$ otherwise.
\end{corollary}

Notice that the condition that $\Fix(\varphi) = \emptyset$ is automatically satisfied when $X$ is an asymptotically stable attractor, since in that case the compactified unstable manifold $W^u_{\infty}$ reduces to the disjoint union of $\tilde{X}$ and $\infty$, there are no essential quasicomponents and the hypotheses are trivially satisfied.  Thus for attractors the proposed generalization of Lefschetz--Hopf theorem holds true.

As mentioned earlier, the action of $\varphi$ partitions $J$ into disjoint irreducible cycles of the form $F_{i_1} \rightarrow F_{i_2} \rightarrow \ldots \rightarrow F_{i_r} \rightarrow F_{i_1} \rightarrow \ldots$ Written out in full, this means that $\tilde{f}^{-1}(F_{i_j}) = F_{i_{j+1}}$, so in particular all the quasicomponents involved in a cycle are homeomorphic to each other. This fact can be exploited in several ways to extract information about certain numerical invariants. For instance, $\trace((f^*)^n|_{\check{H}^q(W^u_{\infty}, \tilde{X})}) \equiv 0 \mod d$, for $q \ge 1$, and this implies that equations in Corollaries \ref{coro:genA_intro} and \ref{coro:genB_intro} are valid $\mod d$ no matter how $\varphi$ acts on $J$. This consequence can be easily tracked in the proofs of those corollaries in Section \ref{sec:proofs34}. The second statement of Corollary \ref{coro:Lef} provides an example of this technique (another example is furnished by Theorem \ref{teo:chiX} below):

\begin{proof}[Proof of Corollary \ref{coro:Lef}] This result is a direct application of the Lefschetz--like formula (see \cite{Franks2}, \cite{HCR}, \cite{LCRPS} or \cite{mrozekLefschetz}) \[i(f^n, X) = \sum_{q \ge 0}(-1)^q \trace(h^q(f^n, X)).\] Upon replacing the formulae for $\trace(h^q(f^n,X))$ given in Corollaries \ref{coro:genA_intro} and \ref{coro:genB_intro} and using the remark above a telescopic cancellation ($\mod d$ if $\gcd(n,d) \neq 1$) occurs that eliminates all the traces of $(\tilde{f}^*)^n|_{\im\ (i^*)^q}$, yielding the equality $i(f^n,X) = \Lambda(f^n|_X)$ (valid $\mod d$ if $\gcd(n,d) \neq 1$). The case in which $X$ is an attractor requires special treatment because the term $\trace(h^0(f^n, X))$ does not vanish, it is equal to 1 (cf. Section \ref{sec:corollaries}), but this contribution cancels with the particular formula in Corollary \ref{coro:genA_intro}. 
\end{proof}

Lefschetz--Hopf theorem sharpens the celebrated original result of Lefschetz, usually referred to as Lefschetz theorem, which states that if $\Lambda(g) \neq 0$ then $\Fix(g) \neq \emptyset$. This theorem has been proven in several instances but does not hold for general arbitrary continua (non--ANR). In 1935 Borsuk \cite{Bor01} (see also \cite{BarSad}, \cite{Bing}, \cite{Hag} and references therein) constructed a locally connected acyclic continuum $K$ of \({\mathbb R}^3\) without the fixed point property, that is, such that there is a continuous map $g \colon K \to K$ without fixed points. From the previous discussion we learn, for instance, that it is impossible to embed $K$ (or any other counterexample to the Lefschetz theorem) in $\mathbb{R}^n$ in such a way that $g$ can be extended to a continuous map $f$ defined on a neighborhood of $K$ and such that $K$ is an attractor for $f$.
Indeed, suppose such an embedding $e$ would be possible. Then $f$ and $K$ would satisfy the hypothesis of Corollary \ref{coro:Lef} and the condition $\Fix(\varphi) = \emptyset$ would be automatically satisfied so $i(f, K) = \Lambda(f|_K) = \Lambda(g) \neq 0$, which would imply that $f$ and thus $g$ has fixed points in $K$.


\medskip


\begin{theorem} \label{teo:chiX} Let $X$ be an isolated invariant continuum for a homeomorphism $f$. Denote by $d$ the greatest common divisor of the periods of the essential quasicomponents of $W^u_{\infty} \setminus (X \cup \infty)$. Assume $X$ has finitely generated cohomology in all higher degrees so that $\chi(X)$ is well defined. Then \[\chi(W^u_{\infty},\infty) \equiv  \chi(X)\ \mod d.\]
\end{theorem}

In specific examples the Conley index is usually explicitly computable whereas the isolated invariant set $X$ itself might be more elusive. We shall show in Section \ref{sec:background} that $\chi(W^u_{\infty},\infty)$ can in turn be computed from the Conley index of $X$, and so the theorem above shows that it is possible to extract some information about the Euler characteristic of the ``unobservable'' $X$ from its ``observable'' Conley index.

\begin{proof}[Proof of Theorem \ref{teo:chiX}] (We do not need to distinguish between $X$ and $\tilde{X}$ or $f$ and $\tilde{f}$ because $f$ is assumed to be a homeomorphism). Clearly the number of essential quasicomponents in $W^u_{\infty} \setminus (X \cup \infty)$ must be divisible by $d$, and then Theorem \ref{teo:main1_intro} implies that $\beta^1(W^u_{\infty},X \cup \infty) \equiv 0 \ \mod d$. Here $\beta^q$ denotes the $q$th Betti number; that is, the dimension of the corresponding \v{C}ech cohomology group (vector space) $\check{H}^q$. As for the higher Betti numbers, for $q \geq 1$ we have $\beta^{q+1}(W^u_{\infty},X \cup \infty) = \beta^{q+1}(W^u_{\infty},X) =  \sum_i \beta^{q+1}(\overline{F}_i)$ where the last equality follows from Theorem \ref{teo:main2_intro}. Since all the essential quasicomponents in a cycle are homeomorphic to each other and so are their closures in $W^u_{\infty}/X$, the $\overline{F}_i$ all have the same Betti numbers and contribute equally to the above sum. The number of summands corresponding to each cycle is divisible by $d$, so it follows that $\beta^{q+1}(W^u_{\infty},X \cup \infty) \equiv 0 \ \mod d$ too. Finally, $\beta^0(W^u_{\infty},X \cup \infty) = 0$ because, by Proposition \ref{prop:reach}, each quasicomponent of $W^u_{\infty} \setminus (X \cup \infty)$ reaches $X$ or $\infty$ (or both, if essential). Thus $\chi(W^u_{\infty},X \cup \infty) \equiv 0 \ \mod d$. The theorem follows from this and the additivity of the Euler characteristic, which implies that $\chi(W^u_{\infty},X \cup \infty) = \chi(W^u_{\infty},\infty) - \chi(X)$.
\end{proof}

Most of the results in this paper are stated for connected $X$ for simplicity. However, since all of them follow from Theorems \ref{teo:main1_intro} and \ref{teo:main2_intro}, which are valid when $X$ has finitely many connected components, they all have appropriate generalizations to this slightly more general case. However, a qualitative leap takes place when one considers sets $X$ having infinitely many connected components: as the reader will see there arise difficulties which, although algebraic in nature, seem to reflect new dynamical phenomena that cannot occur when $X$ is connected or has only finitely many connected components. For instance, we will prove the following result:

\begin{theorem} \label{teo:components} Assume that the first cohomological Conley index $h^1(f,X)$ has a nonzero eigenvalue $\lambda \in \mathbb{C}$ that is not a root of unity. Then at least one of the following holds:
\begin{itemize}
	\item[(i)] $X$ has infinitely many connected components.
	\item[(ii)] $f^* \colon \check{H}^1(X;\mathbb{C}) \to \check{H}^1(X;\mathbb{C})$ has $\lambda$ as an eigenvalue.
\end{itemize}

In particular, if the phase space is $\mathbb{R}^2$ and $f$ is a homeomorphism, then $X$ must have infinitely many connected components because (ii) cannot hold.
\end{theorem}

Unfortunately our present understanding of the interplay between algebra and dynamics when $X$ has infinitely many connected components is still rather poor, so we will have to content ourselves with a characterization of the eigenvalues and eigenvectors of $\tilde{f}^*$ on $\check{H}^1(W^u_{\infty},\tilde{X} \cup \infty)$ (not nearly as neat as that of Theorem \ref{teo:main1_intro}) and a proof of Theorem \ref{teo:components}.

\subsection*{Our approach to the problem} Having given a taste of the sort of results that will be obtained in this paper, we will now present the general reasoning behind the proof of Theorem \ref{teo:main1_intro} (we warn the reader that the ideas presented here will turn out to be too simplistic and are therefore not in their final form). For this informal discussion we shall dispense with the notational distinction between $X$, $\tilde{X}$, $f$ and $\tilde{f}$.

Begin by writing $W^u_{\infty}$ as the union of the two open invariant sets \[U := W^u_{\infty} \setminus \infty \quad\quad \text{and} \quad\quad V := W^u_{\infty} \setminus X\] and consider the associated Mayer--Vietoris sequence \begin{equation} \label{eq:mvietoris_mot0} \xymatrix{ \ldots & \ar[l] \check{H}^{q+1}(W^u_{\infty},\infty) & \ar[l]_-{\Delta} \check{H}^q(U\cap V) & \ar[l] \check{H}^q(U) \oplus \check{H}^q(V,\infty) & \ar[l] \ldots} \end{equation} In the absence of hyperbolicity the local topology of all the spaces involved in the sequence can be very complicated, which further justifies our use of \v{C}ech cohomology.

Recall that $W^u_{\infty}$ has a Morse decomposition consisting of the repeller $X$ and the attractor $\infty$, and notice that $U$ is the basin of repulsion of $X$ and $V$ is the basin of attraction of $\infty$. Now, a classical result of Hastings \cite{hastings1}, later on generalized by other authors (see for instance \cite{sanjurjo3}, \cite{gunthersegal1} or  \cite{kapitanski1}), implies that the inclusion of an attractor of a \emph{flow} in its basin of attraction induces isomorphisms in cohomology. We emphasize the fact that this result is valid for continuous dynamics, and simple examples show that it is not generally true for discrete dynamics (this will be explained in more detail later on). Nevertheless, let us take it as a heuristic guide and accept that we can use it in our situation. Then, observing that $V$ is the basin of attraction of the single point $\infty$, we may expect that its (reduced) cohomology should vanish. Similarly, the cohomology of $U$ should be the same as that of $X$, and therefore the above sequence could be rewritten as \begin{equation} \label{eq:mvietoris_mot} \xymatrix{ \ldots & \ar[l] \check{H}^{q+1}(X) & \ar[l] \check{H}^{q+1}(W^u_{\infty},\infty) & \ar[l]_-{\Delta} \check{H}^q(U\cap V) & \ar[l] \check{H}^q(X) & \ar[l] \ldots} \end{equation}

Writing a second copy of this same sequence below it and connecting the two with vertical arrows given by $f^*$ (since each term in the sequence is the cohomology of an $f$--invariant set, this construction makes sense), the naturality of the Mayer--Vietoris sequence implies that the resulting diagram is commutative and therefore relates the cohomological Conley index in degree $q+1$ to the action of $f$ on $\check{H}^{q+1}(X)$ and on $\check{H}^q(U \cap V)$, which is one degree lower. In particular, for $q = 0$ this suggests that the cohomological Conley index in degree $1$ is related to the action of $f^*$ on $\check{H}^1(X)$ and on $\check{H}^0(U \cap V)$. At least intuitively, the latter should have a description of some sort in terms of permutations of the connected components of the set $U \cap V$, which is precisely $W^u_{\infty} \setminus (X \cup \infty)$. This is reminiscent of the content of Theorem \ref{teo:main1_intro} above.


If the dynamics were given by a flow, the argument outlined would be perfectly valid and straightforward to formalize. However, in our present case of discrete dynamics we need to:
\begin{itemize}
	\item[(i)] Find a manageable description of $\check{H}^0(U \cap V)$.
	\item[(ii)] Account for the fact that, in general, the inclusion of an attractor in its basin of attraction does not induce isomorphisms in \v{C}ech cohomology in discrete dynamics.
\end{itemize}

Addressing (i) requires a moderate amount of work but is not especially hard because, looking back at the exact sequence \eqref{eq:mvietoris_mot0}, all that is needed is a description of the quotient of $\check{H}^0(U \cap V)$ modulo the image of $\check{H}^0(U) \oplus \check{H}^0(V,\infty)$, which turns out to be rather manageable.

By contrast, (ii) poses a serious difficulty indeed because the discrete counterpart of Hastings' theorem is violated even in very simple situations. For instance, consider the homeomorphism $g(x) := x/2$ defined on $\{0\} \cup \{2^k : k \in \mathbb{Z}\}$: clearly $\{0\}$ is a global attractor for $g$, but the $0$--dimensional cohomology of the attractor and its basin of attraction obviously do not coincide. (One can easily modify and generalize this example to worsen the situation). There are several papers in the literature that obtain discrete analogues of Hastings' theorem (see \cite{gobbino1}, \cite{moronpaco1}, \cite{pacoyo1}, \cite{mio4}) but they all require some sort of ``niceness'' condition on the topology of the basin of attraction; at the very least that it be an absolute neighbourhood retract. It is clear that in our context we cannot expect $U$ and $V$ to meet any such condition (we will illustrate this with the well known Smale horseshoe in Section \ref{sec:background}) and so we will need to make a detour to study, as a general problem of independent interest, the relation between the cohomology of an attractor and its basin of attraction in arbitrary spaces. This will lead us, through the heuristic argument described in the following paragraph, to introduce \emph{power series} in cohomology.

Consider a continuous map $g$ having an attractor $K$ with basin of attraction $Z$ (for the situation considered above we would let $g = f$, $Z = V$ and $K = \infty$ or $g = f^{-1}$, $Z = U$ and $K = X$) and suppose we want to prove that the inclusion $K \subseteq Z$ induces isomorphisms in cohomology or, equivalently, that $\check{H}^*(Z,K) = 0$. We may be more modest and try to show that $g^*$ has no (nonzero) eigenvalues on $\check{H}^*(Z,K)$. This amounts to proving that $g^* - \lambda \cdot {\rm Id}$ is injective for $\lambda \neq 0$, which would certainly by true if $g^* - \lambda \cdot {\rm Id}$ had an inverse. And indeed, we can prove that this is the case by playing a notational game: \begin{equation} \label{eq:crazy} (g^*-\lambda \cdot {\rm Id})^{-1} = \frac{1}{g^* - \lambda \cdot {\rm Id}} = -\frac{1}{\lambda}\cdot  \frac{1}{{\rm Id}-\frac{1}{\lambda}\, g^*} = -\frac{1}{\lambda}\cdot \sum_{j=0}^{\infty} \frac{1}{\lambda^j} (g^*)^j \end{equation} where in the last step we have used the formula for the summation of a geometric series. What we will see in Section \ref{sec:series} is that these purely formal manipulations are actually perfectly legitimate. We will show how to assign to any formal power series $A(x) = \sum_{j=0}^{\infty} a_j x^j$ an endomorphism $A(f^*)$ of $\check{H}^*(Z,K)$ in a way that is consistent with the usual operations of formal addition and multiplication of power series so that the following relations hold: \begin{equation} \label{eq:ring_hom_intro} (A+B)(g^*) = A(g^*)+B(g^*) \quad \quad \text{and} \quad\quad (AB)(g^*) = A(g^*) \circ B(g^*).\end{equation} In this framework, the computations performed in \eqref{eq:crazy} can be formalized by considering the (finite) power series $A(x) := x - \lambda$ and its formal inverse $B(x) = -\nicefrac{1}{\lambda} \sum_{j=0}^{\infty} \nicefrac{x^j}{\lambda^j}$ and observing that the relations $A(x) B(x) = B(x) A(x) =  1$ imply, using \eqref{eq:ring_hom_intro}, that $A(g^*) \circ B(g^*) = B(g^*) \circ A(g^*) = {\rm Id}$ so that $B(g^*)$ really is an inverse for $g^* - \lambda \cdot {\rm Id}$.  For later reference we state this result as a corollary:

\begin{corollary} \label{cor:no_eigen} $g^*$ has no nonzero eigenvalues on $\check{H}^*(Z,K)$.
\end{corollary}

As a consequence we obtain a one-line proof for the classical result of Hastings about flows mentioned earlier and the various extensions to more general situations (semiflows and arbitrary phase spaces) that are discussed in the literature.

\begin{corollary} \label{cor:flows}
If $K$ is an attractor for a flow or a semiflow, the inclusion $K \subseteq Z$ induces isomorphisms in (\v{C}ech) cohomology.
\end{corollary}
\begin{proof} Let $g$ be the time--one map of the flow (or semiflow). The flow (or semiflow) provides a homotopy $H \colon (Z,K) \times [0,1] \to (Z,K)$ from $g \colon (Z,K) \to (Z,K)$ to the identity. Thus $\lambda = 1$ is an eigenvalue of $g^* \colon \check{H}^*(Z,K) \to \check{H}^*(Z,K)$, which contradicts the above unless $\check{H}^*(Z,K) = 0$.
\end{proof}

Notice that, for $\lambda = 1$, both $g^* - \lambda \cdot {\rm Id}$ and the series for $(g^* - \lambda \cdot {\rm Id})^{-1}$ obtained in Equation \ref{eq:crazy} have coefficients in $\mathbb{Z}$; therefore, Corollary \ref{cor:flows} is valid for cohomology with integer coefficients or, in fact, with coefficients in any commutative ring with unit. Similarly, the proof shows that the role of the flow (or semiflow) is just to provide a homotopy between $g$ and the identity; hence, the corollary is also valid for any discrete dynamical system whose generator $g$ is homotopic to the identity on the pair $(Z,K)$.

Returning back to our original problem of showing that the inclusions $X \subseteq U$ and $\infty \subseteq V$ induce isomorphisms in cohomology, it will turn out that this is not necessarily true, which is not surprising given the rather fragile status of the discrete counterpart of Hastings' theorem. However, the use of series will allow us to perform a more detailed analysis of the situation and conclude that certain connecting homomorphisms in (a variant of) the exact sequence \eqref{eq:mvietoris_mot0} are zero, and this will be enough for our general strategy, after some modifications, to go through.

\subsection*{Assumptions on the phase space} We have deliberately omitted so far any assumptions regarding the phase space $M$ where the dynamics takes place, and we close this Introduction with a brief discussion of this topic. The phase space will always be assumed to be paracompact and Hausdorff. Besides that, the \emph{only} additional condition that is required for Theorems \ref{teo:main1_intro} and \ref{teo:main2_intro} to hold is that $\check{H}^*(W^u_{\infty};\mathbb{K})$ be finitely generated. This is ensured as soon as any of the following conditions is satisfied:
\begin{itemize}
	\item[(i)] $X$ has an index pair $(N,L)$ such that $\check{H}^*(N/L;\mathbb{K})$ is finitely generated.
	\item[(ii)] The phase space $M$ is a locally compact, metrizable absolute neighbourhood retract (ANR).
	\item[(iii)] The phase space $M$ is a topological manifold.
\end{itemize}

Each of these conditions implies the previous one. In the results of Section \ref{sec:corollaries} we will need to restrict ourselves to phase spaces that are topological manifolds to ensure that certain duality property of the Conley index is satisfied; other than that we will simply assume tacitly that $M$ is such that at least condition (i) is satisfied.

\subsection*{Structure of the paper} In an attempt to make the paper easier to read we have not followed a strictly logical order, but rather tried to present the ideas in an order that makes they seem as motivated and natural as possible. We begin with Section \ref{sec:corollaries}, where some results concerning the fixed point index are derived from the main theorems stated above. Section \ref{sec:background} reviews the basics of the Conley index and the construction of the compactified unstable manifold, its intrinsic topology, and several other background results that are needed later on. Even if the reader is familiar with the construction of the intrinsic topology it would be advisable to read this section, since the case of continuous maps (rather than homeomorphisms) requires some modifications. Section \ref{sec:proofs34} llustrates how the generalizations of Theorems \ref{teo:A} and \ref{teo:B} that we are looking for (Corollaries \ref{coro:genA_intro} and \ref{coro:genB_intro}) follow as simple algebraic exercises from the main Theorems \ref{teo:main1_intro} and \ref{teo:main2_intro}. Section \ref{sec:series} introduces the powerful tool of series in cohomology and, by way of example, it also includes some versions of Hastings' theorem for discrete dynamical systems that are not used elsewhere in the paper. Finally, having all the necessary elements at our disposal, the strategy using the Mayer--Vietoris sequence described above will be carried out without too much effort in Section \ref{sec:describe1}, proving Theorems \ref{teo:main1_intro} and \ref{teo:main2_intro}. A number of technical results are postponed to three appendices to avoid disrupting the exposition along the paper.

\section{Applications to the fixed point index} \label{sec:corollaries}

The intrinsic understanding of the topology of the unstable manifold of an isolated invariant set allowed us to give a description of the cohomological Conley index in Corollaries \ref{coro:genA_intro} and \ref{coro:genB_intro}. In this section we exploit these results to prove statements about the fixed point index of isolated invariant continua. We will make use of some basic concepts and results from Conley index theory. The reader unfamiliar with these will find the necessary information (or appropriate references) in Section \ref{sec:background}.

There are two special types of invariant sets that sometimes have to be treated separately. A compact invariant set $X$ is called an (asymptotically stable) \emph{attractor} if the following two conditions are satisfied: (i) $X$ has a neighbourhood $V$ such that for every $p \in V$ the forward iterates of $p$ approach $X$ indefinitely; that is, they eventually enter any prescribed neighbourhood of $X$, and (ii) $X$ has a neighbourhood basis that consists of positively invariant sets. The maximal neighbourhood $V$ for which (i) is satisfied is called the \emph{basin of attraction} of $X$ and will be denoted by $\mathcal{A}(X)$. It is an open, positively invariant set. When the phase space is locally compact the positively invariant neighbourhoods mentioned in (ii) can be chosen to be compact, and any sufficiently small such neighbourhood $N$ turns out to be an isolating neighbourhood for $X$ (in fact, $X = \cap_{n \geq 0} f^n(N)$), so that $X$ is an isolated invariant set. The 0--cohomological Conley index of $X$ is then very easy to describe: if $X$ is a connected attractor and $N$ is a connected positively invariant isolating neighborhood as above, $(N, \emptyset)$ is an index pair for $f$ (and $f^n$) and $X$. Conversely, the existence of such an index pair $(N,L)$ where $N$ is connected and $L$ is empty implies that $X$ is a connected attractor. As a consequence (cf. \cite[Subsection 3.2]{HCR}), $h^0(f^n, X)$ is trivial except when $X$ is an attractor and in this case $\trace(h^0(f, X)) = 1$.

On the opposite end, if a compact invariant set $X$ has a neighborhood $N$ such that $N \subseteq f(N)$ and $X = \Inv(N)$ we say $X$ is a \emph{repeller}. This condition is stronger if $f$ is not injective than the usual $X = \cap_{n \ge 0} f^{-n}(N)$ (see Lemma \ref{lem:invm=inv}).
In the case $f$ is a homeomorphism, it can be shown (cf. \cite[Prop. 4]{LCRPS}) that $h^d(f, X)$ vanishes unless $X$ is a repellor and in this case $\trace(h^d(f, X)) = \pm 1$ depending on whether $f$ preserves or reverses orientation.

The computation or estimation of the fixed point indices of the iterates of a homeomorphism or a continuous map at an isolated fixed point, or more generally an isolated invariant set \(X\), is a difficult problem which reflects important information about the dynamics around \(X\) specially in low dimensions. The above theorems provide techniques to establish some new results and to obtain new proof of some strong known theorems. The link between the fixed point and the Conley indices is provided by the Lefschetz--like formula (see \cite{mrozekLefschetz}, \cite{Franks2},  \cite{LCRPS} or \cite{HCR})
\begin{equation}\label{eq:Lef}
i(f^n, X) = \sum_{q \ge 0} (-1)^q \trace(h^q(f^n, X))
\end{equation}
that relates the fixed point index of $f^n$ in an isolated invariant continuum $X$ with the traces of the cohomological Conley indices. If the phase space is a manifold of dimension $d$ the indices with $q \ge d + 1$ vanish because the groups $\check{H}^q(N, L)$ become trivial for any index pair $(N, L)$. Moreover, see Section \ref{sec:background}, in differentiable manifolds it is always possible to choose index pairs such that $N, L$ and $N \setminus L$ are manifolds with boundary so $\check{H}^q(N, L) = H^q(N, L)$. This election is implicit in the results of this chapter and \v{C}ech cohomology is replaced with singular cohomology without lose of generality.

The term $q = 1$ in the sum of Equation \eqref{eq:Lef}, $\trace(h^1(f^n, X))$, is described in Corollary \ref{coro:genA_intro}: it decomposes as the sum of $\#\Fix(\varphi^n) - 1$ ($\varphi$ being the permutation of the essential quasicomponents induced by $\tilde{f}$) and $\trace((\tilde{f}^*)^n|_{\im \ i^*})$. Recall that $i^*$ is the map induced by the inclusion $i \colon \tilde{X} \subseteq W_{\infty}^u$ in one--dimensional cohomology. By exploiting the relation between $X$ and $\tilde{X}$ and $f^*$ and $\tilde{f}^*$ (see Section \ref{sec:background}), a way to control the contribution of $\trace((\tilde{f}^*)^n|_{\im \ i^*})$ is to impose constraints on $\check{H}^1(X)$ or on the map induced by $f$ on this cohomology group.

Suppose for the time being that $f$ is a continuous map defined on an open subset $U$ of $\mathbb{R}^2$.  Next corollary extends, when \(X\) is a singleton, the main result of \cite{HCRP} which allowed to give a partial answer of Shub conjecture for maps in $\mathbb{S}^2$ (see \cite{Shub} or \cite{ShubSullivan} and \cite{GMN} or \cite{IPRX} for some very recent advances).

\begin{corollary}\label{coro:R2indicesperiodicos}
Let  \(U \subset S^2\) be an open set, \(f \colon U \rightarrow f(U) \subset S^2\) be a continuous map and \(X \subseteq U\) be an isolated invariant continuum that is not a repeller and is completely invariant, i.e. $f^{-1}(X) = X$.
Then, $i(f^n, X) \le 1$ for infinitely many even $n \in \mathbb N$ and the inequality is strict provided $X$ is not an attractor.
Furthermore, if $f$ is a homeomorphism the sequence \(\{i(f^n, X)\}_{n \ge 1}\) is periodic.


\end{corollary}
\begin{proof}
Assume that \(f\) is just a continuous map.
Take an index pair $(N, L)$ for $X$ such that $N$ is proper and connected. Let us first prove that if $X$ is not a repeller $\trace(h^2(f^n, X))$ vanishes for every $n \ge 1$ or, equivalently, $f^* \colon H^2(N, L) \to H^2(N, L)$ is nilpotent.
Since $N, L$ are proper subsets of $\mathbb{S}^2$, any 2--cocycle in $H^2(N, L)$ is determined by a connected component of $\mathbb{S}^2 \setminus L$ that is completely contained in $N$. Denote these components by $\{B_1, \ldots, B_m\}$. For the map $f^* \colon H^2(N, L) \to H^2(N, L)$ to be non--nilpotent we need an infinite chain $B_{i_1} \mapsto B_{i_2} \mapsto B_{i_3} \mapsto \ldots$ such that the image by $f$ of each component covers the next one. Since the set of components is finite, $f^k(B_i) \supset B_i$ for some $i$ and $k \in \mathbb{N}$.
By continuity we have that $f^k(\overline{B_i}) \supset \overline{B_i}$ and, as a consequence, $X' = \Inv(f^k, \overline{B_i})$ (the maximal $f^k$--invariant subset of $\overline{B_i}$) is non--empty. From the properties of index pairs we obtain that $X' \subset B_i$ and the connectedness of $X$ implies that $k = 1$ and $X = X'$. It follows that $X$ is a repeller.
Notice further that if $f$ were a homeomorphism then $(B_i \cup L, L)$ would be an index pair, the index map induced in $(B_i \cup L)/L$ would be homotopic to a map of degree $1$ or $-1$ in the 2--sphere and the traces of $h^2(f^n, X)$ would be $1$ or $(-1)^n$, respectively.

Now, Equation (\ref{eq:Lef}) and Corollary \ref{coro:genA_intro} yield $i(f^n, X) = -\trace(h^1(f^n, X)) = 1 - \#\Fix(\varphi^n) - \trace((f^*)^n|_{\im \, i^*})$.
If $n$ is multiple of the l.c.m. $r$ of all the periods of the essential quasicomponents (there are a finite number of them by Theorem \ref{teo:main1_intro}) then $\varphi^n$ is the identity map. By Proposition \ref{prop:noessential}, if there is no essential quasicomponent then $X$ is an attractor. On the other hand, the term $\trace((f^*)^n|_{\im \, i^*})$ can be bounded using the following result (see \cite[Lemma 2.3]{Franks2}): given any real matrix $A$, the trace of $A^l$ is non--negative for infinitely many values of $l \in \mathbb N$. In particular, $\trace((f^*)^{2rl}|_{\im \, i^*}) \ge 0$ for infinitely many $l$. The inequalities in the statement now follow from these considerations.


Alexander's duality provides an isomorphism between $\check{H}^1(X)$ and $\widetilde{H}_0(S^2 \setminus X)$. When \(f\) is a homeomorphism, the action of $f^*$ in the former can be understood as the dual (up to sign) of the map induced by $f$ in $\widetilde{H}_0(S^2 \setminus X)$ and the latter has a very special form because $f$ permutes the connected components of $S^2 \setminus X$. Although this set may be infinite, it is straightforward to prove that the non--zero eigenvalues of the restriction of the map to a finite dimensional invariant subspace of $\widetilde{H}_0(S^2 \setminus X)$ are roots of unity. As a consequence the term $\trace((f^*)^n|_{\im \, i^*})$ is periodic in $n$ and this trivially implies that $\{i(f^n, X)\}_n$ is a periodic sequence.
\end{proof}

The particular case of acyclic $X$ of the previous result was the main theorem of \cite{HCRP}. In the case $f$ is an orientation--preserving homeomorphism and $X$ is a fixed point (neither attracting nor repelling), the exact form of the sequence of indices $i(f^n, p)$ was established by Le Calvez and Yoccoz in \cite{LCY}:
$$
i(f^{n}, p) = \begin{cases}
1-rq & \text{ if } n \in r{\mathbb N} \\
1 & \text{ if } n \notin r{\mathbb N}
\end{cases}
$$
The integer $r \ge 1$ can be easily interpreted in our terms as follows:
all the essential quasicomponents (recall there are finitely many of them) have the same period, and this common period is $r$.
For the sake of completeness, let us give a hint on the argument. Since $f$ is injective, the essential quasicomponents $\{F_1, \ldots, F_k\}$ are legitimate subsets of the plane. Let $N$ be an isolating neighborhood of $p$ and $\cap_{n \ge 0} f^{n}(N)$ the set of points of $N$ whose backward trajectory does not exit $N$. The intersection $F_i \cap N$ splits in several components, take $F'_i$ the union of the components that are adherent to $p$ (with the intrinsic topology, see Section 3). It follows that $f$ permutes the sets $F'_i$ exactly in the same way as it permuted the quasicomponents $F_i$. However, $\{F'_1, \ldots, F'_k\}$ have a circular order that is preserved by the action of $f$: there is an integer $s \ge 1$ such that $f(F'_i) \supset F'_{i+s}$ for every $i$. We conclude that all the essential quasicomponents have the same period, $r$, and, consequently, they come in a number $rq$ which is multiple of $r$.

The knowledge of the structure of the closure of periodic orbits of homeomorphisms of the 2--sphere is an interesting problem that was stated explicitly by Le Calvez in his lecture in the ICM 2006 \cite{LCICM2006}. The question of whether \(\overline{\Per(f)}\) is isolated as invariant set has been studied in \cite{RPSal} and \cite{Sal}. The following result can be deduced from Corollary \ref{coro:R2indicesperiodicos}:

\begin{corollary}
Let \(f\colon \mathbb{S}^2 \rightarrow \mathbb{S}^2\) be a continuous map with \(\mathrm{deg}(f) \neq 0\). Assume that \(K\) is an isolated invariant continuum such that \(\Per(f) \subseteq K\) and $f^{-1}(K) = K$. Then \(K \) is a repeller.
\end{corollary}

\begin{proof}
Suppose that $K$ as in the statement exists and it is not a repeller.
By Corollary \ref{coro:R2indicesperiodicos}, \(i(f^n, K) \leq 1 \) for infinitely many even \(n \in {\mathbb N}\). 
Applying Lefschetz--Hopf theorem to $f$ we have \(\Lambda(f^n)=1+\mathrm{deg}(f)^n= i(f^n, K) \) and we deduce that $\mathrm{deg}(f) = 0$.
\end{proof}

An extra tool to analyze higher order terms is the duality between cohomological Conley indices due to Szymczak \cite{szy} in the discrete setting (see also \cite{HCR} for a shorter proof). It states that, given an isolated invariant set $X$ of a homeomorphism $f$ defined on an open set of ${\mathbb R}^d$ (automatically $X$ is also isolated for $f^{-1}$), for any $0 \le q \le d$ the $(d-q)$--cohomological index of $X$ and $f$ is dual, up to sign, to the $q$--cohomological index of $X$ and $f^{-1}$. The sign $+1$ or $-1$ depends on whether $f$ preserves or reverses orientation, respectively. We insist that this only works when the dynamics is invertible. Thus from now on we assume that $f$ is a homeomorphism defined on an open subset of an oriented 3--manifold.

In dimension $3$, Szymczak's duality together with Corollary \ref{coro:genA_intro} provides a description of $\trace(h^2(f^n, X))$ in the following terms: \[\pm \trace(h^2(f^n,X)) = -1 + \#\Fix(\psi^n) + \trace((f^*)^{-n}|_{\im\ j^*})\] where $\psi$ is the permutation induced by $f^{-1}$ on the set of essential quasicomponents of the compactified \emph{stable} manifold $W^s_{\infty}$ (endowed with its intrinsic topology) and $j \colon X \subseteq W^s_{\infty}$ denotes the inclusion. As mentioned earlier in the Introduction, since $f$ is assumed to be a homeomorphism there is no need to distinguish between $\tilde{X}$ and $X$ or $\tilde{f}$ and $f$.



\begin{corollary} \label{cor:periodic} Let \(U \subset {\mathbb R}^3\) be an open set and
\(f\colon U \rightarrow f(U) \subset {\mathbb R}^3\) be a homeomorphism.
Let $X \subseteq \mathbb{R}^3$ be an arbitrary isolated invariant continuum. The sequence \(\{i(f^n, X)\}_n\) is periodic provided the eigenvalues of  \(f^* \colon  \check{H}^1(X;\mathbb{K}) \to \check{H}^1(X;\mathbb{K})\) are roots of unity.

This condition is satisfied, for instance, when $\check{H}^1(X;\mathbb{K}) = 0$ or when the topological entropy of \(f|_{X}\) is zero. In particular, if \(f\) is a volume contracting homeomorphism and the sequence \(\{i(f^n, X)\}_n\) is unbounded then \(\check{H}^1(X) \ne 0\) and the topological entropy of \(f|_{X}\) is nontrivial.
\end{corollary}

\begin{proof} The Lefschetz formula \eqref{eq:Lef} and Szymczak's duality entail, through the discussion of the previous paragraph, that $i(f^n, X)$ is a sum of terms of the form $-1+\#\Fix(\varphi^n)$ and $-1+\#\Fix(\psi^n)$, which are certainly periodic, and $\trace((f^*)^n|_{\im\ i^*})$ and $\trace((f^*)^{-n}|_{\im\ j^*})$. Thus the sequence of indices is periodic if and only if these latter terms are periodic in $n$, and this is granted by the assumption that the eigenvalues of $f^*$ on $\check{H}^1(X)$ are roots of unity. Indeed, every eigenvalue $\mu_k$ of $f^*|_{\im\ i^*}$ is an eigenvalue of $f^*$ on $\check{H}^1(X)$, so each $\mu_k$ is a root of unity. Since the trace of an endomorphism is the sum of its eigenvalues (counted with their multiplicity), it follows that $\trace((f^*)^n|_{\im\ i^*}) = \sum \mu_k^n$, which clearly implies that the sequence of traces is periodic. A similar argument applies for $\trace((f^*)^{-n}|_{\im\ j^*})$.

The key element to address the second statement is a result of Manning \cite{Mann} that bounds the topological entropy $h(f|_{\#})$ from below in terms of the logarithm of the spectral radius of $(f_{\#})_*\colon H_1(N/L; \mathbb{Q}) \to H_1(N/L; \mathbb{Q})$. Here $f_{\#}\colon N/L \to N/L$ denotes the index map (see Section \ref{sec:background}) associated to the index pair $(N, L)$ and we assume $N/L$ is a compact absolute neighborhood retract (this is true if, for example, $N$ and $L$ are manifolds with boundary). Since the non--wandering set of $f_{\#}$ is equal to $X \cup \{[L]\}$, from the properties of the topological entropy we deduce that $h(f_{\#}) = h((f_{\#})|_X) = h(f|_X)$. Thus, if $h(f|_X)$ vanishes, by duality, the spectral radius of $(f|_{\#})^*\colon H^1(N/L; \mathbb{Q}) \to H^1(N/L; \mathbb{Q})$ is bounded by 1. By the universal coefficient theorem the previous map is given by an integer matrix $A$. It is a standard arithmetic consequence from the fact that the trace of $A^n$ is always an integer and the modulus of the eigenvalues is bounded by 1 that all the non--zero eigenvalues are roots of unity (cf. \cite[Proposition 2]{LCRPS}).
\end{proof}

Let us apply this result to some particular cases of interest:

\begin{theorem} \label{teo:S1} Let \(U \subset {\mathbb R}^3\) be an open set and
\(f\colon U \rightarrow f(U) \subset {\mathbb R}^3\) be a homeomorphism.
Assume $X$ is an isolated invariant continuum such that \(\check{H}^*(X;\mathbb{Z})=\check{H}^*(\mathbb{S}^1;\mathbb{Z})\). Then the index sequence $\{i(f^n, X)\}_n$ is periodic.
\end{theorem}
\begin{proof} Since $f$ is a homeomorphism, the induced homomorphism $f^*$ on $\check{H}^1(X;\mathbb{Z}) = \mathbb{Z}$ must be invertible; thus, $f^* = \pm {\rm Id}$. By the universal coefficient theorem it follows that $f^* = \pm {\rm Id}$ also on $\check{H}^1(X;\mathbb{Q})$. Thus the hypotheses of Corollary \ref{cor:periodic} are trivially satisfied and the result follows.
\end{proof}

Our techniques can be applied to compute or estimate the sequence of fixed point indices of a given homeomorphism at arbitrary isolated compact polyhedra \(X\) in dimension 3. In the particular case where \(X\) is a cohomology 2--sphere we have:

\begin{theorem} \label{teo:S2} Let \(U \subset {\mathbb R}^3\) be an open set, \(f\colon U \rightarrow f(U) \subset {\mathbb R}^3\) be a homeomorphism and $X \subset U$ be an isolated invariant continuum. Assume that \(\check{H}^*(X;\mathbb{K})=\check{H}^*(\mathbb{S}^2;\mathbb{K})\). Then the index sequence \(\{i(f^n, X)\}_{n \ge 1}\) is periodic. Furthermore, if $f$ reverses orientation, $i(f^n, X) \le 2$ and if, in addition, $f$ fixes the components of the complement of $X$ the inequality is strict for odd $n$. Thus in particular, if \(f\colon  {\mathbb R}^3 \rightarrow  {\mathbb R}^3\) is an orientation--reversing homeomorphism, \(i(f,X) \leq 1\).
\end{theorem}
\begin{proof} The periodicity of the index sequence follows directly from Corollary \ref{cor:periodic}. If $f$ is orientation--reversing, duality implies $i(f^n, X) = -\trace(h^1(f^n, X)) - \trace(h^1(f^{-n}, X))$ and, since the first cohomology group of $X$ vanishes, this reduces to $1-\#\Fix(\varphi^n)+1-\#\Fix(\psi^n)$. Here $\varphi$ and $\psi$ stand, as usual, for the permutations induced by $f$ and $f^{-1}$ among the (finite) sets of essential quasicomponents of the unstable and stable manifolds (minus $X$ and their points at infinity). This yields the inequality $i(f^n, X) \le 2$.

Notice that the statement on the connected components of the complement of $X$ is based on purely cohomological information: Alexander duality guarantees that $f^*\colon \check{H}^2(X;\mathbb{K}) \to \check{H}^2(X;\mathbb{K})$ is conjugate to the map induced by $f$ in $\widetilde{H}_0(S^3 \setminus X; \mathbb{K})$ (to be more precise we have to multiply this map by $-1$ because $f$ reverses orientation). We say that $f|_X$ preserves/reverses orientation if $f^* = \id$/$-\id$.
Suppose now that the equality $i(f^n,X) = 2$ is attained for some odd $n$. Then $\Fix(\varphi) = \emptyset$, so the essential quasicomponents of the unstable manifold are permuted without fixed elements and Corollary \ref{coro:Lef} applies to conclude that $i(f^n, X) = \Lambda(f^n|_X)$. However, since $f$ reverses orientation but fixes the components of the complement of $X$ we have that $f|_X$ also reverses orientation and $\Lambda(f^n|_X) = 1 + (-1)^n$. The sum vanishes because $n$ is odd, contradicting the assumption $i(f^n, X) = 2$.
\end{proof}

\begin{corollary}
Let $f$ be an orientation--reversing homeomorphism of a 3--manifold $B$ whose boundary $S$ is homeomorphic to $\mathbb{S}^2$. Assume further that the integral cohomology of $B$ is trivial, i.e., equal to the cohomology of a point. If $S$ is isolated as invariant set then $\Fix(f) \cap \int(B) \neq \emptyset$. In other words, if $\Fix(f) \subseteq S$, for every compactum $I \subseteq \int(B)$ there exists $x \in \int(B)$ whose full orbit is disjoint from $I$.
\end{corollary}

We remark that the cohomological assumption on $B$ does not force it to be homeomorphic to a 3--ball: for instance, $B$ could be the complement of an open standard ball inside the Poincar\'e sphere.

\begin{proof} Denote by $M$ the double of $B$ along $S$; that is, the manifold obtained from two disjoint copies of $B$ by identifying their boundaries via the identity. $M$ can be written as $M = B \cup B'$, with $B \cap B' = S$, and a straightforward computation using the Mayer--Vietoris sequence of this decomposition then shows that $M$ has the integral \v{C}ech cohomology of $\mathbb{S}^3$. (In principle, the Mayer--Vietoris sequence requires the decomposition to consist of open sets; however, in the present case of \v{C}ech cohomology and compact sets this is not necessary as discussed in p. \pageref{pag:mvieto}).

The map $f$ induces an orientation--reversing homeomorphism, denoted $Df$, on $M$ by acting independently on each copy of $B$. Applying Lefschetz--Hopf theorem on $M$ we obtain $2 = \Lambda(Df) = i(Df, M)$. If we suppose that $S$ is isolated as invariant set in $B$ then it is also isolated in $M$ and by Theorem \ref{teo:S2} we conclude that $i(Df, S) \le 1$. Consequently, $i(Df, M) \neq i(Df, S)$ so there are fixed points in $M \setminus S$ and, by symmetry, also in $\int(B)$.
\end{proof}

The following corollary can be interpreted in terms of minimal dynamics (those having all their orbits dense): there are no orientation--reversing homeomorphism of $S^3$ which are minimal away from a finite number of invariant circles in which the map is homotopic to the identity. Notice that those invariant circles can be made of fixed points and homo/heteroclinic trajectories.

\begin{corollary}
Let $f$ be an orientation--reversing homeomorphism of $\mathbb{S}^3$ that leaves invariant a knot or a link $K$. Assume further that $f|_K$ preserves orientation and $K$ contains all the fixed points of $f$. Then $K$ is not an isolated invariant set, i.e. for every neighborhood $U$ of $K$ there exists a point $x \in U \setminus K$ such that the orbit of $x$ is contained in $U$.
\end{corollary}
\begin{proof}
Suppose that $K$ is an isolated invariant set. With the notations introduced earlier,
$$i(f, K) = 2 - \#\Fix(\varphi) - \#\Fix(\psi) - \trace(f^*|_{\im\ i^*}) - \trace((f^{-1})^*|_{\im\ j^*}).$$ The last two summands are non--positive because the restrictions of $f$ and $f^{-1}$ to $K$ preserve orientation. The first two terms are non--positive as well and if $\#\Fix(\varphi) = 0$ it follows from Corollary \ref{coro:Lef} that $i(f, K) = \Lambda(f|_K) = 0$. Thus, regardless of whether $\Fix(\varphi)$ is empty or not, we may write $i(f, K) \le 1$. However, Lefschetz--Hopf theorem applied to $f$ yields $i(f, \mathbb{S}^3) = \Lambda(f) = 2$ and since $i(f, \mathbb{S}^3) \neq i(f, K)$, this means that there must exist fixed points outside $K$.
\end{proof}

Given a compact metric space \(M\), a homeomorphism \(f\colon M \to M\) is said to be
{\em expansive} if there exist \(c>0\) such that if \(x, y \in M\) and \(dist(f^m(x), f^m(y)) \leq c\) for every \(m \in {\mathbb Z}\) then \(x=y\). There are many papers in the literature about expansiveness of homeomorphisms \cite{lewowicz}, \cite{magne}  and references with very hard problems \cite{hiraide}.

It is easy see that this is equivalent to the diagonal
\(\Delta \subseteq M \times M\) to be an isolated invariant set for \(f \times f\colon M \times M \to M \times M\). For arbitrary continuous maps  it has been introduced the natural notion of positive expansiveness but using the above characterization one can generalize the definition of  expansiveness for continuous maps. A continuous map  \(f\colon M \to M\)  is expansive if
\(\Delta \subseteq M \times M\) is an isolated invariant set for \(f \times f\colon  M \times M \to M \times M\) so Theorems 1, 2  and 6 and Corollaries 3, 4 and 5  are general results that also provide tools to study expansiveness of maps defined in arbitrary compact manifolds. In particular there are only finitely many essential quasicomponents in $W^u_{\infty} \setminus (\tilde{\Delta} \cup \infty)$.  In a forthcoming paper we shall study some applications of our results to this theory.

\section{Some background definitions and results \label{sec:background}}

In this section we recall the basic definitions of the discrete Conley index and the approach of Robbin and Salamon in terms of the intrinsic topology in the unstable manifold, generalizing it to the setting of continuous maps rather than homeomorphisms. Throughout this section coefficients for homology or cohomology are assumed to be taken in a field $\mathbb{K}$ that is not always reflected in the notation.

\subsection{The cohomological Conley index} Consider a continuous map $f$ defined on a locally compact space.
The object of study are invariant sets, that is, sets $X$ that satisfy $f(X) = X$.
 A full orbit of $f$ through a point $p$ is a sequence $(p_i)_{i \in \mathbb{Z}}$ such that $f(p_{i}) = p_{i+1}$ and $p_0 = p$. Given a set $A$, $\mathrm{Inv}(A)$ is defined as the set of points for which there is a full orbit of $f$ passing through them completely contained in $A$. Evidently, $f(\mathrm{Inv}(A)) = \mathrm{Inv}(A)$, i.e., $\Inv(A)$ is an invariant set.
A compact invariant set $X$ is said to be \emph{isolated} or \emph{locally maximal} provided there is a neighborhood $A$ of $X$ such that $X = \mathrm{Inv}(A)$. The compact neighborhoods $A$ of $X$ satisfying $X = \Inv(A)$ are called \emph{isolating neighborhoods} of $X$.

The Conley index associates a topological invariant to a compact isolated invariant set $X$ through a compact pair whose purpose is to encapsulate the local dynamics around $X$. An \emph{index pair} for $f$ and $X$ is a compact pair $(N, L)$, $L \subseteq N$ that satisfies:
\begin{itemize}
\item[(i)] $\overline{N \setminus L}$ is an isolating neighborhood of $X$.
\item[(ii)] $f(N \setminus L) \subseteq N$, that is, the points that are mapped outside $N$ are contained in $L$.
\item[(iii)] $f(L) \cap N \subseteq L$, that is, $L$ is positively invariant in $N$.
\end{itemize}

A classical result in the theory of Conley index \cite{mischaikowmrozek} states that, provided the phase space is a locally compact absolute neighborhood retract, any neighbourhood of an isolated invariant set contains an index pair such that $\check{H}^*(N,L;\mathbb{K})$ is finite dimensional. (Recall that the class of absolute neighbourhood retracts is very wide and includes topological manifolds and simplicial complexes). Furthermore, if the phase space is a differentiable or PL manifold then $N$ and $L$ can be chosen to be manifolds with boundary \cite{franksricheson}. The idea behind this fact is the existence a robust subclass of index pairs that can be constructed as follows. Take an isolating neighborhood $N$ of $X$ that satisfies $f^{-1}(N) \cap N \cap f(N)$ (called \emph{isolating block}) and let $L$ be a small neighborhood in $N$ of $\{x \in N: f(x) \notin \int(N)\}$. Then, see \cite{franksricheson}, any compact pair $(N', L')$ close to $(N, L)$ is an index pair for $f$ and $X$. Choosing $N'$ and $L'$ to be manifolds the result follows.

Consider the pointed quotient space $(N/L,[L])$, where $[L]$ denotes the point consisting of $L$ alone. Define a map $f_{\#} \colon  (N/L,[L]) \to (N/L,[L])$ by \[f_{\#}(p) := \left\{ \begin{array}{cl} f(p) & \text{ if } p, f(p) \in N \setminus L; \\ \left[L\right] & \text{ otherwise.} \end{array} \right.\] This is usually called an \emph{index map} and it follows from the requirements on index pairs that it is continuous. It is this continuity, rather than the particular details of $(N,L)$, that the construction of Robbin and Salamon hinges on. (In fact, they define and index pair as a compact pair $(N,L)$ that satisfies (i) above and makes $f_{\#}$ continuous).

The index map induces a homomorphism $f^*_{\#} \colon  \check{H}^*(N/L,[L]) \to \check{H}^*(N/L,[L])$ in cohomology. We recall that coefficients are taken in some field $\mathbb{K}$ that we omit from the notation. Both $\check{H}^*(N/L,[L])$ and $f^*_{\#}$ depend on the particular index pair $(N,L)$ chosen to perform this construction, but the \emph{shift equivalence class} of $f^*_{\#}$ does not. The $q$--cohomological Conley index can then be legitimately defined as the shift equivalence class of $f^*_{\#} \colon  \check{H}^q(N/L,[L]) \to \check{H}^q(N/L,[L])$, and we will denote it by $h^q(f,X)$.

Shift equivalence is an equivalence relation defined between maps in an abstract category; in this case, between endomorphisms of a vector space and, in fact, of a finite dimensional vector space because the pair $(N,L)$ can be chosen to have finitely generated \v{C}ech cohomology as mentioned above. We will not need to recall the definition of shift equivalence here (the interested reader can find it in \cite{franksricheson}), but just mention the following: in the category of finite dimensional vector spaces, two endomorphisms that are shift equivalent have the same nonzero eigenvalues, counted with their multiplicity, and, in particular, the same trace. Therefore, it makes sense to speak of the nonzero eigenvalues of $h^q(f,X)$, or of the trace of $h^q(f, X)$. We denote the latter by ${\rm trace}(h^q(f,X))$. 




\subsection{The intrinsic topology} Now we recall the definition of the intrinsic topology given by Robbin and Salamon \cite{robinsalamon1}. Since their paper is set up for diffeomorphisms, we go into some detail to highlight the differences that appear when carrying over their construction to the case of merely \emph{continuous} maps.

Let $X$ be an isolated invariant set for a continuous map $f$ and choose an index pair $(N,L)$ for $X$. Consider the inverse sequence of pointed topological spaces \begin{equation} \label{sec:NL} \xymatrix{(N/L,[L]) & (N/L,[L]) \ar[l]_-{f_{\#}} & (N/L,[L]) \ar[l]_-{f_{\#}} & \ldots \ar[l]}\end{equation} and denote provisionally its inverse limit by $(W_{\infty},\infty)$. Being an inverse limit of a sequence of compact spaces, $W_{\infty}$ is a compact (Hausdorff) space. Any point $\xi \in W_{\infty}$ is a sequence $\xi = (\xi_i)_{i \geq 0}$ where $\xi_i \in N/L$ and $f_{\#}(\xi_{i+1}) = \xi_i$. The point $\infty$ corresponds to the constant sequence $\xi = ([L], \ldots, [L], \ldots)$. Any other $\xi \in W_{\infty}$ must have, according to the definition of $f_{\#}$, one of the forms \begin{subequations} \begin{align} \xi & = (p_0,p_1,\ldots) \label{eq:xinoL} \\ \xi & = (\underbrace{[L],[L],\ldots,[L]}_{(n)},p_n,p_{n+1},\ldots) \label{eq:xiL} \end{align} \end{subequations} where the $p_i$ denote points in $N \setminus L$, $f(p_{i+1}) = p_i$ and, in the case of \eqref{eq:xiL}, $f(p_n) \in L$.

The map $f_{\#}$ induces a continuous map $\tilde{f}$ on $W_{\infty}$ in the standard way $\tilde{f}(\xi) := (f_{\#}(\xi_i))_{i \geq 0}$. More explicitly, for $\xi = (\xi_i)_{i \geq 0}$ one has $\tilde{f}(\xi) = (f_{\#}(\xi_0),\xi_0,\xi_1,\ldots)$ which makes it clear that $\tilde{f}$ is a homeomorphism, its inverse being a left shift. We emphasize that this is the case regardless of whether the original map $f$ is a homeomorphism or more generally a continuous map, which is the situation that we consider in this paper. The following result, which was anticipated earlier and justifies our interest in $W_{\infty}$, will be proved later on:

\begin{proposition} \label{prop:conley} Let the phase space be a topological manifold or, more generally, a locally compact absolute neighbourhood retract. Then $(W_{\infty},\infty)$ has finitely generated \v{C}ech cohomology and, moreover, the induced homomorphism $\tilde{f}^* \colon  \check{H}^q(W_{\infty},\infty) \to \check{H}^q(W_{\infty},\infty)$ belongs to the $q$--cohomological Conley index of $X$.
\end{proposition}

We shall show now that the dynamics of $\tilde{f}$ in $W_{\infty}$ is fairly simple. First we will see that $\infty$ is an asymptotically stable attractor. Recall that $\tilde{f}(\xi_0, \xi_1, \ldots) = (f_{\#}(\xi_0), \xi_0, \xi_1, \ldots)$. Clearly every point $\xi$ of the form displayed in Equation \eqref{eq:xiL} above is attracted to $\infty$, since each additional iterate adds an $[L]$ to the initial part of the sequence $(\xi_i)$. Similarly, points $\xi$ of the form given in \eqref{eq:xinoL} will be attracted to $\infty$ as long as each $p_i$ eventually reaches $L$ upon iteration. Otherwise stated, a point $\xi$ does \emph{not} belong to the basin of attraction of $\infty$ if, and only if, it is of the form $\xi = (p_0,p_1,\ldots)$ where all the forward iterates of each $p_i$ stay in $N \setminus L$. Since through each $p_i$ there is a negative semiorbit contained in $N$ (namely, $p_{i+1}, p_{i+2}, \ldots$), this latter condition is equivalent to saying that each $p_i$ belongs to ${\rm Inv}(N) = X$. Thus, letting $\tilde{X}$ be \[\tilde{X} := \{(\xi_i)_{i \geq 0} \in W_{\infty} : \xi_i \in X \text{ for all $i$}\},\] it follows from the above that the basin of attraction of $\infty$ is precisely $W_{\infty} \setminus \tilde{X}$. The set $\tilde{X}$ is therefore a repeller in $W_{\infty}$ (the so-called dual repeller of $\infty$ in the parlance of Morse decompositions) and its basin of repulsion is $W_{\infty} \setminus \infty$.

Summing up, if $X$ is an isolated invariant set for a continuous map $f$ and $W_{\infty}$ is defined as above, the following hold true:
\begin{itemize}
	\item[(i)] $f$ induces in a natural way a homeomorphism $\tilde{f}$ on $W_{\infty}$.
	\item[(ii)] The pair $(\infty,\tilde{X})$ is an attractor-repeller decomposition of $W_{\infty}$.
\end{itemize}

Notice that $\tilde{X}$ can be identified with $X$ in a straightforward manner when $f|_X$ is a homeomorphism, but otherwise there is no inmediately apparent relation between them. Later on we will establish the following result:

\begin{proposition} \label{prop:Xtilde} If $\check{H}^q(X)$ is finite dimensional, then so is $\check{H}^q(\tilde{X})$ and, moreover, the induced homomorphisms $f^* \colon  \check{H}^q(X) \to \check{H}^q(X)$ and $\tilde{f}^* \colon  \check{H}^q(\tilde{X}) \to \check{H}^q(\tilde{X})$ are shift equivalent. In particular, they have the same nonzero eigenvalues and the same trace.
\end{proposition}


When $f$ is injective there is a natural identification of $W_{\infty} \setminus \infty$, as a point set, with the unstable manifold $W^u$ of $X$ (this is essentially \cite[Theorem 7.3, p. 390]{robinsalamon1}). The topology that $W^u$ receives under this identification is called the \emph{intrinsic topology} and, since $W_{\infty}$ is tautologically the one-point compactification of $W_{\infty} \setminus \infty$, it can also be regarded as the one--point compactification of the unstable manifold $W^u$ endowed with the intrinsic topology. However, since in this paper $f$ is merely continuous but not necessarily injective, this identification is not valid any more. This is the only aspect of the construction of Robbin and Salamon that does not carry over to our situation.

Let us see how the above affects the basic strategy sketched in the Introduction. Replacing $W^u_{\infty}$ with $W_{\infty}$, $f$ with $\tilde{f}$, and $X$ with $\tilde{X}$ and making use of the attractor-repeller decomposition of $W_{\infty}$ given by (ii), the Mayer--Vietoris exact sequence \eqref{eq:mvietoris_mot} now reads \[ \xymatrix{\ldots & \ar[l] \check{H}^{q+1}(\tilde{X}) & \ar[l] \check{H}^{q+1}(W_{\infty},\infty) & \ar[l]_-{\Delta} \check{H}^q(U\cap V) & \ar[l] \check{H}^q(\tilde{X}) & \ar[l] \ldots} \] so it relates the homomorphism induced by $\tilde{f}$ on $\check{H}^{q+1}(W_{\infty},\infty)$, which according to Proposition \ref{prop:conley} belongs to the $(q+1)$-cohomological Conley index of $X$, to the homomorphisms that $\tilde{f}$ induces on $\check{H}^q(U \cap V)$ and $\check{H}^{q+1}(\tilde{X})$. This involves the set $\tilde{X}$ rather than $X$, the original dynamical object of interest in our problem. In most applications we shall only be interested in the eigenvalues or the trace of the Conley index, which in turn only involves the eigenvalues of $\tilde{f}$ on $\check{H}^{q+1}(\tilde{X})$ and, according to Proposition \ref{prop:Xtilde} above, this information \emph{can} be recovered from the action of $f^*$ on $\check{H}^*(X)$.

The only element foreign to the original problem that remains is $U \cap V = W_{\infty} \setminus (\tilde{X} \cup \infty)$, which in general differs from $W^u \setminus X$. However, as the reader will see this is not really an issue because in our arguments the explicit form of $U \cap V$ will play no role.

\subsection{An example}\label{subsec:horseshoe} Consider the realization of Smale's horsehoe in the plane suggested by Figure \ref{fig:horse}. $D$ is a region of the plane consisting of a square $N$ (in thicker outline) and two semicircles. The dynamics is generated by a homeomorphism $f \colon D \to f(D) \subseteq D$ that squeezes $N$ horizontally and bends it into the shape of a horseshoe, as in Figure \ref{fig:horse}.(b). The image of the square $N$ intersects $N$ itself in two vertical rectangles. The maximal invariant subset of $N$ is the horseshoe $\Lambda$, which in this realization is the Cartesian product of two Cantor sets. Together with an attracting fixed point $p$ outside $N$, these are the only two isolated invariant subsets of $D$ and therefore the set of points whose orbits travel from $\Lambda$ to $\{p\}$ is equal to $W^u(\Lambda) \setminus \Lambda$, where $W^u(\Lambda)$ denotes the unstable manifold of $\Lambda$. 


\begin{figure}[h!]
\null\hfill
\subfigure{
\begin{pspicture}(0,0)(4.6,9.2)
	\rput[bl](0,0){\scalebox{0.75}{\includegraphics{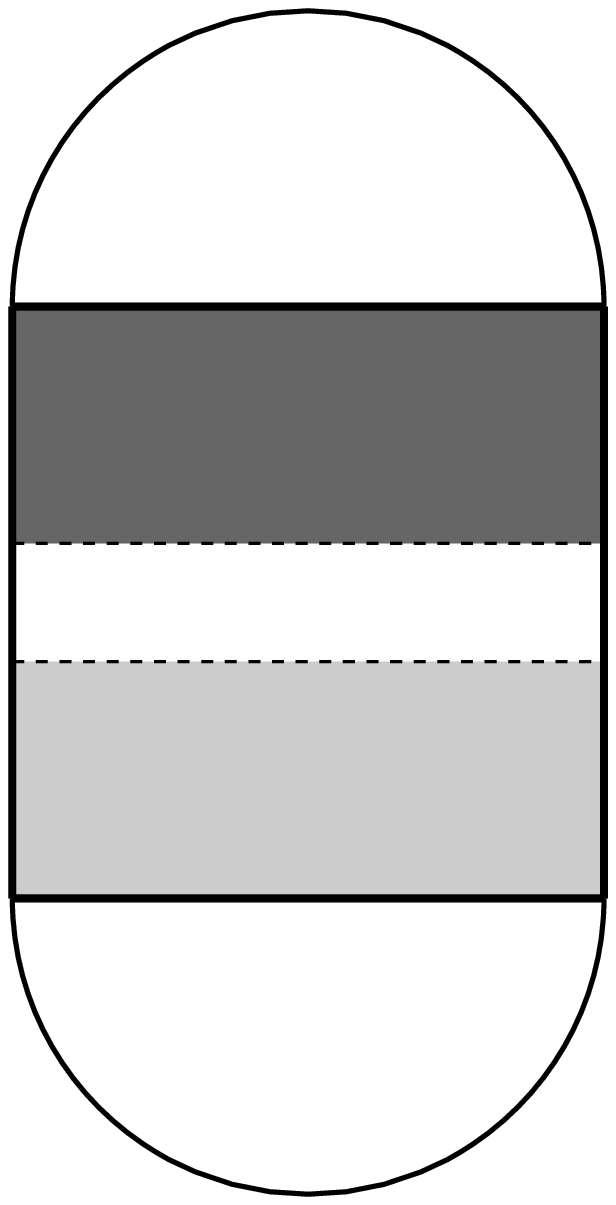}}}
	\rput[bl](3,1.8){$N$}
	\rput[bl](0.2,0.2){$D$}
\end{pspicture}
}
\hfill
\subfigure{
\begin{pspicture}(0,0)(4.6,9.2)
	\rput[bl](0,0){\scalebox{0.75}{\includegraphics{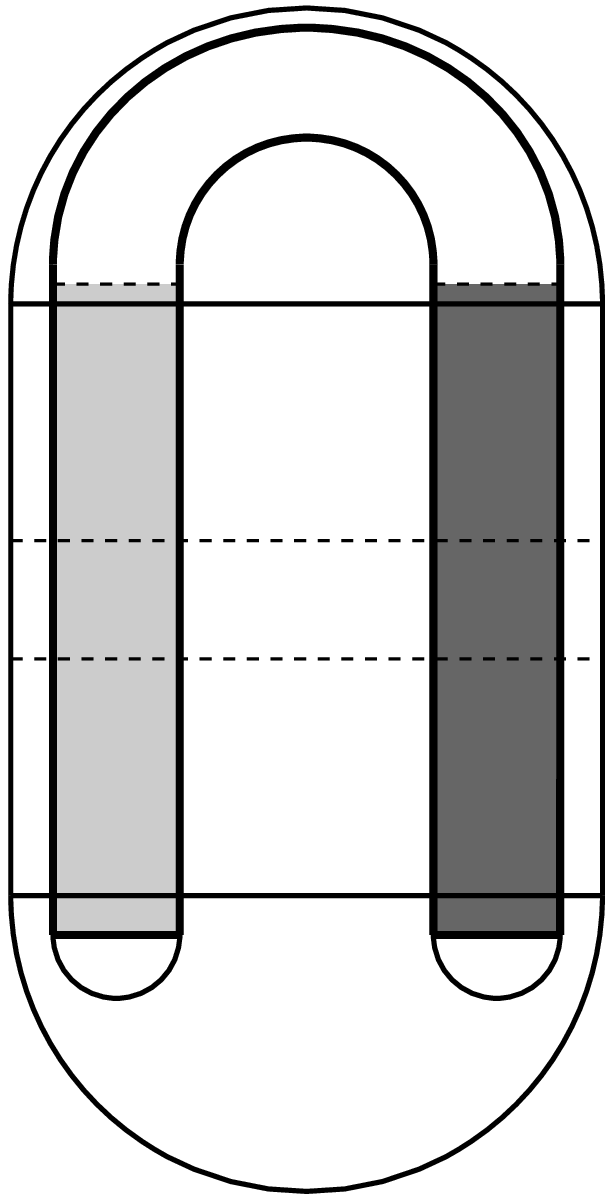}}}
	\rput[bl](1.2,1.2){$f(D)$}
\end{pspicture}
}
\hfill
\subfigure{
\begin{pspicture}(0,0)(4.6,9.2)
	\rput[bl](0,0){\scalebox{0.75}{\includegraphics{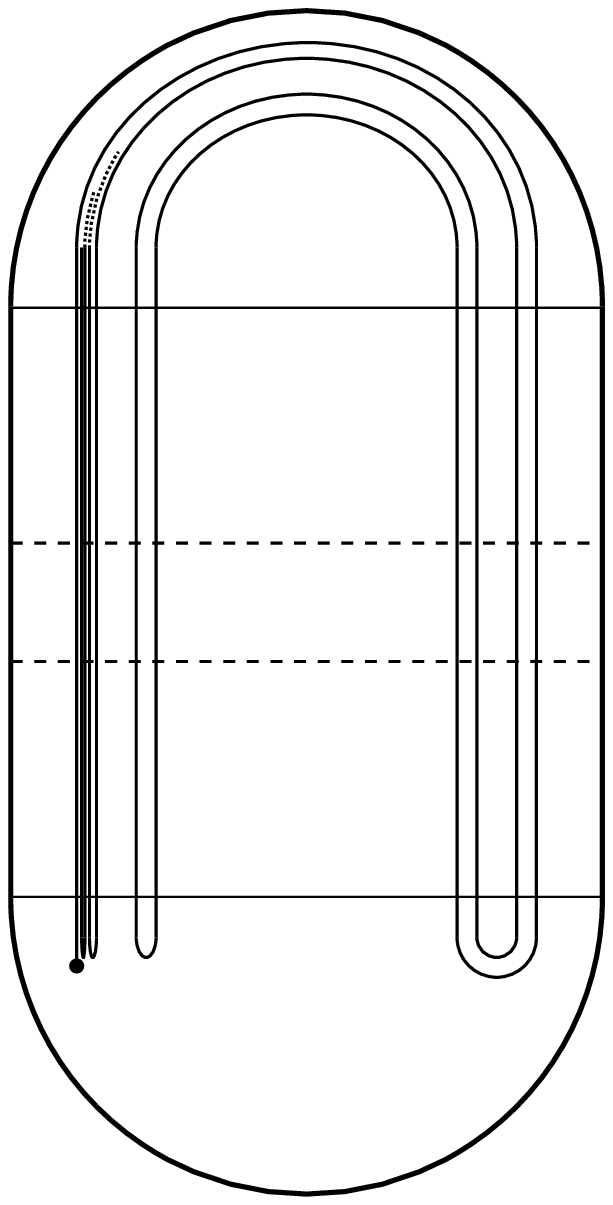}}}
	\rput[bl](1.2,1.6){$W^u(\Lambda)$}
	\rput[b](0.5,1.4){$p$}
\end{pspicture}
}
\hfill\null
\caption{The unstable manifold of a horseshoe \label{fig:horse}}
\end{figure}

The portion of $W^u(\Lambda)$ contained in $N$ is equal to the union of (a Cantor set of) vertical lines that pass through $\Lambda$. Each of these vertical lines is separated by $\Lambda$ into countably many open intervals. The intrinsic topology of the unstable manifold is finer than the ambient topology that it inherits from the plane. Indeed, for two points $x, y \in W^u(\Lambda) \cap N$ to be close in the intrinsic topology they must belong to the same vertical line. Otherwise there would exist some $n \in \mathbb{N}$ for which $f^{-n}(x)$ and $f^{-n}(y)$ lie in different halves of $N$ so $(x, \ldots, f^{-n}(x), \ldots)$ and $(y, \ldots, f^{-n}(y), \ldots)$ would not be close in the inverse limit. Notice, however, that the closure of $W^u(\Lambda)$ in the usual topology coincides with its one--point compactification $W^u(\Lambda) \cup \{p\} = \cap_{n \ge 0} f^n(N)$ and, as a point set, this is equivalent to the one--point compactification of $W^u(\Lambda)$ with the intrinsic topology, $W^u_{\infty}$. Thus, with our usual notation there is a correspondence $\infty = p$, $X = \Lambda$, $U = W^u(\Lambda)$ and $V = \{p\} \cup W^u(\Lambda) \setminus \Lambda$.


The set of connected components of $V$ is uncountable: in every vertical line of $W^u(\Lambda) \cap N$ there are already countably many components of $V$ and there are uncountably many such lines. Each connected component corresponds to an arc in $W^u(\Lambda)$ connecting two different points of $\Lambda$ except for an arc $F$ that reaches $p$. (Notice that in this case both the ambient and the intrinsic topology on $W^u_{\infty}(\Lambda)$ yields the same components and quasicomponents). Since every connected component is also a quasicomponent, the only essential quasicomponent in the example is $F$.

Evidently, Hastings' Theorem fails dramatically in this example, both $V$ and $U/X$ have very complicated topology: $V$ has uncountably many quasicomponents and the first integral cohomology group of $U/X$ is not finitely generated (to see this one can examine the long exact sequence of $(U, X)$). Notice, however, that this observation agrees with Lemma \ref{lem:uncountable}.


\subsection{Proofs of Proposition \ref{prop:conley} and \ref{prop:Xtilde}} As a preparation we need a rather general observation from elementary linear algebra.

Let $V$ be a $\mathbb{K}$--vector space and let $\varphi \colon  V \to V$ be an endomorphism of $V$. Consider the constant direct sequence \[\mathcal{S} \ : \ \xymatrix{V_{(0)} = V \ar[r]^-{\varphi} & V \ar[r]^-{\varphi} & V \ar[r]^-{\varphi} & \ldots}\] and denote by $V_{\infty}$ its direct limit and by $\pi \colon  V_{(0)} \to V_{\infty}$ the canonical projection from the first term of the sequence to its direct limit. Recall that $V_{\infty}$ is the direct sum of the (infinitely many) copies of $V$ that appear in $\mathcal{S}$ modulo the smallest equivalence relation $\sim$ that identifies each $v$ lying in any one copy of $V$ with its image $\varphi(v)$, which lies in the next.

Similarly, consider two copies of the sequence $\mathcal{S}$ connected by vertical arrows given by $\varphi$. This diagram is commutative and therefore the vertical arrows induce a homomorphism $\varphi_{\infty} \colon  V_{\infty} \to V_{\infty}$ between the direct limits. Then, with these notations, one has the following:

\begin{lemma} \label{lem:dirlim} Suppose $V$ is finite dimensional, so that the sequence of images $V \supseteq {\rm im}\ \varphi \supseteq {\rm im}\ \varphi^2 \supseteq \ldots$ stabilizes; that is, there exists $m$ such that ${\rm im}\ \varphi^m = {\rm im}\ \varphi^{m+1} = \ldots$ Denote by $W$ this subspace of $V$. Then:
\begin{itemize}
	\item[(1)] The restriction $\pi|_W \colon  W \to V_{\infty}$ is an isomorphism.
	\item[(2)] Moreover, $\pi|_W$ provides a conjugation between the restriction $\varphi|_W \colon  W \to W$ and $\varphi_{\infty} \colon  V_{\infty} \to V_{\infty}$.
	\item[(3)] The maps $\varphi$ and $\varphi_{\infty}$ are shift equivalent.
\end{itemize}

The isomorphism $\varphi|_W$ is usually called the \emph{Leray reduction} of $\varphi$.
\end{lemma}

In the proof of the lemma we will make use of two simple properties of shift equivalence:
\begin{itemize}
	\item[(i)] Suppose $\varphi$ is an endomorphism of a vector space $V$. Then $\varphi$ and the restriction $\varphi|_{{\rm im}\ \varphi} \colon  {\rm im}\ \varphi \to {\rm im}\ \varphi$ are shift equivalent.
	\item[(ii)] Two isomorphisms are conjugate if and only if they are shift equivalent.
\end{itemize}

\begin{proof} (1) The proof is easy and very similar to that of \cite[Proposition 6.(2), p. 2830]{pacoyo1}, so we only sketch it. A general element of $V_{\infty}$ is represented by a sum of elements belonging to different copies of $V$, but each of them can be pushed forward along the sequence until they all lie in the same copy $V_{(j)}$ of $V$, where they can be summed to yield an element $v$. Pushing the latter forward yet another $m$ times we have $v \sim \varphi^m(v)$, where now $\varphi^m(v)$ sits in $W$. Thus, every element of $V_{\infty}$ can be represented by an element of $W$ and so the direct limit of $\mathcal{S}$ remains unchanged if we replace the sequence with \[\mathcal{S}' \ : \ \xymatrix{W \ar[r]^-{\varphi|_W} & W \ar[r]^-{\varphi|_W} & W \ar[r]^-{\varphi|_W} & \ldots}\] It only remains to observe that $\varphi|_W$ is an isomorphism (it is surjective by the definition of $W$ and, since $W$ is finite dimensional because $V$ is finite dimensional by assumption, it must be an isomorphism) and so the direct limit of $\mathcal{S}'$ is canonically isomorphic to $W$.

(2) This is a straightforward computation.

(3) By repeated application of property (i) of shift equivalence $\varphi, \varphi|_{{\rm im}\ \varphi}, \varphi|_{{\rm im}\ \varphi^2}, \ldots, \varphi|_W$ are all shift equivalent. According to part (1) of this lemma the latter is conjugate to $\varphi_{\infty}$, and so both are also shift equivalent by (ii) above.
\end{proof}

As a complement we have the following result:

\begin{remark} \label{rem:betti} Let $m(0)$ denote the algebraic multiplicity (possibly zero) of $\lambda = 0$ as a root of the characteristic polynomial of $\varphi$. Then $\dim\ V_{\infty} = \dim\ V - m(0)$.
\end{remark}
\begin{proof} By Lemma \ref{lem:dirlim}.(1) we have $\dim\ V_{\infty} = \dim\ W$ where, we recall, $W = \im\ \varphi^m$; thus $\dim\ V_{\infty} = \dim\ V - \dim\ \ker\ \varphi^m$. It is easy to see from the definition of $m$ that $V = (\im\ \varphi^m) \oplus (\ker\ \varphi^m)$. Both summands are $\varphi$ invariant, so the eigenvalues of $\varphi$ are the union of those of $\varphi|_{\im\ \varphi^m}$ and $\varphi|_{\ker\ \varphi^m}$. The first of these is an isomorphism (again because of the definition of $m$) and has nonzero eigenvalues. The second is nilpotent so it has a single eigenvalue $\lambda = 0$ with algebraic multiplicity $\dim\ \ker\ \varphi^m$. Thus $m(0) = \dim\ \ker\ \varphi^m$, completing the proof.
\end{proof}

We are now ready to prove Propositions \ref{prop:conley} and \ref{prop:Xtilde}.

\begin{proof}[Proof of Proposition \ref{prop:conley}] The definition of $(W_{\infty},\infty)$ together with the continuity property of \v{C}ech cohomology imply that $\check{H}^*(W_{\infty},\infty)$ can be computed as the direct limit of \[\xymatrix{\check{H}^*(N/L,[L]) \ar[r]^-{(f_{\#})^*} & \check{H}^*(N/L,[L]) \ar[r]^-{(f_{\#})^*} & \check{H}^*(N/L,[L]) \ar[r]^-{(f_{\#})^*} & \ldots}\] Notice that this sequence has the form considered earlier with $V = \check{H}^*(N/L,[L])$, $\varphi = (f_{\#})^*$ and $V_{\infty} = \check{H}^*(W_{\infty},\infty)$. Moreover, it is straightforward to check that $\varphi_{\infty} = \tilde{f}^*$. As mentioned earlier, the assumption about the phase space guarantees that $(N,L)$ can be chosen to have finite dimensional cohomology. Thus $V$ is finite dimensional and so Lemma \ref{lem:dirlim} can be applied. Part (1) implies that $\check{H}^*(W_{\infty},\infty)$ is finite dimensional (because it is isomorphic to a subspace of a finite dimensional space) and part (2) means that $(f_{\#})^* \colon  \check{H}^*(N/L,[L]) \to \check{H}^*(N/L,[L])$ and $\tilde{f}^* \colon  \check{H}^*(W_{\infty},\infty) \to \check{H}^*(W_{\infty},\infty)$ are shift equivalent, proving the proposition.
\end{proof}

When analyzing a specific isolated invariant set $X$ it is often the case that index pairs and the associated index map can be explicitly computed, even with numerical methods (see \cite{Mro2} and its references), rendering the Conley index an essentially computable object (this is precisely one of the reasons why the Conley index is so useful). As a consequence, the Betti numbers of $W^u_{\infty}$ are also explicitly computable. Indeed, in the situation just considered in the previous paragraph while proving Proposition \ref{prop:conley}, an application of Remark \ref{rem:betti} shows that $\beta^q(W^u_{\infty},\infty) = \beta^q(N/L,[L]) - m_q(0)$, where $m_q(0)$ is the algebraic multiplicity of $\lambda = 0$ as a root of the characteristic polynomial of $(f_{\#})^*$ in $\check{H}^q(N/L,[L])$. This also justifies our claim made in the statement of Theorem \ref{teo:chiX} that $\chi(W^u_{\infty},\infty)$ is explicitly computable.

\begin{proof}[Proof of Proposition \ref{prop:Xtilde}] One only needs to observe that $\tilde{X}$ can alternatively be defined, taking into account that $f = f_{\#}$ on $X$, as the inverse limit \[\tilde{X} = \varprojlim\ \left\{ \ \xymatrix{X & X \ar[l]_-{f} & X \ar[l]_-{f} & \ldots \ar[l]} \ \right\}\] Then Lemma \ref{lem:dirlim} applies directly to $V = \check{H}^q(X)$ and $\varphi = f^*$ to yield the desired result.
\end{proof}

\subsection{Notations and remarks}\label{subsec:convention} To aid intuition it is useful (and harmless) to think of $W_{\infty}$ as the compactification of the unstable manifold $W^u$ (endowed with the intrinsic topology) with the point at infinity $\infty$ although, as we have already explained, this is not rigorously true. Thus \emph{from now on we shall revert to the more expressive notation $W^u_{\infty}$ for $W_{\infty}$ and speak of the ``compactified unstable manifold''}. This accords with the notation used in the Introduction and should not cause any confusion.

One more comment may be in order. We have defined $W^u_{\infty} (=W_{\infty})$ as the inverse limit of the inverse sequence \eqref{sec:NL}, which in principle could depend on the choice of the index pair $(N,L)$, but this is not reflected in the notation; in fact, we speak of ``the'' compactified unstable manifold as if it were a canonical object. It turns out that this is the case. This was proved by Robbin and Salamon \cite[Theorem 7.3, p. 390]{robinsalamon1} for diffeomorphisms (and translates directly to homeomorphisms) but holds as well for continuous maps. For completeness we have included a proof in Appendix \ref{app:canonical}, although we will not make use of this result.

\begin{proposition}\label{prop:unstablecanonical}
$W_{\infty}$ is a canonical object, it does not depend on the choice of index pair.
\end{proposition}

Another technical details arises when one replaces $f$ by an iterate $f^n$, $n \ge 1$. It turns out that $X$ is still an isolated invariant set for $f^n$ and one can apply our results to this map. However, as the next proposition shows, all the objects associated to $f^n$ can be recovered from those associated to $f$. The proof is also postponed to Appendix \ref{app:canonical}.

\begin{proposition}\label{prop:iteratesappendix}
For any $n \ge 1$ the following statements hold:
\begin{itemize}
\item $X$ is an isolated invariant set for $f^n$.
\item The compactified unstable manifold for $f^n$, denoted $W_{n, \infty}$, is homemorphic to $W_{\infty}$.
\item The maps $\widetilde{f^n} \colon W_{n, \infty} \to W_{n, \infty}$ and $(\tilde{f})^n \colon W_{\infty} \to W_{\infty}$ are conjugate.
\end{itemize}
\end{proposition}

Note that Proposition \ref{prop:iteratesappendix} explains why the conclusion of Corollary \ref{coro:genA_intro} is expressed in terms of the fixed points of $\varphi^n$: the set of essential quasicomponents of $W_{n, \infty}$ and $W_{\infty}$ is identical and the permutation induced by $\widetilde{f^n}$ in the former corresponds to the $n$th power of the permutation of the latter induced by $\tilde{f}$.

The description above together with Proposition \ref{prop:conley} yield as a corollary the relationship between the Conley indices of $X$ for $f$ and $f^n$: the $n$th power of an endomorphism contained in $h^q(f, X)$ belongs to $h^q(f^n, X)$. This fact can be deduced in a simple manner using adapted index pairs: start with an index pair $(N, L)$ and replace it with $(N, L_n)$, index pair for $f$ and $f^n$, where $L_n$ consists of the points $x$ of $N$ such that $\{x, f(x), \ldots, f^n(x)\} \cap L \neq \emptyset$, then the associated index maps $f_{\#}$ and $f^n_{\#}$ satisfy $(f_{\#})^n = f^n_{\#}$ and the conclusion follows. A consequence of this connection is that the eigenvalues of $h^q(f^n, X)$ are exactly the $n$th powers of the eigenvalues of $h^q(f, X)$. Thus, the knowledge of the spectrum of the latter is enough to compute the traces of $h^q(f^n, X)$ for any $n \ge 1$.

\subsection{The quasicomponents of $W^u_{\infty} \setminus (\tilde{X} \cup \infty)$} Recall that the \emph{quasicomponent} of a point $z$ in a space $Z$ is the intersection of all the clopen sets that contain $z$.

Elaborating on a definition already given in the Introduction, let us say that a component or a quasicomponent of $W^u_{\infty} \setminus (\tilde{X} \cup \infty)$ \emph{reaches} $\infty$ if its closure in $W^u_{\infty}$ contains $\infty$ and that it \emph{reaches} $\tilde{X}$ if its closure in $W^u_{\infty}$ has a nonempty intersection with $\tilde{X}$. An essential quasicomponent, then, is one that reaches both $\infty$ and $\tilde{X}$.

Throughout this section we will always assume that $W^u_{\infty}$ has finitely many connected components. This is equivalent to assuming that  the \v{C}ech cohomology group $\check{H}^0(W^u_{\infty},\infty)$ is finitely generated, and it is guaranteed whenever Proposition \ref{prop:conley} is applicable; in particular, whenever the phase space is a topological manifold or, more generally, a locally compact absolute neighbourhood retract. (In fact, it holds whenever $X$ has an index pair $(N,L)$ such that $N$ has finitely many connected components).

The following remark is very simple but worth singling out for later reference:

\begin{remark} \label{rem:components} Let $C$ be a component of $W^u_{\infty}$ other than the one that contains $\infty$. Then $C$ is a component of $\tilde{X}$.
\end{remark}
\begin{proof} First notice that, since $W^u_{\infty}$ has finitely many connected components, there exists a power of $\tilde{f}$ that leaves $C$ invariant. Now, $C$ is a compact (being a component of the compact space $W^u_{\infty}$) periodic subset of $W^u_{\infty} \setminus \infty$, which is the basin of repulsion of $\tilde{X}$. These two facts together imply that $C \subseteq \tilde{X}$ and, in particular, there exists a component $\tilde{X}_0$ of $\tilde{X}$ such that $C \subseteq \tilde{X}_0$. In turn, there exists a component of $W^u_{\infty}$ that contains $\tilde{X}_0$, but this must be $C$ itself (otherwise one component of $W^u_{\infty}$ would contain another). Thus $C = \tilde{X}_0$, as was to be shown.
\end{proof}

\begin{proposition} \label{prop:reach} Every component of $W^u_{\infty} \setminus (\tilde{X} \cup \infty)$ reaches either $\infty$, $\tilde{X}$ or both. Consequently the same is true of its quasicomponents, because they are unions of components.
\end{proposition}
\begin{proof} We make use of the following result \cite[Theorem 1.10, p. 101]{wilder1}: if $Z$ is a connected, compact (Hausdorff) space, $K \subseteq Z$ is a closed nonempty subspace and $C$ is a component of $Z \setminus K$, then $\overline{C} \cap K$ is nonempty. This result can be readily generalized to a non connected $Z$: by considering each component $Z_i$ of $Z$ and the (possibly empty) closed subset $K_i := K \cap Z_i$ in turn, one concludes that in this more general situation either $\overline{C} \cap K$ is nonempty or $C$ actually coincides with a component of $Z$.

Now let $Z = W^u_{\infty}$ and $K = \tilde{X} \cup \infty$ and let $C$ be a connected component of $Z \setminus K = W^u_{\infty} \setminus (\tilde{X} \cup \infty)$. It follows from the above that either (i) $\overline{C}$ has a nonempty intersection with $K$, in which case it clearly reaches $\infty$ or $\tilde{X}$ as was to be shown, or (ii) $C$ is actually a connected component of $W^u_{\infty}$. In this case $C$ must be the component that contains $\infty$ or either a component of $\tilde{X}$ by Remark \ref{rem:components}, but none of these is possible since $C$ is a subset of $W^u_{\infty} \setminus (\tilde{X} \cup \infty)$.
\end{proof}

The next result characterizes when there exist essential quasicomponents:

\begin{proposition} \label{prop:noessential} The following are equivalent:
\begin{itemize}
	\item[(i)] $X$ is an asymptotically stable attractor.
	\item[(ii)] $W^u_{\infty}$ consists only of $\tilde{X}$ and the point $\infty$.
	\item[(iii)] There are no essential quasicomponents.
\end{itemize}
\end{proposition}

To establish the proposition we need the following technical lemma, which will also be useful later on. Its proof is postponed to Appendix \ref{app:two}.

\begin{lemma} \label{lem:quasi4} Let $O \subseteq W^u_{\infty} \setminus (\tilde{X} \cup \infty)$ be a clopen set that contains all the essential quasicomponents. Define the sets \[G_{\tilde{X}} := \bigcup \{F : F \text{ is a quasicomponent that reaches $\tilde{X}$ and such that } F \cap O = \emptyset\}\] and \[G_{\infty} := \bigcup \{F : F \text{ is a quasicomponent that reaches $\infty$ and such that } F \cap O = \emptyset\}.\] Then:
\begin{itemize}
	\item[(1)] $W^u_{\infty} \setminus (\tilde{X} \cup \infty) = O \uplus G_{\tilde{X}} \uplus G_{\infty}$.
	\item[(2)] $G_{\tilde{X}}$ is a clopen subset of $W^u \setminus \tilde{X}$ and $G_{\infty}$ is a clopen subset of $W^u_{\infty} \setminus \infty$.
	\item[(3)] $\overline{G}_{\tilde{X}}$ is disjoint from $\infty$ and $\overline{G}_{\infty}$ is disjoint from $\tilde{X}$. Here the closures are taken in $W^u_{\infty}$ (and are therefore compact).
\end{itemize}
\end{lemma}

Now we can prove Proposition \ref{prop:noessential}:

\begin{proof}[Proof of Proposition \ref{prop:noessential}] The implications (i) $\Rightarrow$ (ii) $\Rightarrow$ (iii) are trivial.

(iii) $\Rightarrow$ (ii) Suppose there are no essential quasicomponents and apply Lemma \ref{lem:quasi4} with $O = \emptyset$. The condition $F \cap O = \emptyset$ in the definitions of $G_{\tilde{X}}$ and $G_{\infty}$ is then vacuous whereas the other condition is $\tilde{f}$--invariant. Thus $G_{\tilde{X}}$ and $G_{\infty}$ are both invariant under $\tilde{f}$. By part (3) of the lemma we see that both $\tilde{X} \cup \overline{G}_{\tilde{X}}$ and $\infty \cup \overline{G}_{\infty}$ are compact invariant subsets of $W^u_{\infty} \setminus \infty$ and $W^u_{\infty} \setminus \tilde{X}$ respectively. But the latter are the basins of repulsion and attraction of $\tilde{X}$ and $\infty$, so this implies that $\overline{G}_{\tilde{X}} \subseteq \tilde{X}$ and $\overline{G}_{\infty} = \{\infty\}$. In particular also $G_{\tilde{X}} \subseteq \tilde{X}$ and $G_{\infty} \subseteq \{\infty\}$, and from part (1) of the lemma it follows that $W^u_{\infty} = \tilde{X} \cup \infty$.

(ii) $\Rightarrow$ (i) It will suffice to prove the following: given any isolating neighborhood $A \subset N \setminus L$ for $X$, there is a neighbourhood $V$ of $X$ such that $\bigcup_{k \geq 0} f^kV \subseteq A$. This guarantees that $X$ is an asymptotically stable attractor.

We reason by contradiction. Suppose that no such neighbourhood $V$ exists. Then there exists a sequence $(p_k)$ in $A \setminus X$ that converges to $X$, in the sense that it eventually enters any given neighbourhood of $X$, and such that the forward trajectory of each $p_k$ exits $A$. For each $k$ denote by $n_k \geq 0$ the first iterate of $p_k$ that exits $A$, so that $\{p_k,f(p_k),\dots,f^{n_k-1}(p_k)\} \subseteq A$ and $f^{n_k}(p_k) \not\in A$. Clearly $n_k \to +\infty$ because $p_k$ approaches $X$.

Since $f(X) = X \subseteq \int(A)$, there exists an open neighbourhood $U$ of $X$ such that both $U$ and $f(U)$ are contained in the interior of $A$. In particular $f^{n_k-1}(p_k)$ cannot belong to $U$ because its next iterate is already outside $A$ (by the definition of $n_k$) whereas the points in $U$ have the property that their next iterate still lies in $A$. Thus the sequence $(f^{n_k-1}(p_k))_{k \geq 0}$ is contained in the compact set $A \setminus U$ and so, after passing to a subsequence, we may assume that it converges to some $q \in A \setminus U$.
As limit of points that belong to $\Inv^-_m(A)$, $q \in \Inv^-_m(A)$ for every $m \ge 1$ so, by Lemma \ref{lem:invm=inv}, $q \in \Inv^-(A)$. It follows that $q$ has a full negative semiorbit in $A$; that is, we can find points $q_i \in A$ such that $q_0 = q$ and $f(q_{i+1}) = q_i$ for every $i$. But then $\xi = (q_0, q_1, q_2, \ldots) \in W^u_{\infty}$ by definition, and $\xi \not\in \tilde{X}$ (because $q_0 = q \in A \setminus U \subseteq A \setminus X$) and $\xi \neq \infty$, because $q_0 = q \not\in L$. Thus we have exhibited a point $\xi \in W^u_{\infty} \setminus (\tilde{X} \cup \infty)$.
\end{proof}

The following particular case will be useful later on:

\begin{corollary} \label{cor:attractor} Let $X$ be connected. Then $W^u_{\infty}$ is connected whenever $X$ is not an asymptotically stable attractor.
\end{corollary}
\begin{proof} Observe that by Proposition \ref{prop:Xtilde} the hypothesis that $X$ is connected implies (through the \v{C}ech cohomology groups in degree zero) that $\tilde{X}$ is connected as well. It follows from Remark \ref{rem:components} that $W^u_{\infty}$ can have at most two connected components: $C_{\infty}$ containing $\infty$ and $\tilde{X}$ itself. The assumption that $W^u_{\infty}$ is not connected then requires that $C_{\infty}$ be disjoint from $\tilde{X}$. The same argument used in the proof of Remark \ref{rem:components}, now applied to the compact invariant subset $C_{\infty}$ of $W^u_{\infty} \setminus \tilde{X}$, shows that $C_{\infty} = \{\infty\}$. Thus $W^u_{\infty} = \tilde{X} \uplus \infty$ and it follows from Proposition \ref{prop:noessential} that $X$ is an asymptotically stable attractor.
\end{proof}

\section{Proofs of Corollaries \ref{coro:genA_intro} and \ref{coro:genB_intro}} \label{sec:proofs34}

Having introduced the necessary background about the intrinsic topology, we are now ready to prove Corollaries \ref{coro:genA_intro} and \ref{coro:genB_intro}, which generalize the main theorems of \cite{HCR}. As mentioned in the Introduction, they follow immediately from Theorems \ref{teo:main1_intro} and \ref{teo:main2_intro} through a simple algebraic computation.

\subsection{Proof of Corollary \ref{coro:genA_intro}} When $X$ is an asymptotically stable attractor Corollary \ref{coro:genA_intro} is vacuously trivial. To see this, first observe that $W^u_{\infty} = \tilde{X} \uplus \infty$ and so in particular $\im\ i^* = \check{H}^1(\tilde{X})$. Also, there are no essential quasicomponents so $\#\Fix(\varphi) = 0$. Since by Proposition \ref{prop:conley} the trace of the first Conley index of $X$ is just the trace of $\tilde{f}^* \colon \check{H}^1(\tilde{X}) \to \check{H}^1(\tilde{X})$, Corollary \ref{coro:genA_intro} reduces to the trivial statement that the trace of $\tilde{f}^*$ on $\check{H}^1(X)$ is the trace of $\tilde{f}^*$ on $\im\ i^*$, which is certainly true.

Suppose now that $X$ is not an (asymptotically stable) attractor. Consider the exact sequence for the triple $(W^u_{\infty},\tilde{X} \cup \infty,\infty)$ in \v{C}ech cohomology. We may identify $\check{H}^*(\tilde{X} \cup \infty,\infty) = \check{H}^*(\tilde{X})$; moreover, $\check{H}^0(W^u_{\infty},\infty) = 0$ because $W^u_{\infty}$ is connected by Corollary \ref{cor:attractor}. Therefore the exact sequence reads, in the lowest degrees, \[\xymatrix{\check{H}^1(\tilde{X})  & \ar[l]_-{i^*} \check{H}^1(W^u_{\infty},\infty) & \ar[l] \check{H}^1(W^u_{\infty},\tilde{X} \cup \infty) & \ar[l] \check{H}^0(\tilde{X}) & \ar[l] 0}\] where $i$ denotes the inclusion $(\tilde{X},\emptyset) \subseteq (W^u_{\infty},\tilde{X} \cup \infty)$. Replacing $\check{H}^1(\tilde{X})$ with the image of $i^*$ yields \[\xymatrix{0 & \im\ i^* \ar[l] & \ar[l]_-{i^*} \check{H}^1(W^u_{\infty},\infty) & \ar[l] \check{H}^1(W^u_{\infty},\tilde{X} \cup \infty) & \ar[l] \check{H}^0(\tilde{X}) & \ar[l] 0}\] Stacking two of these sequences on top of each other and connecting them with $\tilde{f}$ we obtain the commutative diagram \[\xymatrix{0 & \im\ i^* \ar[l] \ar[d]^{\tilde{f}^*|_{\im\ {i}^*}} & \ar[l]_-{i^*} \check{H}^1(W^u_{\infty},\infty)  \ar[d]^{\tilde{f}^*} & \ar[l] \check{H}^1(W^u_{\infty},\tilde{X} \cup \infty) \ar[d]^{\tilde{f}^*} & \ar[l] \check{H}^0(\tilde{X}) \ar[d]^{(\tilde{f}|_{\tilde{X}})^*} & \ar[l] 0 \\ 0 & \im\ i^* \ar[l] & \ar[l]_-{i^*} \check{H}^1(W^u_{\infty},\infty) & \ar[l] \check{H}^1(W^u_{\infty},\tilde{X} \cup \infty) & \ar[l] \check{H}^0(\tilde{X}) & \ar[l] 0  }\]

All the vector spaces that appear in the above diagram have finite dimension; hence, the traces of the vertical arrows are all well defined. Among them we identify:
\begin{itemize}
	\item By Proposition \ref{prop:conley}, $\trace\ (\tilde{f}^* \colon \check{H}^1(W^u_{\infty},\infty) \to \check{H}^1(W^u_{\infty},\infty)) = \trace\ h^1(f,X)$.
	\item By Theorem \ref{teo:main1_intro}, $\trace\ (\tilde{f}^* \colon \check{H}^1(W^u_{\infty},\tilde{X} \cup \infty) \to \check{H}^1(W^u_{\infty},\tilde{X} \cup \infty)) = \#\Fix(\varphi)$.
	\item Since $X$ is connected by assumption, it follows from Proposition \ref{prop:Xtilde} in degree zero that $\tilde{X}$ is connected too. Therefore $\trace\ (\tilde{f}^* \colon \check{H}^0(\tilde{X}) \to \check{H}^0(\tilde{X})) = 1$.
\end{itemize}

The additivity of the trace then gives the equation \[ \trace\ h^1(f,X) = \#\Fix(\varphi) + \trace\ (\tilde{f}^*|_{\im\ {i}^*}) - 1,\] which is precisely what we wanted.

\subsection{Proof of Corollary \ref{coro:genB_intro}} The condition that $\tilde{f}$ does not fix any essential quasicomponent implies, through the description of $\check{H}^q(W^u_{\infty},\tilde{X})$ given in Theorem \ref{teo:main2_intro} for $q > 1$, that the trace of $\tilde{f}^*$ on $\check{H}^q(W^u_{\infty},\tilde{X})$ is zero. An appropriate truncation of the exact sequence for the pair $(W^u_{\infty},\tilde{X})$ yields \[\xymatrix{0 & \ar[l] {\rm im}\ (i^*)^q & \check{H}^q(W^u_{\infty}) \ar[l] & \check{H}^q(W^u_{\infty},\tilde{X}) \ar[l] & \check{H}^{q-1}(\tilde{X}) \ar[l] & {\rm im}\ (i^*)^{q-1} \ar[l] & 0 \ar[l]}\] and the additivity of the trace then gives \[\trace(\tilde{f}^*|_{\check{H}^q(W^u_{\infty})}) =  \trace(\tilde{f}^*|_{{\rm im}\ (i^*)^q}) - \trace(\tilde{f}^*|_{\check{H}^{q-1}(\tilde{X})}) + \trace(\tilde{f}^*|_{{\rm im}\ (i^*)^{q-1}})\] where we have already omitted the summand corresponding to $\check{H}^q(W^u_{\infty},\tilde{X})$ since that is zero as just mentioned. By Proposition \ref{prop:conley} the trace on the left hand side of the equality coincides with the trace of $h^q(f,X)$ (the absence of the basepoint $\infty$ is irrelevant here since $q > 1$), and by Proposition \ref{prop:Xtilde} we may replace the middle summand with $\trace(f^*|_{\check{H}^{q-1}(X)})$. This proves the corollary.

\section{Summing power series in cohomology \label{sec:series}}

In the Introduction we provided some heuristic motivation to study the sum of power series in cohomology as a means to obtain information about the relation between the cohomology of an attractor and its basin of attraction. In this section we develop and formalize these ideas. We set ourselves in a very general context: let $Z$ be a paracompact space, $g \colon  Z \to Z$ a continuous map and $K \subseteq Z$ a compact, asymptotically stable attractor for $g$ whose basin of attraction is all of $Z$. Recall that this means that the following two conditions are satisfied: (i) $K$ has a neighbourhood basis comprised of positively invariant sets and (ii) for every $p \in Z$ and every neighbourhood $P$ of $K$ there exists an iterate $n_0$ such that $g^{n_0}(p) \in P$. (And in particular, choosing $P$ to be positively invariant by (i) one has $g^n(p) \in P$ for every $n \geq n_0$.)

We denote by $\check{H}^q(Z,K)$ the \v{C}ech cohomology groups of the pair $(Z,K)$. Coefficients are assumed to be taken in some ring $R$ (commutative and with unit) that will not always be displayed explicitly in the notation. The cohomology groups $\check{H}^q(Z,K)$ are $R$--modules. In any dimension $q \geq 0$ there is an induced endomorphism $g^* \colon  \check{H}^q(Z,K) \to \check{H}^q(Z,K)$. We use the standard notation $(g^*)^j$ to mean the composition of $g^*$ with itself $j$ times, setting $(g^*)^j = {\rm id}$ for $j = 0$.

Recall that a \emph{formal power series} with coefficients in $R$ is an expression of the form $A(x) := \sum_{j=0}^{\infty} a_j x^j$, where $x$ is just a dummy variable that serves as a placeholder and the coefficients $a_j \in R$. Formal power series can be added and multiplied in the natural way: if $A(x) = \sum_{j=0}^{\infty} a_j x^j$ and $B(x) = \sum_{j=0}^{\infty} b_j x^j$, then \[(A+B)(x) := \sum_{j=0}^{\infty} (a_j+b_j)x^j \quad\quad \text{and} \quad\quad (AB)(x) := \sum_{j=0}^{\infty} c_j x^j, \text{where } c_j := \sum_{k=0}^j a_k b_{j-k}.\] The set of all power series in the indeterminate $x$ with coefficients in $R$, endowed with these operations, is a commutative ring with unit denoted by $R[[x]]$. A property of this ring that we shall use repeatedly is the following: a series $A(x) = \sum_{j=0}^{\infty} a_j x^j$ has a multiplicative inverse if, and only if, its independent term $a_0$ has a multiplicative inverse in the ring $R$ (this fact is easily proved by writing out the equation $A(x) B(x) = 1$ in full and solving for the unknown coefficients $b_j$ in terms of the $a_j$).


Evidently a polynomial $A(x) = \sum_{j=0}^N a_j x^j$ can be considered as a (finite) power series. Formally replacing $x$ with $g^*$ in $A(x)$ yields the expression $A(g^*) := \sum_{j=0}^N a_j (g^*)^j$, which is a well defined endomorphism of $\check{H}^q(Z,K)$. Our goal in to use the extra dynamical structure on $(Z,K)$, namely that $K$ is an attractor with basin of attraction $Z$, to give a meaningful definition of $A(g^*)$ when $A(x)$ is not a polynomial but more generally a formal power series, thus ascribing a precise meaning to expressions of the form $A(g^*) = \sum_{j=0}^{\infty} a_j (g^*)^j$ as an endomorphism of $\check{H}^q(Z,K)$. This interpretation will be consistent with the one for finite series (that is, for polynomials) and verify the relations \begin{equation} \label{eq:ring_hom} (A+B)(g^*) = A(g^*)+B(g^*) \quad \quad \text{and} \quad\quad (AB)(g^*) = A(g^*) B(g^*),\end{equation} where for notational simplicity we shall henceforth omit the symbol $\circ$ and simply intend the product $A(g^*) B(g^*)$ on the right hand side of the second formula to mean the composition of the endomorphisms $A(g^*)$ and $B(g^*)$. In a fancier language, these properties mean that summation of series will provide a homomorphism from $R[[x]]$ to ${\rm Hom}(\check{H}^*(Z,K;R))$.

Since $K$ is a forward invariant set, the map $g$ descends to a continuous map $\bar{g} \colon  Z/K \to Z/K$ on the quotient space $Z/K$ obtained by collapsing $K$ to a single point $a$. Clearly $a$ is still an asymptotically stable, global attractor for $\bar{g}$. The strong excision property of \v{C}ech cohomology (see \cite[p. 318]{spanier1}) implies that there is an isomorphism $\check{H}^*(Z,K) \cong \check{H}^*(Z/K,a)$ that conjugates $g^*$ and ${\bar{g}}^*$, so instead of defining the summation of series on $\check{H}^*(Z,K)$ we may as well define it on $\check{H}^*(Z/K,a)$. Thus, renaming $Z/K$ and $\bar{g}$ again as $Z$ and $g$ for notational simplicity, for the rest of this section we will set ourselves, without loss of generality, in the case when $g \colon  Z \to Z$ is a continuous map having a single point $a$ as a global, asymptotically stable attractor.

\subsection{Preliminary results about cohomology} In our arguments we will need to construct explicitly certain cohomology classes, and this is most easily done in the Alexander--Spanier cohomology theory. This theory coincides with the more familiar one of \v{C}ech for paracompact Hausdorff spaces \cite[Corollary 8, p. 334]{spanier1}, which certainly encompasses all the situations of interest to us, so we shall treat both theories as equivalent.

Let us begin by recalling the definition of Alexander--Spanier cohomology as given in \cite[pp. 306ff.]{spanier1}. For each degree $q = 0,1,2, \ldots$ denote by $C^q(Z)$ the set of maps $\xi \colon  Z^{q+1} \to R$. We remark that these are just maps in the set-theoretical sense. Clearly $C^q(Z)$ is an $R$--module with pointwise addition and multiplication by scalars. The coboundary homomorphism $\partial \colon  C^q(Z) \to C^{q+1}(Z)$ is defined by the standard formula \[(\partial(\xi))(z_0,\ldots,z_q,z_{q+1}) = \sum_{j=0}^{q+1} (-1)^j \xi(z_0,\ldots,\hat{z}_j,\ldots,z_{q+1})\] where the hat over the entry $\hat{z}_j$ means that it has to be omitted from the arguments on which $\xi$ is evaluated. As usual, $\partial^2 = 0$.

An element $\xi \in C^q(Z)$ is said to be \emph{locally zero} if there exists a covering of $Z$ by open sets such that $\xi(z_0, \ldots, z_q)$ vanishes whenever all the $z_i$ belong to the same member of the covering. It is easy to see that this condition is equivalent to requiring that $\xi$ vanishes on an open neighbourhood of the diagonal $\{(z, \ldots, z) : z \in Z\}$ of $Z^{q+1}$.

Denote by $\overline{C}^q(Z)$ the quotient of $C^q(Z)$ modulo the locally zero functions. The coboundary of a locally zero map is easily seen to be also locally zero, so that the coboundary map $\partial$ descends to a homomorphism $\overline{\partial} \colon  \overline{C}^q(Z) \to \overline{C}^{q+1}(Z)$ between the quotients. One thus obtains a chain complex $\{\overline{C}^q(R),\overline{\partial}\}$ whose cohomology is, by definition, the Alexander--Spanier cohomology of the space $Z$.

An element $u \in \check{H}^q(Z)$ is, then, obtained from a map $\xi \in C^q(Z)$ from a double quotienting process: first one considers the class $\overline{\xi}$ of $\xi$ in the quotient $\overline{C}^q(Z)$ and then the class $[\overline{\xi}]$ of $\overline{\xi}$ after quotienting by ${\rm im}\ \overline{\partial}$. We call $\xi$ a \emph{representative} of $u$. Notice that, in order to define a cohomology class, $\overline{\xi}$ must be a cocycle; that is, $\overline{\partial}(\overline{\xi})$ must be zero in $\overline{C}^q(Z)$. This amounts to requiring that $\partial(\xi)$ be locally zero as an element of $C^{q+1}(Z)$. Conversely, any $\xi \in C^q(Z)$ such that $\partial(\xi)$ is locally zero defines a cohomology class $u = [\overline{\xi}] \in \check{H}^q(Z)$. It is straightforward to check that two maps $\xi, \xi' \in C^q(Z)$ whose coboundaries are locally zero define the same cohomology classes in $\check{H}^q(Z)$ if and only if there exists $\varphi \in C^{q-1}(Z)$ such that $\xi-\xi'-\partial \varphi$ is locally zero.
\smallskip

{\it Notation.} To simplify notation as much as possible we shall frequently omit parentheses and commas as long as this does not cause any confusion. Thus for instance we may write the expression for the coboundary map $\partial$ on $C^q(Z)$ simply as $\partial \xi (z_0 \ldots z_q z_{q+1}) = \sum_{j=0}^{q+1} (-1)^j \xi(z_0 \ldots \hat{z}_j \ldots z_{q+1})$. For a map $\xi \in C^q(Z)$ we will on occasion use the expression ``$\xi$ vanishes on $P \subseteq Z$'' to mean that $\xi(z_0\ldots z_q)$ is zero whenever all the $z_i \in P$; that is, that $\xi$ vanishes on $P^{q+1}$. Hopefully these conventions will not cause any confusion.
\smallskip

Any map $g\colon Z_1 \to Z_2$ induces a homomorphism $g^* \colon  C^q(Z_2) \to C^q(Z_1)$ given by $\xi \longmapsto g^*\xi$ with $g^*\xi$ defined by $g^*\xi(z_0 \ldots z_q) := \xi(g(z_0) \ldots g(z_q))$. One can easily check that this $g^*$ commutes with the respective coboundary operators of $Z_1$ and $Z_2$. If $g$ is continuous then $g^*$ sends locally zero functions to locally zero functions and therefore descends to another homomorphism $\overline{C}^q(Z_2) \to \overline{C}^q(Z_1)$. This map still commutes with the coboundary maps of $Z_1$ and $Z_2$, and so induces a homomorphism $g^* \colon  \check{H}^q(Z_2) \to \check{H}^q(Z_1)$. Notice that, in our spirit of keeping notation simple, we make no notational distinction between the homomorphisms induced by $g$ at the levels of $C^q$ and $\check{H}^q$.

We shall now prove two simple results about representatives of cohomology classes that will be needed later on. To avoid having to distinguish the cases $q = 0$ and $q \geq 1$ it is more convenient to work in the relative cohomology group $\check{H}^q(Z,a)$ instead of the absolute cohomology group $\check{H}^q(Z)$. Its definition exactly parallels all that has been laid out above, with the only exception that in place of $C^q(Z)$ we now consider $C^q(Z,a)$, the set of maps $\xi \colon  Z^{q+1} \to R$ that vanish on $a$; that is, $\xi(a \ldots a) = 0$. For completeness one defines $C^q(Z,a) = \{0\}$ for $q < 0$.

\begin{lemma} \label{lem:poincare} Let $a \in Z$. Assume that a map $\xi \in C^q(Z,a)$ satisfies $\partial \xi = 0$ on some neighbourhood $P$ of $a$. Then there exists $\alpha \in C^{q-1}(Z,a)$ such that $\xi = \partial \alpha$ on $P$.
\end{lemma}

\begin{proof} When $q = 0$ the lemma simply claims that $\xi = 0$ on $P$, since $C^{-1}(Z,a) = \{0\}$. This is very easy to prove: for any point $z \in P$ the equality $\partial \xi = 0$ on $P$, evaluated on $(az)$, implies that $\xi(z) = \xi(a) = 0$.

In the case $q \geq 1$, set $\alpha(z_0 \ldots z_{q-1}) := \xi(a z_0 \ldots z_{q-1})$. Evidently $\alpha(a \ldots a) = 0$ because the same is true of $\xi$. It is easy to check that $\partial \alpha (z_0 \ldots z_q) = \xi(z_0 \ldots z_q) - \partial \xi (az_0 \ldots z_q)$. When $z_0, \ldots, z_q$ all belong to $P$ the second summand in the right hand side vanishes because of the hypothesis concerning $\partial \xi$. Thus $\partial \alpha = \xi$ on $P$, as was to be proved.
\end{proof}

\begin{lemma} \label{lem:represent} Let $a \in Z$. Let $u$ be a cohomology class in $\check{H}^q(Z,a)$. Then:
\begin{itemize}
	\item[(i)] There exists a map $\xi \in C^q(Z,a)$ that vanishes on a neighbourhood of $a$ and represents $u$; that is, $u = [\overline{\xi}]$.
	\item[(ii)] If $\xi' \in C^q(Z,a)$ is any other such map, then there exists $\varphi \in C^{q-1}(Z,a)$ that also vanishes on some neighbourhood of $a$ and satisfies that $\xi-\xi'-\partial \varphi$ is locally zero.
\end{itemize}
\end{lemma}


\begin{proof} (i) Let $\xi_0 \in C^q(Z,a)$ be any cocycle such that $[\overline{\xi_0}] = u$. Since $\xi_0$ is a cocyle, $\partial \xi_0$ is locally zero so, in particular, it vanishes on some neighbourhood $P$ of $a$. By Lemma \ref{lem:poincare} there exists $\alpha \in C^{q-1}(Z,a)$ such that $\xi_0 = \partial \alpha$ on $P$. Set $\xi := \xi_0 - \partial \alpha$. By construction $\xi$ vanishes on $P$ and, since it differs from $\xi_0$ only by a coboundary, $\xi$ is a cocycle whose cohomology class in $\check{H}^q(Z,a)$ is still $u$.

(ii) Since both $\xi$ and $\xi'$ represent the same cohomology class, there exists a map $\varphi_0 \in C^{q-1}(Z,a)$ such that $\xi - \xi' - \partial \varphi_0$ is locally zero. By assumption there exist neighbourhoods $P$ and $P'$ of $a$ over which $\xi$, $\xi'$ vanish. Moreover, since $\xi-\xi'-\partial \varphi_0$ is locally zero, it vanishes over some neighbourhood $Q$ of $a$ too. Replacing $P$, $P'$ and $Q$ by the intersection of the three and denoting the latter by $P$ again we may simply assume that $\xi$, $\xi'$ and $\xi - \xi' - \partial \varphi_0$ all vanish over $P$. This entails that $\partial \varphi_0$ also vanishes over $P$, so by Lemma \ref{lem:poincare} there exists $\alpha \in C^{q-2}(Z,a)$ such that $\varphi_0 = \partial \alpha$ over $P$. Define $\varphi' := \varphi_0 - \partial \alpha$. By construction $\varphi'$ vanishes over $P$. Moreover, $\xi - \xi' - \partial \varphi' = \xi - \xi' - \partial \varphi_0 + \partial \partial \alpha = \xi-\xi'-\partial \varphi_0$ because $\partial^2 = 0$, so $\xi-\xi'-\partial \varphi'$ is locally zero because the same is true of $\xi-\xi'-\partial \varphi_0$.
\end{proof}

\subsection{Summing cohomology classes} Let us return to the case of interest to us, where dynamics are present. Recall that we had reduced ourselves to the case where $g \colon  Z \to Z$ is a continuous map having a single point $a$ as a global, asymptotically stable attractor.

\smallskip

{\it Provisional definition.} Let a sequence of coefficients $a_j \in R$ be fixed once and for all. For any map $\xi \in C^q(Z,a)$ that vanishes on a neighbourhood of the attracting point $a$, define $S(\xi) \in C^q(Z,a)$ by the following formula: \begin{equation} \label{eq:def_T1} S(\xi)(z_0 \ldots z_q) := \sum_{j=0}^{\infty} a_j \xi(g^j(z_0) \ldots g^j(z_q)). \end{equation} The requirement that $\xi$ vanishes on a neighbourhood of $a$ (say $P$) guarantees that this definition is correct. Indeed: given a particular tuple $(z_0 \ldots z_q)$ there exists $j_0$ such that $g^j(z_0),\ldots,g^j(z_q) \in P$ for every $j \geq j_0$, and so the infinite series of Equation \eqref{eq:def_T1} actually truncates to a finite summation with $j$ running only up to $j = j_0$. (Notice however that $j_0$ depends on the particular tuple $z_0 \ldots z_q$ being considered.)
\smallskip

\begin{proposition} \label{prop:summation} $S$ has the following properties:
\begin{itemize}
	\item[(i)] If $\xi_1, \xi_2 \in C^q(Z,a)$ both vanish on a neighbourhood of $a$, so does any linear combination $c_1\xi_1 + c_2\xi_2$, and the relation $S(c_1\xi_1+c_2\xi_2) = c_1S(\xi_1)+c_2S(\xi_2)$ holds.
	\item[(ii)] If $\xi$ vanishes on a neighbourhood of $a$ so does $\partial \xi$, and the relation $\partial (S(\xi)) = S(\partial \xi)$ holds.
	\item[(iii)] If $\xi$ is locally zero (so in particular it vanishes in a neighbourhood of $a$), then so is $S(\xi)$.
\end{itemize}
\end{proposition}
\begin{proof} The proofs of (i) and (ii) are easy and we omit them. To prove (iii) we need the following auxiliary result:
\smallskip

{\it Claim.} For any $q \geq 0$, the diagonal of $Z^{q+1}$ has a basis of neighbourhoods that are positively invariant under $g \times \stackrel{(q+1)}{\ldots} \times g$.

{\it Proof of claim.} For the proof of this claim only, symbols carrying a tilde will denote objects related to the product space $Z^{q+1}$ whereas symbols without a tilde denote objects concerning the original space $Z$. Set $\tilde{g} := g \times \stackrel{(q+1)}{\ldots} \times g$ and $\tilde{Z} := Z^{q+1}$. Let $\tilde{U}$ be a neighbourhood of the diagonal of $\tilde{Z}$. We need to show that there is another neighbourhood $\tilde{V}$ of the diagonal such that $\tilde{V} \subseteq \tilde{U}$ and $\tilde{V}$ is positively invariant under $\tilde{g}$.

Let $P$ be a positively invariant neighbourhood of $a$ in $Z$ so small that $\tilde{P} := P^{q+1}$ is contained in $\tilde{U}$. For any point $z \in Z$ we make the following construction. Since $a$ is a global attractor in $Z$ there exist $j_0$ and a neighbourhood $W$ of $z$ in $Z$ such that $g^{j_0}(W) \subseteq P$. Since $P$ is positively invariant under $g$, all the subsequent iterates of $W$ are contained in $P$ too; that is, $g^j(W) \subseteq P$ for every $j \geq j_0$. Set $\tilde{z} :=(z,\stackrel{(q+1)}{\ldots},z) \in \tilde{Z}$ and $\tilde{W} := W^{q+1} \subseteq \tilde{Z}$. Clearly $\tilde{W}$ is a neighbourhood of $\tilde{z}$ in $\tilde{Z}$. Since all the points $\tilde{z}, \tilde{g}(\tilde{z}), \ldots, \tilde{g}^{j_0-1}(\tilde{z})$ belong to the diagonal of $\tilde{Z}$ and are therefore contained in $\tilde{U}$, by choosing $W$ sufficiently small we can guarantee that $\tilde{W}$ is in turn so small that $\tilde{W}, \tilde{g}(\tilde{W}),\ldots,\tilde{g}^{j_0-1}(\tilde{W})$ are all contained in $\tilde{U}$. Also, since we had $g^j(W) \subseteq P$ for $j \geq j_0$, we have $\tilde{g}^j(\tilde{W}) \subseteq \tilde{P} \subseteq \tilde{U}$ for $j \geq j_0$. Thus, setting $\tilde{V}_{\tilde{z}} := \bigcup_{j \geq 0} \tilde{g}^j(\tilde{W})$ we see that $\tilde{V}_{\tilde{z}} \subseteq \tilde{U}$ and by construction $\tilde{V}_{\tilde{z}}$ is positively invariant under $\tilde{g}$. It only remains to let $\tilde{V} := \bigcup_{z \in Z} \tilde{V}_{\tilde{z}}$: this is a neighbourhood of the diagonal of $\tilde{Z}$ that is contained in $\tilde{U}$ and is positively invariant under $\tilde{g}$. This concludes the proof of the claim. $_{\blacksquare}$
\medskip

Back to the proof of part (iii) of the proposition, assume that $\xi$ is locally zero. Then there exists a neighbourhood $U$ of the diagonal of $Z^{q+1}$ such that $\xi$ vanishes over $U$. By the claim above there exists a positively invariant neighbourhood $V \subseteq U$ of the diagonal. Clearly $\xi$ vanishes over $V$ too, and this entails that $S(\xi)$ also does: indeed, if $(z_0,\ldots,z_q) \in V$ then $(g^j(z_0),\ldots,g^j(z_q)) \in V$ for every $j \geq 0$, so $\xi(g^j(z_0) \ldots g^j(z_q)) = 0$ for every $j \geq 0$. Thus $S(\xi)(z_0 \ldots z_q) = 0$.
\end{proof}

\begin{proposition} Let $u \in \check{H}^q(Z,a)$ be some cohomology class. By Lemma \ref{lem:represent}.(i) there exists a map $\xi \in C^q(Z,a)$ such that $u = [\overline{\xi}]$ and $\xi$ vanishes on some neighbourhood of $a$. Consider the map $S(\xi) \in C^q(Z,a)$ defined by Equation \eqref{eq:def_T1}. Then:
\begin{itemize}
	\item[(a)] $\overline{S(\xi)}$ is a cocycle, so that it indeed determines a cohomology class $v = [\overline{S(\xi)}] \in \check{H}^q(Z,a)$;
	\item[(b)] moreover, $v$ is independent of the choice of $\xi$.
\end{itemize}
\end{proposition}
\begin{proof} (a) We have $\partial (S(\xi)) = S(\partial \xi)$ by Proposition \ref{prop:summation}.(ii). Since $\overline{\xi}$ represents a cohomology class it must be a cocycle; that is, $\partial \xi$ is locally zero. By Proposition \ref{prop:summation}.(iii) this implies that $S(\partial \xi)$ is locally zero too. Hence $\partial S(\xi)$ is locally zero and therefore $\overline{S(\xi)}$ is a cocycle in $\overline{C}^q(Z,a)$, so it does define a cohomology class $v := [\overline{S(\xi)}]$.

(b) Suppose that $\xi'$ is another map in $C^q(Z,a)$ that represents $u$ and vanishes on a neighbourhood of $a$. Let $v' := [\overline{S(\xi')}]$. We need to show that $v = v'$, which amounts to proving that there exists $\psi \in C^{q-1}(Z,a)$ such that $S(\xi) - S(\xi') - \partial \psi$ is locally zero. By Lemma \ref{lem:represent}.(ii) there exists $\varphi \in C^{q-1}(Z,a)$ such that $\xi - \xi' - \partial \varphi$ is locally zero and $\varphi$ vanishes in some neighbourhood of $a$. The latter condition implies that $S(\varphi)$ is well defined, and we set $\psi := \partial (S(\varphi))$. Using Proposition \ref{prop:summation}.(i) and (ii) and observing that each of $\xi$, $\xi', \varphi$ vanish on a neighbourhood of $a$ we may write the difference $S(\xi) - S(\xi') - \partial \psi$ as \[S(\xi)-S(\xi')-\partial (S(\varphi)) = S(\xi)-S(\xi')-S(\partial \varphi) = S(\xi-\xi'-\partial \varphi). \] Since $\xi-\xi'-\partial \varphi$ is locally zero, so is $S(\xi-\xi'-\partial \varphi)$ by Proposition \ref{prop:summation}.(iii). Thus we see that $S(\xi) - S(\xi') - \partial \psi$ is locally zero, which was to be proved.
\end{proof}

The above results justify the following definition:

\begin{definition} Let $u \in \check{H}^q(Z,a)$ and $a_j \in R$. Choose a map $\xi \in C^q(Z,a)$ such that $u = [\overline{\xi}]$ and $\xi$ vanishes on some neighbourhood of $a$. Then the sum of the series $\sum_{j=0}^{\infty} a_j (g^*)^j(u)$ is defined as the cohomology class $v = [\overline{S(\xi)}]$, where $S(\xi)$ is given by Equation \eqref{eq:def_T1}.
\end{definition}

It is evident that this definition reduces to the standard one for finite sums.


We are finally ready to interpret power series in $g^*$ as endomorphisms of $\check{H}^q(Z,a)$. For any formal power series $A(x) = \sum_{j=0}^{\infty} a_j x^j \in R[[x]]$ define a map $A(g^*) \colon  \check{H}^q(Z,a) \to \check{H}^q(Z,a)$ by \[A(g^*) \colon  u \longmapsto \sum_{j=0}^{\infty} a_j (g^*)^j(u).\] That $A(g^*)$ is indeed a homomorphism of $R$--modules is a straightforward consequence of Proposition \ref{prop:summation}.(i) and the definition of $\sum_{j=0}^{\infty} a_j (g^*)^j(u)$. Thus, it only remains to check that the equalities of Equation \ref{eq:ring_hom} are satisfied. This is the content of the following result:

\begin{proposition} \label{prop:T_hom} For any two series $A,B \in R[[x]]$ one has $(A+B)(g^*) = A(g^*) + B(g^*)$ and $(AB)(g^*) = A(g^*)B(g^*)$.
\end{proposition}
\begin{proof} We will only prove that $(AB)(g^*) = A(g^*) B(g^*)$, since the proof of the other assertion is similar but easier. We begin by analyzing how the composition $A(g^*)B(g^*)$ acts on a cohomology class $u \in \check{H}^q(Z,a)$. Let $v := B(g^*)(u)$ and $w := A(g^*)(v) = A(g^*)B(g^*)(u)$.

Set $A(x) = \sum_{j=0}^{\infty} a_j x^j$ and $B(x) = \sum_{k=0}^{\infty} b_k x^k$. Pick a map $\xi \in C^q(Z,a)$ that represents $u$; that is, $u = [\overline{\xi}]$, and vanishes on a neighbourhood $P$ of $a$. Without loss of generality we may take $P$ to be positively invariant. By definition $v = [\overline{\eta}]$ where $\eta \in C^q(Z,a)$ is given by \begin{equation} \label{eq:comp_eta} \eta(z_0 \ldots z_q) =  \sum_{k=0}^{\infty} b_k \xi(g^k(z_0) \ldots g^k(z_q)).\end{equation} Since $P$ was chosen to be positively invariant, not only $\xi$ but also $(g^*)^k\xi$ vanish on $P$ for any $k$, so the map $\eta$ vanishes on the same neighbourhood $P$ of $a$. Therefore we may use it to compute $w = [\overline{\alpha}]$ with $\alpha \in C^q(Z,a)$ given by \begin{equation} \label{eq:comp_alpha} \alpha(z_0 \ldots z_q) = \sum_{j=0}^{\infty} a_j \eta(g^j(z_0)\ldots g^j(z_q)).\end{equation}

Evaluating Equation \eqref{eq:comp_eta} on $(g^j(z_0) \ldots g^j(z_q))$ rather than on $(z_0 \ldots z_q)$ and replacing it in Equation \eqref{eq:comp_alpha} yields the following expression: \[\alpha(z_0 \ldots z_q) = \sum_{j=0}^{\infty} \sum_{k=0}^{\infty} a_jb_k \xi(g^{k+j}(z_0) \ldots g^{k+j}(z_q))\] which, upon eliminating the summation index $k$ in favour of $\ell := j+k$, turns into \[\alpha(z_0 \ldots z_q) = \sum_{j=0}^{\infty} \sum_{\ell=j}^{\infty} a_j b_{\ell-j} \xi(g^{\ell}(z_0) \ldots g^{\ell}(z_q)).\] We recall that all the series are actually finite sums, so these manipulations are legitimate. Interchanging the order of summation and modifying the summation limits accordingly, \begin{equation} \label{eq:comp} \alpha(z_0 \ldots z_q) = \sum_{\ell=0}^{\infty} \sum_{j=0}^{\ell} a_j b_{\ell-j} \xi(g^{\ell}(z_0) \ldots g^{\ell}(z_q)) = \sum_{\ell=0}^{\infty} c_{\ell} \xi(g^{\ell}(z_0)\ldots g^{\ell}(z_q)),\end{equation} where in the last step we have just expressed the sum in terms of $c_{\ell} := \sum_{j=0}^{\ell} a_j b_{\ell - j}$. These $c_{\ell}$ are readily recognized as the coefficients of the product power series $(AB)(x)$, so Equation \eqref{eq:comp} can be interpreted as stating that $(AB)(g^*)(u)$ is represented by $\alpha$; that is, $(AB)(g^*)(u) = [\overline{\alpha}]$. Since by construction $\alpha$ was a representative of $w = A(g^*)B(g^*)(u)$, it follows that $(AB)(g^*)(u) = A(g^*)B(g^*)(u)$ as was to be shown.
\end{proof}

\subsection{A discrete version of Hastings' theorem} We are now going to obtain a suitable version of Hastings' theorem in the setting of discrete dynamics, characterizing when the inclusion $K \subseteq Z$ induces isomorphism in \v{C}ech cohomology with integer coefficients $R = \mathbb{Z}$. Although this result will not be used in the paper, we feel that it provides an interesting illustration of the use of power series in cohomology.

\begin{lemma}\label{lem:uncountable} If $\check{H}^q(Z,K;R)$ is countable, then for every $u \in \check{H}^q(Z,K;R)$ there exists $d \in \mathbb{N}$ such that $(g^*)^d(u) = 0$. In particular, if $g$ is a homeomorphism then $\check{H}^q(Z,K;R) = 0$.
\end{lemma}
\begin{proof} Fix any element $u \in \check{H}^q(Z,K;R)$. Let $\mathcal{S} := \{0,1\}^{\mathbb{N}}$ denote the set of sequences $(a_j)$ such that each $a_j$ is either $0$ or $1$. To each of these sequences we associate $v := \sum_{j=0}^{\infty} a_j (g^*)^j(u)$. Since $\mathcal{S}$ is uncountable but $\check{H}^q(Z,K;R)$ is countable by assumption, the correspondence $\mathcal{S} \ni (a_j) \longmapsto v \in \check{H}^q(Z,K;R)$ cannot be injective; that is, there must exist two different sequences $(a_j)$ and $(a'_j)$ in $\mathcal{S}$ such that $\sum_{j=0}^{\infty} a_j (g^*)^j(u) = \sum_{j=0}^{\infty} a'_j (g^*)^j(u)$. According to Equation \eqref{eq:ring_hom} we can rewrite this relation as $\sum_{j=0}^{\infty} (a_j-a'_j) (g^*)^j(u) = 0$. Consider the power series $D(x) := \sum_{j=0}^{\infty} (a_j-a'_j) x^j \in R[[x]]$. Letting $d$ be the first entry where the sequences $(a_j)$ and $(a'_j)$ differ, one has $D(x) = x^d D_0(x)$ with $D_0(x) := \sum_{j=d}^{\infty} (a_j-a'_j)x^{j-d}$. The independent term of $D_0$ is the difference $a_d - a'_d$, which is $\pm 1$ because the sequences from $\mathcal{S}$ consist only of zeroes and ones and $a_j \neq a'_j$. Thus $D_0$ has an inverse $G \in R[[x]]$. Multiplying $D(x) = x^d D_0(x)$ through by $G(x)$ yields $ G(x) D(x) = x^d$. In turn, this implies $ G(g^*)D(g^*)(u) = (g^*)^d(u)$ because of Equation \eqref{eq:ring_hom}. However, by construction $D(g^*)(u) = 0$, and it follows that $(g^*)^d(u) = 0$ too. If $g^*$ is an isomorphism this implies $u = 0$.
\end{proof}

The following theorem is the discrete version of Hastings' theorem that we announced earlier. It broadly generalizes one of the main results (Theorem 2) of \cite{pacoyo1}.

\begin{theorem} \label{teo:ANR} Let $K$ be an attractor for a homeomorphism of a locally compact, metrizable ANR. Take coefficients in $R = \mathbb{Z}$. The inclusion of $K$ in its basin of attraction induces isomorphisms in \v{C}ech cohomology if, and only if, the latter has countable \v{C}ech cohomology groups.
\end{theorem}
\begin{proof} As usual, let $Z$ be the basin of attraction of $K$. We will make use of the fact that an attractor in a locally compact ANR always has countable \v{C}ech cohomology (see Proposition \ref{prop:att_countable} in Appendix \ref{app:one}). Thus if the inclusion $K \subseteq Z$ induces isomorphisms, evidently $Z$ has countable \v{C}ech cohomology. Conversely, assume that $Z$ has countable \v{C}ech cohomology. Since $K$ also has countabe \v{C}ech cohomology, it follows from the exact sequence of the pair $(Z,K)$ that $\check{H}^*(Z,K;\mathbb{Z})$ is countable too. Now the previous lemma implies that the group is actually zero.
\end{proof}

As a consequence we have the following:

\begin{corollary} Let $K$ be an attractor for a homeomorphism $g$ of $\mathbb{R}^n$. Take coefficients in $R = \mathbb{Z}$. The inclusion of $K$ in its basin of attraction induces isomorphisms in \v{C}ech cohomology if, and only if, $K$ has finitely generated cohomology.
\end{corollary}
\begin{proof} Suppose first that the inclusion $i \colon  K \subseteq Z$ induces isomorphisms in cohomology. Let $N$ be a compact manifold that satisfies $K \subseteq N \subseteq Z$ and denote by $j \colon  K \subseteq N$ and $k \colon  N \subseteq Z$ the inclusions. Since $i = k \circ j$ and $i^*$ is an isomorphism, $j^*$ is surjective. But $N$, being a compact manifold, has finitely generated \v{C}ech cohomology and therefore so does $K$.

To prove the converse, consider the $n$--sphere $\mathbb{S}^n$ as the compactification of $\mathbb{R}^n$ with a point $\infty$ and extend the dynamical system to all of $\mathbb{S}^n$ by letting $\infty$ be fixed. Denote by $R = \mathbb{S}^n \setminus Z$ and by $Z' = \mathbb{S}^n \setminus K$. It is straightforward to check that $R$ is a repeller whose basin of repulsion is $Z'$. Now, by Alexander duality $H_*(Z') = H^{n-1-*}(K)$ is finitely generated and it follows from the universal coefficient theorem that $H^*(Z')$ is countable. The previous theorem then guarantees that the inclusion $R \subseteq Z'$ induces isomorphisms in integral cohomology, so in particular $\check{H}^*(R)$ has finitely generated integral cohomology. The same argument using Alexander duality and the universal coefficient theorem then shows that $H^*(Z)$ is countable, and the corollary follows from the previous theorem.
\end{proof}

\section{The cohomology of $(W^u_{\infty},\tilde{X})$ \label{sec:describe1}}

We are finally ready to carry out the strategy sketched in the Introduction to obtain a description of the cohomology of $(W^u_{\infty},\tilde{X})$ and the action of $\tilde{f}^*$ on it, proving Theorems \ref{teo:main1_intro} and \ref{teo:main2_intro} (these correspond to Theorems \ref{teo:main1} and \ref{teo:essential} in the present section). As mentioned earlier, in degree one it is convenient to work with the relative group $\check{H}^1(W^u_{\infty},\tilde{X} \cup \infty)$ rather than $\check{H}^1(W^u_{\infty},\tilde{X})$. By Proposition \ref{prop:noessential} both are related (via the inclusion) by $\check{H}^1(W^u_{\infty},\tilde{X} \cup \infty) = \check{H}^1(W^u_{\infty},\tilde{X}) \oplus \mathbb{K}$ unless $X$ is an asymptotically stable attractor, in which case $\check{H}^1(W^u_{\infty},\tilde{X} \cup \infty) = 0$. Similarly, in higher degrees $q \geq 2$ we shall use the strong excision property of \v{C}ech cohomology to write $\check{H}^q(W^u_{\infty},\tilde{X}) = \check{H}^q(W^u_{\infty}/\tilde{X},[\tilde{X}]) = \check{H}^q(W^u_{\infty}/\tilde{X})$, and it is in terms of this latter quotient space that we shall express our results.

It might be useful to have some rough visual reference for the cohomology we intend to compute. Let $W^u_{\infty}/\tilde{X}$ denote the space that results from $W^u_{\infty}$ by collapsing $\tilde{X}$ to a single point $[\tilde{X}]$. Figure \ref{fig:quotient} provides a very schematic drawing intended to illustrate the \emph{a priori} very complicated structure of $W^u_{\infty}/\tilde{X}$. Clearly the dynamics on $W^u_{\infty}$ descends to the quotient space $W^u_{\infty}/\tilde{X}$ yielding also an attractor--repeller structure: $[\tilde{X}]$ is a repeller whose basin of repulsion is $U/\tilde{X}$ and $\infty$ is an attractor whose basin of attraction is $V$. The relative cohomology group $\check{H}^*(W^u_{\infty},\tilde{X} \cup \infty)$ can be thought of as the cohomology of the space that results from $W^u_{\infty}/\tilde{X}$ by collapsing $[\tilde{X}]$ and $\infty$ to a single point.

\begin{figure}[h!]
\begin{pspicture}(0,-0.5)(2,6)
	\rput[bl](0,0){\scalebox{0.75}{\includegraphics{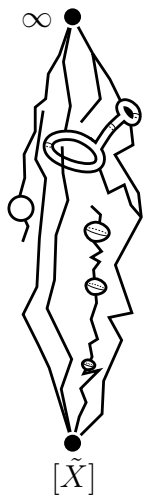}}}
	\rput[bl](0.6,-0.5){$[\tilde{X}]$}
	\rput[tl](0.2,5.9){$\infty$}
\end{pspicture}
\caption{The quotient $W^u_{\infty}/\tilde{X}$ \label{fig:quotient}}
\end{figure}

The set $U \cap V$ that appears in the sequence \eqref{eq:mayer0} can be written and thought of in several equivalent ways: $U \cap V = W^u_{\infty} \setminus (\tilde{X} \cup \infty) = (W^u_{\infty}/\tilde{X}) \setminus ([\tilde{X}] \cup \infty)$. In the sequel we will sometimes alternate between these notations without further comment. Recall that those quasicomponents of this set that are adherent to both $[\tilde{X}]$ and $\infty$ are called essential. We will prove in Proposition \ref{prop:describeH1} that when $X$ has finitely many connected components there are only finitely many essential quasicomponents.

Figure \ref{fig:quotient} only illustrates part of the complexity that may be found in $W^u_{\infty}$, in that it shows that every quasicomponent $F$ of $W^u_{\infty} \setminus (\tilde{X} \cup \infty)$ can have its own, nonzero cohomology. The situation is worsened by the fact that the iterates $\tilde{f}^nF$ (which are also quasicomponents of $W^u_{\infty} \setminus (\tilde{X} \cup \infty)$ and are not shown in Figure \ref{fig:quotient}) are homeomorphic to $F$ and so they also have nonzero cohomology; moreover, it can be shown that they accumulate on essential quasicomponents. Thus one expects that $W^u_{\infty} \setminus (\tilde{X} \cup \infty)$ will generally have a non finitely generated cohomology and a very complicated topological structure.

\medskip

{\bf Standing assumptions.} \emph{From now on we shall always assume that $W^u_{\infty}$ has finitely generated \v{C}ech cohomology (with coefficients in the field $\mathbb{K}$) in every degree.} As discussed in Section \ref{sec:background}, this is guaranteed as soon as the phase space is a locally compact absolute neighbourhood retract; for instance, a topological manifold.
\medskip

The following remark will be useful later on to prove Corollary \ref{cor:minimal}:

\begin{remark}\label{rmk:Egeneral} As the reader will see, the results in this section do not make use of the specific nature of $W^u_{\infty}$ but, rather, only of the facts that: (i) it has finitely generated \v{C}ech cohomology, (ii) it has an attractor-repeller decomposition where the attractor is a singleton. In particular, if $W$ is any compact invariant subset of $W^u_{\infty}$ that contains $\tilde{X}$ and $\infty$ and has finitely generated cohomology, the forthcoming results are still true if we replace $W^u_{\infty}$, $U$ and $V$ by the more general $W$, $W \setminus \infty$ and $W \setminus \tilde{X}$. (Definitions, however, have to be treated with care: for example, not every quasicomponent of $U \cap V$ contained in $W$ is automatically a quasicomponent of $W \setminus (\tilde{X} \cup \infty)$).
\end{remark}

\subsection{A Mayer--Vietoris sequence} Recall that the argument hinged on a Mayer--Vietoris sequence \eqref{eq:mvietoris_mot0} which exploited the attractor--repeller structure of $W^u_{\infty}$. The relative form of the results of Section \ref{sec:series} suggest that it is best to consider the relative Mayer--Vietoris sequence \begin{equation} \label{eq:mayer} \xymatrix{ \ldots & \ar[l] \check{H}^{q+1}(W^u_{\infty},\tilde{X} \cup \infty) & \ar[l]_-{\Delta} \check{H}^q(U\cap V) & \ar[l] \check{H}^q(U,\tilde{X}) \oplus \check{H}^q(V,\infty) & \ar[l] \ldots} \end{equation} where (as before) $U := W^u_{\infty} \setminus \infty$ and $V := W^u_{\infty} \setminus \tilde{X}$. If the dynamics were generated by a flow, the relative cohomology groups $\check{H}^q(U,\tilde{X})$ and $\check{H}^q(V,\infty)$ would vanish and so $\Delta$ would be an isomorphism. The following theorem is the appropriate generalization to the discrete case that concerns us:

\begin{theorem} \label{teo:mvietoris} Suppose that $\check{H}^q(X)$ has finite dimension. Then the truncated Mayer--Vietoris sequence \[\xymatrix{0 & \ar[l] \check{H}^{q+1}(W^u_{\infty},\tilde{X} \cup \infty) & \ar[l]_-{\Delta} \check{H}^q(U\cap V) & \ar[l] \check{H}^q(U,X) \oplus \check{H}^q(V,\infty)}.\] is exact.
\end{theorem}

We need the following auxiliary lemma. The notation of its statement is the same as in Section \ref{sec:series}.

\begin{lemma} \label{lem:f_dimensional} Let coefficients be taken in a field $\mathbb{K}$. Then the action of $g^*$ in every finite dimensional, $g^*$--invariant subspace $S$ of $\check{H}^*(Z,A;\mathbb{K})$ is nilpotent. In particular, if $g$ is a homeomorphism then $g^*$ is an isomorphism and so $S = \{0\}$.
\end{lemma}
\begin{proof} Consider the restriction $g^*|_S \colon  S \to S$. Since $S$ is finite dimensional, $g^*|_S$ has a well defined, nonzero characteristic polynomial $P(x)$. Factor $P(x)$ as $P(x) = x^d P_0(x)$, where $P_0(x)$ has a nonzero independent term. Since coefficients are taken in a field there exists $Q(x) \in \mathbb{K}[[x]]$ such that $Q(x)P_0(x) = 1$, and multiplying through by $x^d$ we obtain $x^d = Q(x) P(x)$. Thus $(g^*)^d = Q(g^*) P(g^*)$. Evaluating this on any vector $u \in S$ yields $(g^*)^d(u) = Q(g^*) P(g^*)(u)$ and, since $P(g^*)(u) = 0$ by the Cayley--Hamilton theorem, it follows that $(g^*)^d(u) = 0$.
\end{proof}

\begin{proof}[Proof of Theorem \ref{teo:mvietoris}] Let $k$ and $\ell$ be, respectively, the inclusions of $(V,\infty)$ and $(U,\tilde{X})$ into $(W^u_{\infty},\tilde{X} \cup \infty)$. It will suffice to prove the following two assertions:
\begin{enumerate}
	\item[(i)] The induced homomorphism $k^*$ is zero in every degree.
	\item[(ii)] If $\check{H}^q(X)$ is finite dimensional, then ${\ell}^*$ is zero in degree $q+1$.
\end{enumerate}

The images of $\check{H}^*(W^u_{\infty},\tilde{X} \cup \infty)$ under $k^*$ and ${\ell}^*$ are evidently $f^*$--invariant, so by the previous lemma it suffices to show that they are finite dimensional (in every degree in the case of $k^*$ and in degree $q+1$ for ${\ell}^*$). Now:

(i) The inclusion $k \colon  (V,\infty) \subseteq (W^u_{\infty},\tilde{X} \cup \infty)$ factors through $(W^u_{\infty},\infty)$, so $k^*$ factors through $\check{H}^*(W^u_{\infty},\infty)$. Since the latter has finite dimension by Proposition \ref{prop:conley}, the same is true of the image of $k^*$.

(ii) Since $\check{H}^q(X)$ is finitely generated by assumption, the same holds true for $\check{H}^q(\tilde{X})$ by Proposition \ref{prop:Xtilde}. The exact sequence \[\xymatrix{\ldots & \ar[l] \check{H}^{q+1}(W^u_{\infty}) & \ar[l] \check{H}^{q+1}(W^u_{\infty},\tilde{X} \cup \infty) & \ar[l] \check{H}^q(\tilde{X} \cup \infty) & \ar[l] \ldots}\] then sandwiches $\check{H}^{q+1}(W^u_{\infty},\tilde{X} \cup \infty)$ between two finite dimensional vector spaces, so that it must be finite dimensional itself. Thus the image of ${\ell}^*$ is finite dimensional.
\end{proof}

\subsection{Cohomology in degree one} When $X$ has finitely many connected components $\check{H}^0(X)$ has finite dimension and according to Theorem \ref{teo:mvietoris} the Mayer--Vietoris sequence \eqref{eq:mayer} in degree $q = 0$ truncates to the following exact sequence: \begin{equation} \label{eq:mayer0} \xymatrix{ 0 & \ar[l] \check{H}^{1}(W^u_{\infty},\tilde{X} \cup \infty) & \ar[l]_-{\Delta} \check{H}^0(U\cap V) & \ar[l] \check{H}^0(U,\tilde{X}) \oplus \check{H}^0(V,\infty) & \ar[l] 0} \end{equation} Our description of $\check{H}^1(W^u_{\infty},\tilde{X} \cup \infty)$ will result from a detailed examination of this sequence.

We first recall some simple, general facts about $0$--dimensional Alexander--Spanier cohomology that will be used without further explanation in the sequel. It follows easily from the definitions (or see \cite[Theorem 5, p.309]{spanier1}) that $\check{H}^0(Z,Y;R)$ is the $R$--module of locally constant functions $\varphi \colon  Z \to R$ that vanish on $Y$. Locally constant means that every point $p \in Z$ has a neighbourhood $U$ over which $\varphi$ is constant or equivalently, as one can easily prove, that for any $c \in R$ the preimage $\varphi^{-1}(c)$ is a clopen set of $Z$. In particular, if $F$ is a quasicomponent of $Z$, any $\varphi \in \check{H}^0(Z,Y;R)$ is constant over $F$: picking $p \in F$ one has that $\varphi^{-1}(\varphi(p))$ is a clopen set of $Z$ that contains $p$; hence it must contain all of $F$ and so $\varphi(F) = \{\varphi(p)\}$. Thus it makes sense to evaluate elements from $\check{H}^0(Z,Y;R)$ over quasicomponents $F$ of $Z$ (or even over quasicomponents of subsets of $Z$). We will be somewhat careless with the notation and treat $\varphi(F)$ or $\varphi|_{F}$ as the single value $\varphi(p)$ rather than the singleton $\{\varphi(p)\}$.

\begin{proposition} \label{prop:qc} An element $\varphi \in \check{H}^0(U \cap V)$ belongs to the kernel of $\Delta$ if, and only if, $\varphi|_F = 0$ for every essential quasicomponent $F$.
\end{proposition}



\begin{proof}[Proof of Proposition \ref{prop:qc}] First suppose that $\varphi \in {\rm ker}\ \Delta$, so that it belongs to the image of the $\check{H}^0(U,\tilde{X}) \oplus \check{H}^0(V,\infty) \to \check{H}^0(U \cap V)$ in the Mayer--Vietoris sequence \eqref{eq:mayer0}. Recall that the latter is defined as the sum of inclusion induced homomorphisms; thus, there exist maps $\psi_U \in \check{H}^0(U,\tilde{X})$ and $\psi_V \in \check{H}^0(V,\infty)$ such that $\varphi = \psi_U|_{U \cap V} + \psi_V|_{U \cap V}$. To unclutter the notation in the sequel we shall omit the symbols denoting the restriction to $U \cap V$ and simply write $\varphi = \psi_U + \psi_V$, in the understanding that this equality holds only on $U \cap V$.

Let $F$ be an essential quasicomponent of $U \cap V$ and let $F'$ be the quasicomponent of $U$ that contains $F$. Notice that $\psi_U$ is constant on $F'$. Denoting by $\overline{F}$ the closure of $F$ in $U$ we have $F \subseteq \overline{F} \subseteq F'$ because $F'$ is closed in $U$. Since $\overline{F} \cap \tilde{X} \neq \emptyset$ because $F$ is essential, $F' \cap \tilde{X} \neq \emptyset$ too. Now $\psi_U$ is constant on $F'$ and vanishes on $\tilde{X}$, so it follows that $\psi_U(F') = 0$. Then $\psi_U(F) = 0$. An entirely analogous argument shows that $\psi_V(F) = 0$, and from the equality $\varphi = \psi_U + \psi_V$ on $F$ we conclude that $\varphi(F) = 0$.

Now suppose that $\varphi$ vanishes over every essential quasicomponent of $U \cap V$. We need to find $\psi_U \in \check{H}^0(U,\tilde{X})$ and $\psi_V \in \check{H}^0(V,\infty)$ such that $\varphi = \psi_U + \psi_V$. Define $\psi_U$ and $\psi_V$ as follows. First set $\psi_U|_{\tilde{X}} = 0$ and $\psi_V(\infty) = 0$. Now, for any point $p \in U \cap V$ let $F$ be the quasicomponent that contains $p$ and set \[\psi_U(p) := \left\{\begin{array}{cl} \varphi(p) & \text{ if $F$ does not reach $\tilde{X}$} \\ 0 & \text{ otherwise} \end{array} \right.\] and \[\psi_V(p) := \left\{\begin{array}{cl} \varphi(p) & \text{ if $F$ does not reach $\infty$} \\ 0 & \text{ otherwise} \end{array} \right.\]

According to this definition and to Proposition \ref{prop:reach}, for every quasicomponent $F$ of $U \cap V$ we have three possible cases (recall that $\varphi$, being locally constant, is constant over every quasicomponent $F$ of $U \cap V$, so it makes sense to speak of the single value $\varphi|_F$):
\begin{itemize}
	\item[(i)] If $F$ reaches $\infty$ but not $\tilde{X}$, then $\psi_U|_F = \varphi|_F$ and $\psi_V|_F = 0$.
	\item[(ii)] If $F$ reaches $\tilde{X}$ but not $\infty$, then $\psi_U|_F = 0$ and $\psi_V|_F = \varphi|_F$.
	\item[(iii)] If $F$ reaches both $\infty$ and $\tilde{X}$, then $\psi_U|_F = \psi_V|_F = 0$.
\end{itemize}

Clearly in cases (i) and (ii) we have $\varphi = \psi_U + \psi_V$. Moreover, since $\varphi$ vanishes on every quasicomponent that reaches both $\infty$ and $\tilde{X}$, we also have $\varphi = \psi_U + \psi_V$ in case (iii). By construction $\psi_U$ vanishes on $\tilde{X}$ and $\psi_V$ vanishes on $\infty$. Thus it only remains to show that $\psi_U$ and $\psi_V$ locally constant, so that they indeed define legitimate $0$--cocycles in their respective cohomology groups. We shall argue for $\psi_U$, because the argument for $\psi_V$ is completely analogous.

Let $O := \varphi^{-1}(0)$. This is a clopen set of $U \cap V$ that contains all the essential quasicomponents since $\varphi$ vanishes on them by assumption. Consider the sets $G_{\tilde{X}}$ and $G_{\infty}$ defined in the statement of Lemma \ref{lem:quasi4}. Write $U = (\tilde{X} \cup G_{\tilde{X}} \cup O) \uplus G_{\infty}$ and observe that, since $G_{\infty}$ is a clopen subset of $U$ by Lemma \ref{lem:quasi4}.(2), this is actually a partition of $U$ into two open sets. Notice that $\psi_U$ can be written as \[\psi_U = \left\{\begin{array}{cl} \varphi & \text{ on } G_{\infty} \\ 0 & \text{ on } \tilde{X} \cup G_{\tilde{X}} \cup O \end{array} \right.\] Since $\varphi$ is locally constant, we see from the above expression that $\psi_U$ is locally constant on the two open sets that partition $U$, and so $\psi_U$ is indeed locally constant on $U$.
\end{proof}

\begin{remark} The proof of Proposition \ref{prop:qc} makes no use of any assumptions about the phase space or the invariant set $\tilde{X}$ itself. Thus the resulting characterization of ${\rm ker}\ \Delta$ holds generally.
\end{remark}

The next proposition provides crucial information about the number of essential quasicomponents:

\begin{proposition} \label{prop:describeH1} The number of essential quasicomponents of $W^u_{\infty} \setminus (\tilde{X} \cup \infty)$ is bounded above by the dimension of $\check{H}^1(W^u_{\infty},\tilde{X} \cup \infty)$. In particular, if $X$ has finitely many connected components then the number of essential quasicomponents is also finite.
\end{proposition}

To prove the proposition we need the following simple lemma:

\begin{lemma} \label{lem:kronecker} Let $Q_1, \ldots, Q_m$ be different quasicomponents of a space $Z$. Then there exist locally constant maps $\varphi_1, \ldots, \varphi_m$ such that $\varphi_i|_{Q_j} = \delta_{ij}$ (the Kronecker delta) for every $1 \leq i,j \leq m$.
\end{lemma}
\begin{proof} It will be enough to show how to construct $\varphi_1$. For each $i = 2, \ldots, m$ there exists a clopen set $W_i$ that contains $Q_1$ and is disjoint from $Q_i$. The intersection $W := \bigcap_{i \geq 2} W_{1,i}$ is a clopen set that contains $Q_1$ and is disjoint from the remaining $Q_2, \ldots, Q_m$. Define $\varphi_1|_{W} := 1$ and $\varphi_1|_{W^c} := 0$. This map is locally constant by definition and evidently satisfies $\varphi_1(Q_j) = \delta_{1j}$ for every $j$.
\end{proof}

\begin{proof}[Proof of Proposition \ref{prop:describeH1}] Denote by $d$ be the dimension of $\check{H}^1(W^u_{\infty},\tilde{X} \cup \infty)$. Let $F_1, \ldots, F_m$ be essential quasicomponents of $U \cap V$ (not necessarily all of them). According to the previous lemma there exist elements $\varphi_1, \ldots, \varphi_m \in \check{H}^0(U \cap V)$ such that $\varphi_i(F_j) = \delta_{ij}$. Set $\xi_i := \Delta(\varphi_i) \in \check{H}^1(W^u_{\infty},\tilde{X} \cup \infty)$. It will be enough to prove that $\{\xi_i : 1 \leq i \leq m\}$ is a linearly independent set, for this implies that $m \leq d$ and therefore establishes the proposition. Consider, then, a null linear combination $\sum_i \lambda_i \xi_i = 0$. Writing this as $\sum_i \lambda_i \xi_i = \sum_i \lambda_i \Delta(\varphi_i) = \Delta(\sum_i \lambda_i \varphi_i)$ we see that $\sum_i \lambda_i \varphi_i \in {\rm ker}\ \Delta$. Thus according to Proposition \ref{prop:qc} the sum $\sum_i \lambda_i \varphi_i$ must vanish on every essential quasicomponent; in particular, it must vanish on each $F_j$. This gives $0 = \sum_i \lambda_i \varphi_i (F_j) = \sum_i \lambda_i \delta_{ij} = \lambda_j$, completing the proof.

Finally, when $X$ has finitely many connected components then $d$ is finite (this was argued in the proof of Theorem \ref{teo:mvietoris}) and so there are only finitely many essential quasicomponents.
\end{proof}

From now on we assume that $X$ has indeed finitely many connected components. This has two consequences: one the one hand, Theorem \ref{teo:mvietoris} shows that \eqref{eq:mayer0} is an exact sequence; on the other hand, $W^u_{\infty} \setminus (\tilde{X} \cup \infty)$ has finitely many essential quasicomponents $F_1, \ldots, F_m$ by Proposition \ref{prop:describeH1}. We now introduce an \emph{evaluation map} as follows. Let $\xi \in \check{H}^1(W^u_{\infty},\tilde{X} \cup \infty)$. According to \eqref{eq:mayer0} there exists $\varphi \in \check{H}^0(U \cap V)$ such that $\Delta(\varphi) = \xi$, where $\Delta$ is the coboundary homomorphism, and given any essential quasicomponent $F$ we define the \emph{evaluation of $\xi$ on $F$} to be the element $\xi(F) := \varphi|_F \in \mathbb{K}$. This definition is correct because $\xi(F)$ does not depend on the element $\varphi$ chosen to compute it: if $\varphi, \varphi' \in \check{H}^0(U \cap V)$ both satisfy $\Delta(\varphi) = \Delta(\varphi') = \xi$ then $\varphi - \varphi' \in {\rm ker}\ \Delta$ and by Proposition \ref{prop:qc} one has $(\varphi - \varphi')|_F = 0$; that is, $\varphi|_F = \varphi|_{F'}$. Finally, the \emph{evaluation map} is the result of collecting the evaluations over the various essential quasicomponents into a single map \[\begin{array}{ccccc} {\rm ev} & : & \check{H}^1(W^u_{\infty},\tilde{X} \cup \infty) & \to & \mathbb{K}^m \\ & & \xi & \longmapsto & (\xi(F_1),\ldots,\xi(F_m)) \end{array}\] It is straightforward to check that ${\rm ev}$ is a linear map. The main result of this section is the following:

\begin{theorem} \label{teo:main1} Let $X$ have finitely many connected components. Then the map ${\rm ev}$ is an isomorphism of vector spaces. Moreover, it conjugates $\tilde{f}^*$ with the permutation that $\tilde{f}^{-1}$ induces among the essential quasicomponents of $W^u_{\infty} \setminus (\tilde{X} \cup \infty)$ in the following sense: if \[\xi \longmapsto (\xi(F_1),\ldots,\xi(F_m))\] then \[\tilde{f}^*\xi \longmapsto (\xi(\tilde{f}F_1),\ldots,\xi(\tilde{f}F_m)).\]
\end{theorem}

Considering $F_1, \ldots, F_m$ as formal symbols, the $\mathbb{K}$--vector space $\langle F_1, \ldots, F_m \rangle$ having the $F_i$ as a basis can be identified with $\mathbb{K}^m$ in the natural way $\sum_i \lambda_i F_i \leftrightarrow (\lambda_1,\ldots,\lambda_m) \in \mathbb{K}^m$. Then Theorem \ref{teo:main1} says that $\check{H}^1(W^u_{\infty},\tilde{X} \cup \infty)$ can be thought of as $\langle F_1, \ldots, F_m \rangle$ and the action of $\tilde{f}^*$ on $\check{H}^1(W^u_{\infty},\tilde{X} \cup \infty)$ corresponds, under this identification, to the permutation of the $F_i$ by the map $\tilde{f}^{-1}$. It is in this form that we presented the above theorem in the Introduction as Theorem \ref{teo:main1_intro}.

\begin{proof}[Proof of Theorem \ref{teo:main1}] The map ${\rm ev}$ is injective. Indeed, suppose ${\rm ev}(\xi) = (0,\ldots,0)$. Pick $\varphi \in \check{H}^0(U \cap V)$ that gets sent to $\xi$ under the coboundary map of the Mayer--Vietoris sequence \eqref{eq:mayer0} and observe that $\varphi(F_i) = 0$ for every $i = 1, \ldots, m$. Since the $F_i$ exhaust all the essential quasicomponents, according to Proposition \ref{prop:qc} we have that $\varphi$ belongs to the kernel of $\Delta$ and so $\xi = \Delta(\varphi) = 0$.

In Proposition \ref{prop:describeH1} we established that $m \leq d$, where $d$ is the dimension of $\check{H}^1(W^u_{\infty},\tilde{X} \cup \infty)$. The previous paragraph implies that $d \leq m$. Therefore $d = m$ and ${\rm ev}$ is an isomorphism. Alternatively, one can check explicitly that ${\rm ev}$ is surjective: given $(\lambda_1, \ldots, \lambda_m) \in \mathbb{K}^m$, let $\varphi_i \in \check{H}^0(U \cap V)$ be as in Lemma \ref{lem:kronecker} and set $\varphi := \sum_i \lambda_i \varphi_i \in \check{H}^0(U \cap V)$. It is straightforward to check that $\xi := \Delta(\varphi)$ gets sent to $(\lambda_1,\ldots,\lambda_m)$ by the evaluation map.

Finally, let us analyze the effect of $\tilde{f}^*$ under the isomorphism ${\rm ev}$. Let $\xi$ be an element of $\check{H}^1(W^u_{\infty},\tilde{X} \cup \infty)$ and let $\varphi \in \check{H}^0(U \cap V)$ such that $\Delta(\varphi) = \xi$. The naturality of the Mayer--Vietoris sequence implies that $\tilde{f}^*\xi = \Delta(\tilde{f}^* \varphi)$. By definition $\tilde{f}^* \varphi = \varphi \circ \tilde{f}$. Thus $(\tilde{f}^*\xi)(F) = (\varphi \circ \tilde{f})(F) = \varphi (\tilde{f}F)$, as claimed.
\end{proof}

\subsection{Cohomology in higher degrees} Finally we venture into the higher dimensional cohomology of $(W^u_{\infty},\tilde{X})$ proving that the essential quasicomponents indeed deserve that name, in the sense that (under the appropriate circumstances) they carry all the cohomological information about $(W^u_{\infty},\tilde{X})$. That is the content of the following theorem:

\begin{theorem} \label{teo:essential} Let $X$ have finitely many connected components. Denote by $F_1, \ldots, F_m$ the (finitely many) essential quasicomponents of $U \cap V$ and let $\overline{F}_j$ be their closures in $W^u_{\infty}/\tilde{X}$ (notice that $\overline{F}_j  = F_j \cup \infty \cup [\tilde{X}]$). For every $q \geq 1$ such that $\check{H}^q(X)$ is finite dimensional there is an isomorphism \[\check{H}^{q+1}(W^u_{\infty}/\tilde{X}) = \oplus_j \check{H}^{q+1}(\overline{F}_j)\] induced by the inclusions $\overline{F}_j \subseteq W^u_{\infty}/\tilde{X}$.
\end{theorem}

To prove Theorem \ref{teo:essential} we need Propositions \ref{prop:inc_zero} and \ref{prop:inc_iso} below, of which the second is a strengthening of the first. Both involve an arbitrary neighbourhood $O \subseteq U \cap V$ of the union of the essential quasicomponents as well as its complement $G$ in $U \cap V$. These two sets $O$ and $G$ could eventually be empty in some cases (for instance, if there are no essential quasicomponents at all, which according to the discussion of Section \ref{sec:background} about the connectedness of $W^u_{\infty}$ corresponds to the case when $X$ is an asymptotically stable attractor) but, to avoid a cumbersome case distinction, we shall prove the propositions under the assumption that both $O$ and $G$ are nonempty. If one (or both) of them are empty, the propositions are still true, but trivially so.

\begin{proposition} \label{prop:inc_zero} Let $O \subseteq U \cap V$ be a clopen neighbourhood of the union of the essential quasicomponents and denote by $G$ its complement in $U \cap V$. Suppose that $\check{H}^q(X)$ has finite dimension for some $q \geq 0$. Then the inclusion $\overline{G} \subseteq W^u_{\infty}/\tilde{X}$ induces the zero homomorphism in cohomology in degree $q+1$.
\end{proposition}
\begin{proof} In the notation of Lemma \ref{lem:quasi4} one has $G = G_{\tilde{X}} \uplus G_{\infty}$, and so $\overline{G} = \overline{G}_{\tilde{X}} \cup \overline{G}_{\infty}$ (unlike Lemma \ref{lem:quasi4}, the closures here are taken in the quotient space $W^u_{\infty}/\tilde{X})$. Part (2) of the lemma implies that $\overline{G}_{\tilde{X}} = G_{\tilde{X}} \cup [\tilde{X}]$ (if nonempty) and $\overline{G}_{\infty} = G_{\infty} \cup \infty$ (if nonempty), so in particular $\overline{G} = \overline{G}_{\tilde{X}} \uplus \overline{G}_{\infty}$ and therefore $\check{H}^*(\overline{G}) \cong \check{H}^*(\overline{G}_{\tilde{X}}) \oplus \check{H}^*(\overline{G}_{\infty})$ because both $\overline{G}_{\tilde{X}}$ and $\overline{G}_{\infty}$ are open in $\overline{G}$. Thus to prove the proposition it suffices to show that both inclusions $\overline{G}_{\tilde{X}} \subseteq W^u_{\infty}/\tilde{X}$ and $\overline{G}_{\infty} \subseteq W^u_{\infty}/\tilde{X}$ induce the zero homomorphism in degree $q+1$.

The same argument used to establish assertion (ii) in the proof of Theorem \ref{teo:mvietoris} shows that the inclusion $(U,\tilde{X}) \subseteq (W^u_{\infty},\tilde{X})$ induces the zero homomorphism in cohomology in degree $q+1$, and the strong excision property of \v{C}ech cohomology then implies that the same is true of the inclusion $U/\tilde{X} \subseteq W^u_{\infty}/\tilde{X}$. Now, part (3) of Lemma \ref{lem:quasi4} guarantees that $\overline{G}_{\tilde{X}}$ is a subset of $U/\tilde{X}$, so the inclusion $\overline{G}_{\tilde{X}} \subseteq W^u_{\infty}/\tilde{X}$ factors through $U/\tilde{X}$ and consequently it also induces the zero homomorphism in cohomology in degree $q+1$, as was to be shown. A similar argument shows that the same holds (in every degree) for the inclusion $\overline{G}_{\infty} \subseteq W^u_{\infty}/\tilde{X}$.
\end{proof}

For the proof of the next proposition we need to make a simple observation. \label{pag:mvieto} Let a compact space $Z$ be written as the union of two compact spaces $Z_1$ and $Z_2$ and consider the corresponding Mayer--Vietoris sequence \begin{equation} \label{eq:mvietoris} \xymatrix{\ldots & \check{H}^q(Z_1 \cap Z_2) \ar[l] & \check{H}^q(Z_1) \oplus \check{H}^q(Z_2) \ar[l] & \check{H}^q(Z) \ar[l] & \ldots \ar[l]}\end{equation} The standard sufficient condition for this sequence to be exact is that the interiors of $Z_1$ and $Z_2$ cover $Z$. However, in the present case the above sequence is always exact. To prove this, let $(U_k)$ and $(V_k)$ be decreasing neighbourhood bases of $Z_1$ and $Z_2$ respectively. For each $k$ there is an exact Mayer--Vietoris sequence for the decomposition $Z = U_k \cup V_k$; furthermore, all these sequences are connected by homomorphisms induced by the inclusions $U_{k+1} \subseteq U_k$ and $V_{k+1} \subseteq V_k$ (shown as unlabeled vertical arrows in the diagram): \[\xymatrix{\ldots & \check{H}^q(U_k \cap V_k) \ar[l] \ar[d] & \check{H}^q(U_k) \oplus \check{H}^q(V_k) \ar[l] \ar[d] & \check{H}^q(Z) \ar[l] \ar[d]^-{{\rm id}}& \ldots \ar[l] \\ \ldots & \check{H}^q(U_{k+1} \cap V_{k+1}) \ar[l] & \check{H}^q(U_{k+1}) \oplus \check{H}^q(V_{k+1}) \ar[l] & \check{H}^q(Z) \ar[l] & \ldots \ar[l]}\]  The direct limit (in $k$) of this sequence of Mayer--Vietoris sequences is, because of the continuity property of \v{C}ech cohomology, nothing but the exact sequence \eqref{eq:mvietoris}. Since the direct limit of a sequence of exact sequences is again an exact sequence, the claim follows.

\begin{proposition} \label{prop:inc_iso} Let $O \subseteq U \cap V$ be a clopen neighbourhood of the union of the essential quasicomponents and denote by $G$ its complement in $U \cap V$. Suppose that $\check{H}^q(X)$ has finite dimension for some $q \geq 0$. Then the inclusion $\overline{O} \subseteq W^u_{\infty}/\tilde{X}$ induces an isomorphism in cohomology in degree $q+1$.
\end{proposition}
\begin{proof} Write $W^u_{\infty}/\tilde{X}$ as the union of the two compact sets $\overline{O}$ and $\overline{G}$. According to the discussion preceding this proposition we may write an exact Mayer--Vietoris sequence \[\xymatrix{ \ldots & \ar[l] \check{H}^{q+1}(\overline{O}) \oplus \check{H}^{q+1}(\overline{G}) & \ar[l] \check{H}^{q+1}(W^u_{\infty}/\tilde{X}) & \ar[l] \check{H}^q(\overline{O} \cap \overline{G}) & \ar[l]}\] The intersection $\overline{O} \cap \overline{G}$ consists of, at most, the pair of points $\infty$ and $[\tilde{X}]$, so for $q \geq 1$ the above sequence truncates to \begin{equation} \label{eq:mvietoris0} \xymatrix{ 0 & \ar[l] \check{H}^{q+1}(\overline{O}) \oplus \check{H}^{q+1}(\overline{G}) & \ar[l] \check{H}^{q+1}(W^u_{\infty}/\tilde{X}) & \ar[l] 0}\end{equation} and, since the inclusion induced map $\check{H}^{q+1}(W^u_{\infty}/\tilde{X}) \to \check{H}^{q+1}(\overline{G})$ is zero according to Proposition \ref{prop:inc_zero}, this implies that $\check{H}^{q+1}(W^u_{\infty}/\tilde{X}) \cong \check{H}^{q+1}(\overline{O})$ as was to be shown.

It remains to analyze the case $q = 0$, which corresponds to the end of the Mayer--Vietoris sequence:
\begin{equation} \label{eq:mvietoris00}
\xymatrix{ \ldots & \ar[l] \check{H}^1(W^u_{\infty}/\tilde{X}) & \ar[l] \check{H}^0(\overline{O} \cap \overline{G}) & \ar[l] \check{H}^{0}(\overline{O}) \oplus \check{H}^{0}(\overline{G}) & \ar[l] \check{H}^{0}(W^u_{\infty}/\tilde{X}) & \ar[l] 0}
\end{equation}

Observe once again that $\overline{O} \cap \overline{G}$ consists of, at most, the two points $\infty$ and $[\tilde{X}]$. Recalling from Lemma \ref{lem:quasi4} that $\overline{G}$ is partitioned in two clopen sets, $\overline{G}_{\tilde{X}}$ and $\overline{G}_{\infty}$, it is clear that a necessary condition for $\overline{O} \cap \overline{G}$ to contain $\infty$ is that $\overline{G}_{\infty} \neq \emptyset$, and similarly for $[\tilde{X}]$. Thus the inclusion induced map $\check{H}^0(\overline{G}_{\tilde{X}}) \oplus \check{H}^0(\overline{G}_{\infty}) \to \check{H}^0(\overline{O} \cap \overline{G})$ is surjective. However, $\check{H}^0(\overline{G}_{\tilde{X}}) \oplus \check{H}^0(\overline{G}_{\infty}) = \check{H}^0(\overline{G})$ because both $\overline{G}_{\tilde{X}}$ and $\overline{G}_{\infty}$ provide a partition of $\overline{G}$ into clopen sets, and so it follows that the inclusion induced map $\check{H}^0(\overline{G}) \to \check{H}^0(\overline{O} \cap \overline{G})$ is surjective. Thus the Mayer--Vietoris sequence \eqref{eq:mvietoris00} splits, proving that the sequence \eqref{eq:mvietoris0} is also exact for $q = 0$. Reasoning as above we prove the proposition for $q = 0$ too.
\end{proof}

The proof of the above proposition also shows that the cohomology group $\check{H}^{q+1}(\overline{G})$ vanishes (for $q \geq 0$), which is the formal counterpart of the intuitive idea that quasicomponents that are not essential do not contribute to the cohomology of $W^u_{\infty}/\tilde{X}$.

\begin{proof}[Proof of Theorem \ref{teo:essential}] Since $X$ has finitely many connected components, $U \cap V$ has finitely many quasicomponents $F_1, \ldots, F_d$ as shown in Proposition \ref{prop:describeH1}. Denote by $E$ their union.
\smallskip

{\it Claim.} Let $p \in U \cap V$ be a point not in $E$. Then there exists a clopen set $O \subseteq U \cap V$ that contains $E$ but not $p$.

{\it Proof of claim.} Let $F$ be the quasicomponent of $U \cap V$ that contains $p$. Since $p \not\in E$, clearly $F$ is different from the $F_i$. Let $G$ be a clopen set that contains $F$ and is disjoint from the $F_i$. Letting $O$ be the complement of $G$ in $U \cap V$ proves the claim.
\smallskip

Let $\mathcal{O}$ be the set of clopen subsets of $U \cap V$ that contain $E$. Clearly $\overline{O} = [\tilde{X}] \cup O \cup \infty$ for each $O \in \mathcal{O}$ and, by the claim above, $\bigcap_{O \in \mathcal{O}} \overline{O} = \overline{E}$. Thus considering $\mathcal{O}$ as a directed set (by the inclusion relation) one has $\varprojlim_{O \in \mathcal{O}} \overline{O} = \overline{E}$ and therefore by the continuity property of \v{C}ech cohomology $\check{H}^*(\overline{E}) \cong \varinjlim_{O \in \mathcal{O}} \check{H}^*(\overline{O})$. The theorem then follows from Proposition \ref{prop:inc_iso}.
\end{proof}

Theorem \ref{teo:essential} states that the cohomological information contained in the unstable manifold is carried by the essential quasicomponents. To complete the picture, we show in Corollary \ref{cor:minimal} that this information is actually contained in connected sets that we could legitimately call \emph{branches}. The subtle points here is that a quasicomponent is not always connected. We first need the following remark:

\begin{remark}\label{rem:connectedness} Let $E$ be the union of the essential quasicomponents of $W^u_{\infty} \setminus (\tilde{X} \cup \infty)$. Then $\overline{E}$ is connected (unless it is empty).
\end{remark}

Notice that this fact would be immediate if the essential quasicomponents $F_j$ were connected; however, we only know that each $F_j$ is a union of connected components of $W^u_{\infty} \setminus (\tilde{X} \cup \infty)$.

\begin{proof}[Proof of Remark \ref{rem:connectedness}] As before, we work in the quotient space $W^u_{\infty}/\tilde{X}$. By Proposition \ref{prop:reach} every component of $U \cap V$ is adherent either to $[\tilde{X}]$ or to $\infty$, and this implies that $\overline{E}$ has at most two connected components: $A_{\infty}$ containing $\infty$ and $A_{\tilde{X}}$ containing $[\tilde{X}]$. Suppose $A_{\infty}$ and $A_{\tilde{X}}$ are different components. Then $\{A_{\infty}, A_X\}$ is a partition of $\overline{E}$ in $\tilde{f}$--invariant clopen sets so the $\omega$--limit of any point of $A_{\tilde{X}}$ is contained in $A_{\tilde{X}}$ and the $\alpha$--limit of any point in $A_{\infty}$ is contained in $A_{\infty}$. This contradicts the attractor--repeller decomposition that $\overline{E}$ inherits from $W^u_{\infty}/\tilde{X}$ unless $E$ is empty.
\end{proof}

Now we can prove Corollary \ref{cor:minimal}. For brevity we will require that $X$ has finitely generated \v{C}ech cohomology in all degrees, but the corollary can also be stated separately for each degree, as in Theorem \ref{teo:essential}.

\begin{corollary} \label{cor:minimal} Assume that $X$ has finitely generated \v{C}ech cohomology in all degrees. As before, let $F_j$ be the essential quasicomponents of $W^u_{\infty} \setminus (\tilde{X} \cup \infty)$ and denote by $E$ their union. There exists an invariant continuum \(C \subseteq \overline{E}\) such that the inclusions $C \subseteq \overline{E} \subseteq W^u_{\infty}/\tilde{X}$ induce isomorphisms in cohomology in all dimensions and \(C \setminus (\tilde{X} \cup \infty)\) consists of finitely many connected sets \(G_i\).
\end{corollary}

\begin{proof} We work in the quotient space $W^u_{\infty}/\tilde{X}$. Consider the family $\mathcal{C}$ of all invariant continua \(C \subseteq W^u_{\infty}/\tilde{X}\) that contain \([\tilde{X}]\) and \(\infty\) and such that the inclusion \(C \subseteq W^u_{\infty}/\tilde{X}\) induces isomorphisms in cohomology in all degrees (in particular, $C$ has finitely generated cohomology because $W^u_{\infty}/\tilde{X}$ does, by assumption). By Theorem \ref{teo:essential} and Remark \ref{rem:connectedness} the set $\overline{E}$ belongs to $\mathcal{C}$. Using the continuity property of \v{C}ech cohomology it is easy to that the intersection of every nested sequence of elements from $\mathcal{C}$ also belongs to $\mathcal{C}$, and it follows from Zorn's lemma that $\mathcal{C}$ has a minimal element (with respect to the inclusion). We claim that any such minimal set $C$ satisfies the corollary. To see this, recall from Remark \ref{rmk:Egeneral} that all the results in this section hold for any compact invariant subset $W$ of $W^u_{\infty}$ that contains $\tilde{X}$ and $\infty$ and has finitely generated \v{C}ech cohomology. In particular, Proposition \ref{prop:describeH1} and Theorem \ref{teo:essential} applied to $W := C$ show that \(C \setminus ([\tilde{X}] \cup \infty)\) has finitely many essential quasicomponents $G_1, \ldots, G_n$ and their union $G$ satisfies that the inclusion \( \overline{G} \subseteq C\) induces isomorphisms in cohomology. It follows that $\overline{G}$ belongs to the family $\mathcal{C}$ (notice that $\overline{G}$ is a continuum by Remark \ref{rem:connectedness}) so, by the minimality of $C$, the equality $C = \overline{G}$ must hold. In particular \(C \setminus ([\tilde{X}] \cup \infty)\) has finitely many quasicomponents (the $G_i$), and then a result from general topology (\cite[Corollary 9.3, p. 19]{wilder1}) guarantees that each quasicomponent is actually a component; that is, each of the $G_i$ is connected.
\end{proof}

\subsection*{Addendum} As promised in the Introduction, we now devote a few lines to discuss the general case when $X$ is not assumed to have finitely many connected components. A quick revision of the proof of Theorem \ref{teo:mvietoris} shows that now there is no guarantee that the Mayer--Vietoris sequence \eqref{eq:mayer} can be truncated to an exact sequence, even in the lowest degree $q = 0$. However, even though we cannot describe the full group $\check{H}^1(W^u_{\infty},\tilde{X} \cup \infty)$, we can still describe the eigenvectors of $\tilde{f}^*$:

\begin{proposition} Let $u \in \check{H}^1(W^u_{\infty},\tilde{X} \cup \infty)$ be an eigenvector of $\tilde{f}^*$ associated to an eigenvalue $\lambda$. Then $u$ has the form $u = \Delta(\varphi)$, where $\varphi \in \check{H}^0(U \cap V)$ satisfies $\varphi(\tilde{f}F) = \lambda \varphi(F)$ for every essential quasicomponent $F$.
\end{proposition}
\begin{proof} If nonzero, the image of $u$ in $\check{H}^1(U,\tilde{X})$ under the homomorphism induced by the inclusion map $(U,\tilde{X}) \subseteq (W^u_{\infty},\tilde{X} \cup \infty)$ would be an eigenvector for $\tilde{f}^*$ (or more precisely for the homomorphism induced by the restriction $\tilde{f}|_{(U,\tilde{X})}$). Since $\tilde{f}$ is a homeomorphism, $u$ would also be an eigenvector for $(\tilde{f}^{-1})^*$ (with eigenvalue $\lambda^{-1}$) and this contradicts Corollary \ref{cor:no_eigen} (with $Z = U$, $g = \tilde{f}^{-1}$, $A = \tilde{X}$). A similar observation applies to the inclusion $(V,\infty) \subseteq (W^u_{\infty},\tilde{X} \cup \infty)$. Thus $u$ is sent to zero by the leftmost arrow of the Mayer--Vietoris sequence \[\xymatrix{\ldots & \ar[l] \check{H}^1(U,\tilde{X}) \oplus \check{H}^1(V,\infty) & \check{H}^1(W^u_{\infty},\tilde{X} \cup \infty) \ar[l] & \check{H}^0(U \cap V) \ar[l]_-{\Delta} & \ldots \ar[l]}\] and so there exists $v \in \check{H}^0(U \cap V)$ such that $u = \Delta(v)$. The eigenvector condition $\tilde{f}^*(u) - \lambda u = 0$ translates into $\tilde{f}^*(v) - \lambda v \in \ker\ \Delta$.

Represent $v$ by a locally constant function $\varphi \colon  U \cap V \to \mathbb{C}$. Recalling from Proposition \ref{prop:qc} that ${\rm ker}\ \Delta$ consists of those locally constant functions that vanish on the essential quasicomponents of $U \cap V$, the condition $\tilde{f}^*(v) - \lambda v \in \ker\ \Delta$ translates into $\varphi(\tilde{f}F) = \lambda \varphi(F)$ for every essential quasicomponent of $U \cap V$, as claimed.
\end{proof}

It is clear from the description of $\check{H}^1(W^u_{\infty},\tilde{X} \cup \infty)$ provided by Theorem \ref{teo:main1_intro} for $X$ having finitely many connected components that the eigenvalues of $\tilde{f}^*$ must be roots of unity. This property is still true even when $X$ has infinitely many connected components:

\begin{proposition} \label{prop:rootsofunity} Every eigenvalue of $\tilde{f}^*$ on $\check{H}^1(W^u_{\infty},\tilde{X} \cup \infty)$ is a root of unity.
\end{proposition}

To prove this result we need a lemma. Let us say that a sequence of quasicomponents $(F_n)$ of $W^u_{\infty} \setminus (\tilde{X} \cup \infty)$ \emph{converges} to another quasicomponent $F$ if there exists a sequence of points $p_n \in F_n$ that converges to a point $p \in F$. With this definition, the following lemma should seem reasonable:

\begin{lemma} \label{lem:converge_essential} Let $F$ be a quasicomponent of $W^u_{\infty} \setminus (\tilde{X} \cup \infty)$. Then the sequence of forward iterates $(\tilde{f}^nF)_{n \geq 0}$ or the sequence of backward iterates $(\tilde{f}^nF)_{n \leq 0}$ has a subsequence that converges to an essential quasicomponent.
\end{lemma}

The proof of the lemma is postponed to Appendix \ref{app:two}.

\begin{proof}[Proof of Proposition \ref{prop:rootsofunity}] Let $\lambda$ be an eigenvalue of $\tilde{f}^*$ and let $u$ be an eigenvector associated to $\lambda$. According to the previous proposition we may write $u = \Delta(\varphi)$ for some locally constant $\varphi \colon U \cap V \to \mathbb{K}$ that satisfies $\varphi(\tilde{f}F) = \lambda \varphi(F)$ for every essential quasicomponent $F$. Notice that $\varphi$ cannot belong to $\ker\ \Delta$, since otherwise the eigenvector $u$ would be zero; thus, by Proposition \ref{prop:qc} there exists an essential quasicomponent $F$ such that $\varphi(F) \neq 0$. Consider its trajectory $(\tilde{f}^nF)_{n \in \mathbb{Z}}$. By the lemma above this has a subsequence $(\tilde{f}^{n_k}F)$ (where $n_k \longrightarrow +\infty$ or $n_k \longrightarrow -\infty$) that converges to an essential quasicomponent $F_0$. Since $\varphi$ is locally constant, it is constant on a neighbourhood of $F_0$ and therefore $\varphi(\tilde{f}^{n_k}F)$ is constant for big enough $k$. However $\varphi(\tilde{f}^{n_k}F) = \lambda^{n_k} \varphi(F)$ and $\varphi(F)$ and $\lambda$ are nonzero, so the only way for this sequence to be eventually constant is that $\lambda^{n_{k+1} - n_k} = 1$ for large $k$. Hence $\lambda$ is a root of unity.
\end{proof}

Now we can prove Theorem \ref{teo:components} from the Introduction. Recall that this result stated that, if $h^1(f,X)$ has a nonzero eigenvalue $\lambda \in \mathbb{C}$ that is not a root of unity, then at least one of the following alternatives must hold:
\begin{itemize}
	\item[(i)] $X$ has infinitely many connected components.
	\item[(ii)] $f^* \colon  \check{H}^1(X;\mathbb{C}) \to \check{H}^1(X;\mathbb{C})$ has $\lambda$ as an eigenvalue.
\end{itemize}

\begin{proof}[Proof of Theorem \ref{teo:components}] (We omit the coefficients $\mathbb{C}$ from the notation). Let $u_0 \in \check{H}^1(W^u_{\infty},\infty)$ be an eigenvector of eigenvalue $\lambda$. If nonzero, the image of $u_0$ in $\check{H}^1(\tilde{X})$ under the inclusion induced map would be an eigenvector of eigenvalue $\lambda$ and by Proposition \ref{prop:Xtilde} alternative (ii) would hold. Thus let us assume that this is not the case and prove that then (i) holds.

Consider the exact sequence for the triple $(W^u_{\infty},\tilde{X} \cup \infty,\infty)$, the polynomial $P(x) := x - \lambda$, and the following commutative diagram: \[\xymatrix{\check{H}^1(\tilde{X}) \ar[d]_{P(\tilde{f}^*)}& \check{H}^1(W^u_{\infty},\infty) \ni u_0 \ar[l] \ar[d]_{P(\tilde{f}^*)} & \check{H}^1(W^u_{\infty},\tilde{X} \cup \infty) \ni v_0 \ar[l] \ar[d]_{P(\tilde{f}^*)} & \check{H}^0(\tilde{X}) \ar[l] \ar[d]_{P(\tilde{f}^*)}\\ \check{H}^1(\tilde{X}) & \check{H}^1(W^u_{\infty},\infty) \ar[l] & \check{H}^1(W^u_{\infty},\tilde{X} \cup \infty) \ar[l] & \check{H}^0(\tilde{X}) \ar[l]}\] (As usual, we have identified $\check{H}^*(\tilde{X} \cup \infty,\infty) = \check{H}^*(\tilde{X})$.) We denote by $G$ the image of $\check{H}^0(\tilde{X})$ in $\check{H}^1(W^u_{\infty},\tilde{X} \cup \infty)$ under the coboundary homomorphism. Observe that the commutativity of the rightmost square of the diagram shows that $P(\tilde{f}^*)(G) \subseteq G$, so that $P(\tilde{f}^*)$ restricts to a map $P(\tilde{f}^*)|_G \colon  G \to G$.

As argued in the previous paragraph, we assume that $u_0$ is annihilated by the topmost left arrow. Thus $u_0$ is the image of some $v_0 \in \check{H}^1(W^u_{\infty},\tilde{X} \cup \infty)$ under the corresponding arrow. Since $u_0$ is an eigenvector of eigenvalue $\lambda$ and therefore $P(\tilde{f}^*)(u_0) = 0$ (where $P(\tilde{f}^*)$ is considered as an endomorphism of $\check{H}^1(W^u_{\infty},\infty)$), the commutativity of the middle square of the above diagram implies that $P(\tilde{f}^*)(v_0) \in G$, where now $P(\tilde{f}^*)$ is considered as an endomorphism of $\check{H}^1(W^u_{\infty},\tilde{X} \cup \infty)$).

By assumption $\lambda$ is not a root of unity. It follows from Proposition \ref{prop:rootsofunity} that the endomorphism $P(\tilde{f}^*)$ of $\check{H}^1(W^u_{\infty},\tilde{X} \cup \infty)$ is injective, and so is its restriction $P(\tilde{f}^*)|_G \colon  G \to G$. Now suppose $X$ had only finitely many components. Then by Proposition \ref{prop:Xtilde} the same would be true of $\tilde{X}$, so $G$ would be finite dimensional and therefore $P(\tilde{f}^*)|_G \colon  G \to G$ would actually be surjective. In particular, since $P(\tilde{f}^*)(v_0) \in G$ and $P(\tilde{f}^*)$ is injective, we would have $v_0 \in G$ but this is impossible since then the eigenvector $u_0$ would be zero. This contradiction shows that $X$ must have infinitely many connected components; that is, alternative (i) must hold.
\end{proof}

We remark that (ii) cannot hold if the phase space is $\mathbb{R}^2$ and $f$ is a homeomorphism. Indeed, in that case by Alexander duality the homomorphism $f^* \colon  \check{H}^1(X) \to \check{H}^1(X)$ is conjugate to $f_* \colon  H_0(\mathbb{R}^2 - X) \to H_0(\mathbb{R}^2 - X)$, and it is clear that the eigenvalues of the latter must be roots of unity since $H_0(\mathbb{R}^2 - X)$ is generated by the connected components of $\mathbb{R}^2 - X$ and $f$ permutes them.

The G--horseshoe is a variant of the well known Smale's horseshoe, described in Subsection \ref{subsec:horseshoe}, that provides a simple illustration of this theorem. In the notation of Figure \ref{fig:Ghorse}, $X$ is the maximal invariant subset of the square $N$ under the homeomorphism $f$ of $\mathbb{R}^2$ whose action on $N$ is illustrated in the figure. The set $X$ is exactly the same as $\Lambda$, the invariant set in Smale's horseshoe, the difference between this example and that of Subsection \ref{subsec:horseshoe} purely resides in the dynamics.

\begin{figure}[h!]
\null\hfill
\subfigure{
\begin{pspicture}(0,0)(4.6,9.2)
	\rput[bl](0,2.1){\scalebox{0.75}{\includegraphics{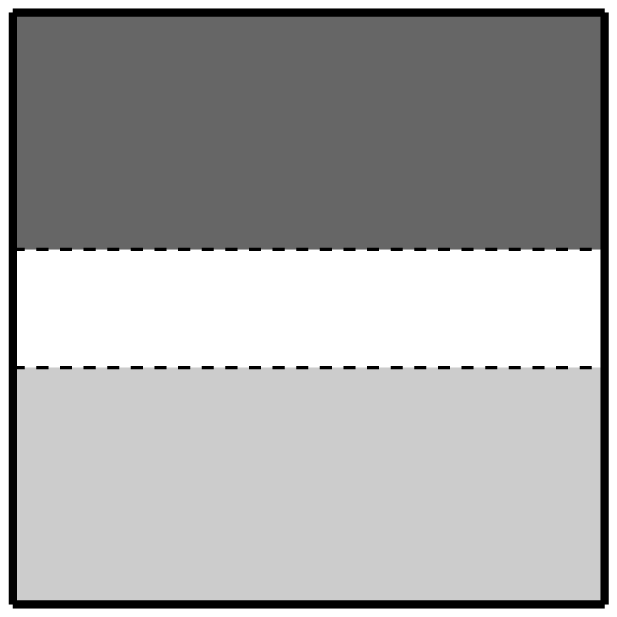}}}
	\rput[bl](3,1.6){$N$}
\end{pspicture}
}
\hfill
\subfigure{
\begin{pspicture}(0,0)(6,8.4)
	\rput[bl](0,0){\scalebox{0.75}{\includegraphics{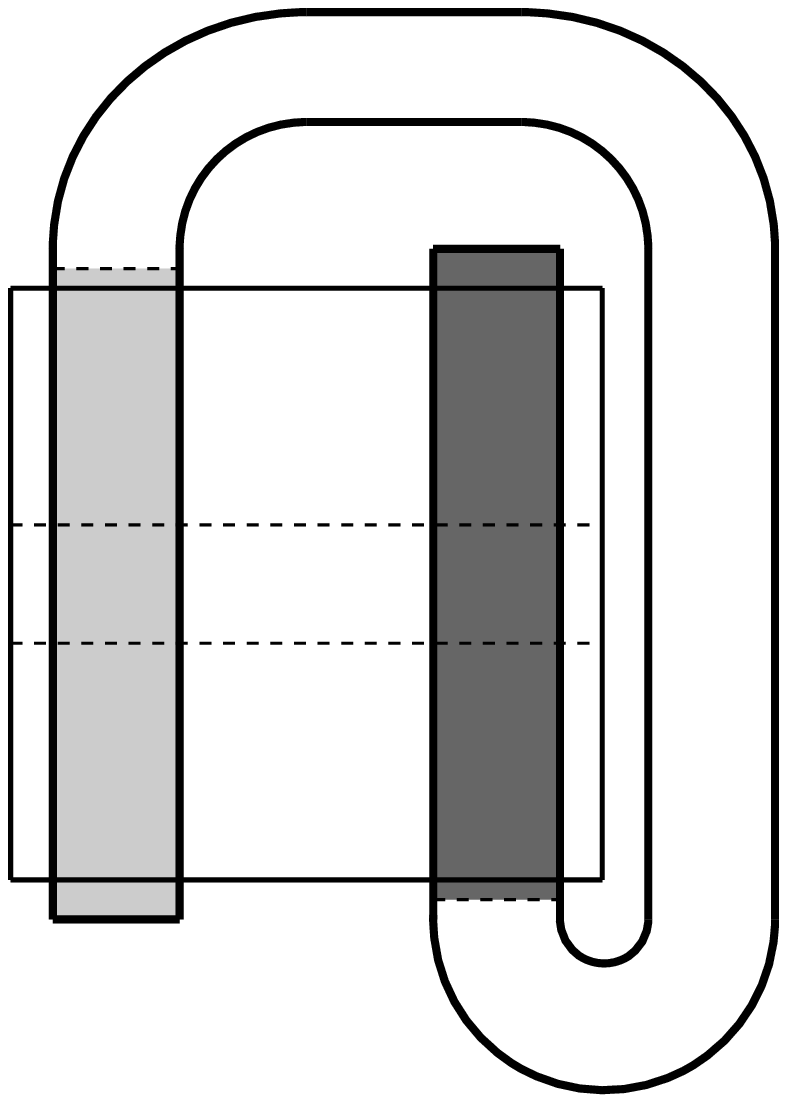}}}
	\rput[tr](1,1){$f(N)$}
\end{pspicture}
}
\hfill\null
\caption{The G--horseshoe \label{fig:Ghorse}}
\end{figure}

By letting $L$ consist of three horizontal stripes of appropriate thickness, two of them at the top and bottom of $N$ and the third at mid height, one can easily construct an index pair $(N,L)$ for $X$. Call these three stripes $L_{\rm top}, L_{\rm bot}$ and $L_{\rm mid}$. Let $a$ be an oriented vertical path that begins in $L_{\rm bot}$ and ends in $L_{\rm mid}$; also, let $b$ join $L_{\rm mid}$ and $L_{\rm top}$ in a similar fashion. Then the homology $H_1(N,L;\mathbb{C})$ is easily seen to have $\{a,b\}$ as a basis, and the action of $f_*$ on this basis is $f_*(a) = a+b$ and $f_*(b) = a+b$. This has eigenvalues $0$ and $2$. The computation in cohomology is dual to this, so it has the same eigenvalues. It follows that $h^1(f,X)$ has the nonzero eigenvalue $\lambda = 2$, and Theorem \ref{teo:components} then implies that $X$ must have infinitely many connected components. This can be checked explicitly by constructing the G--horseshoe and noting that it is a Cantor set.

\appendix
\section{\label{app:one}}

In this appendix we make use of some results from infinite dimensional topology to prove the following proposition:

\begin{proposition} \label{prop:att_countable} Let $K$ be an attractor for a homeomorphism $f$ of locally compact, metrizable absolute neighbourhood retract $M$. Then $K$ has countable \v{C}ech cohomology with integer coefficients.
\end{proposition}
\begin{proof} Denote by $Q$ the Hilbert cube and consider the homeomorphism of $M \times Q \times [0,1)$ defined as $\hat{f} := f \times {\rm id}_Q \times (t \longmapsto t^2)$; that is, $\hat{f}(p,q,t) := (f(p),q,t^2)$. Clearly $\hat{f}$ is a homeomorphism that has $\hat{K} := K \times Q \times \{0\}$ as an attractor with basin of attraction $\hat{\mathcal{A}} = \mathcal{A} \times Q \times [0,1)$, where $\mathcal{A}$ is the basin of attraction of $K$ in $M$.
\smallskip

{\it Claim.} $\hat{K}$ has a compact neighbourhood $P \subseteq \hat{\mathcal{A}}$ that is a compact absolute neighbourhood retract (henceforth, ANR).
\smallskip

{\it Proof of claim.} First observe that $\mathcal{A}$ is separable. To see this, pick a compact neighbourhood $F$ of $K$ in $\mathcal{A}$. Since $F$ is compact and metrizable (being a subset of the metrizable phase space $M$), it is separable. Then $\mathcal{A} = \bigcup_{n \leq 0} f^nF$ is a countable union of separable spaces, therefore also separable.

Now we are entitled to apply the ANR theorem of Edwards \cite[Chapter XIV]{chapman2} to conclude that $\mathcal{A} \times Q$ is a $Q$--manifold and furthermore $\mathcal{A} \times Q \times [0,1) = \hat{\mathcal{A}}$ is triangulable \cite[Corollary 1, p. 330]{chapman1}. This means that there exist a countable, locally finite simplicial complex $S$ and a homeomorphism $h$ such that $\hat{\mathcal{A}} \cong |S| \times Q$ via $h$.

Since $h(\tilde{K})$ is a compact subset of $|S| \times Q$, its projection ${\rm pr}_{|S|}\ h(\tilde{K})$ onto the $|S|$ factor is a compact set, so it has a polyhedral compact neighbourhood $N$ in $|S|$. Then $N \times Q$ is a compact ANR neighbourhood of $h(K)$ in $|S| \times Q$, and its preimage $P$ under $h$ is a compact ANR neighbourhood of $\hat{\tilde{K}}$ in $\hat{\mathcal{A}}$. This concludes the proof of the claim. $_{\blacksquare}$
\smallskip

We continue with the proof of the proposition. Let $P$ be as described in the claim above. Let $g$ denote a power of $\hat{f}$ big enough so that $g(P) \subseteq P$. This guarantees that the sets $g^n(P)$ form a decreasing nested sequence of compact sets, and clearly their intersection is precisely $\hat{K}$. Thus the cohomology of $\hat{K}$ is the direct limit of the direct sequence $H^*(P) \to H^*(g(P)) \to H^*(g^2(P)) \to \ldots$ where the arrows denote the inclusion induced homomorphisms. Being a compact ANR, the space $P$ is homotopy dominated by a finite simplicial complex $K$ (\cite[Corollary 6.2, p. 139]{hu}) and therefore its cohomology groups are quotients of those of $K$, which are finitely generated. Thus the cohomology groups of $P$ are also finitely generated. The direct limit of the sequence above is therefore a countable set, proving the proposition.
\end{proof}

\section{\label{app:two}}

In this Appendix we prove Lemmas \ref{lem:quasi4} and \ref{lem:converge_essential}, which are concerned with the quasicomponent structure of $W^u_{\infty} \setminus (\tilde{X} \cup \infty)$ and were used earlier on in Sections \ref{sec:background} and \ref{sec:describe1}. To simplify language, and when there is no risk of confusion, throughout this appendix we will always refer to the quasicomponents of $W^u_{\infty}\setminus (\tilde{X} \cup \infty)$ simply as ``quasicomponents'', and similarly for its components.

\subsection{Auxiliary results} We begin by establishing some auxiliary results concerning the convergence of quasicomponents. Recall that a sequence of quasicomponents $F_n$ converges to another quasicomponent $F$ if there exists a sequence of points $p_n \in F_n$ that converges to a point $p \in F$. We denote this situation by writing $(F_n) \longrightarrow F$. There follow some elementary facts related to this definition:

\begin{proposition} \label{prop:convergence} Let $(F_n)$ be a sequence of quasicomponents of $W^u_{\infty} \setminus (\tilde{X} \cup \infty)$.
\begin{enumerate}
	\item If $(F_n)$ converges to $F$, the same is true of every subsequence of $(F_n)$.
	\item If $(F_n)$ converges to both $F$ and $F'$ then $F = F'$.
\end{enumerate}
\end{proposition}
\begin{proof} Part (1) is trivial. As for (2), begin by picking sequences of points $p_n, q_n \in F_n$ converging to $p \in F$ and $q \in F'$ respectively. Let $O$ be any clopen subset of $W^u_{\infty} \setminus (\tilde{X} \cup \infty)$ such that $p \in O$. Since $p_n \to p$, there exists $n_0$ such that for $n \geq n_0$ one has $p_n \in O$. It follows that $F_n \subseteq O$, since a quasicomponent is either contained in or disjoint from any given clopen set. Consequently the sequence $(q_n)_{n \geq n_0}$ is contained in $O$, which is closed, so its limit $q$ also belongs to $O$. We have thus shown that every clopen set that contains $p$ also contains $q$, so both points belong to the same quasicomponent of $W^u_{\infty} \setminus (\tilde{X} \cup \infty)$.
\end{proof}

\begin{proposition} \label{prop:quasi2} Let $(F_n) \longrightarrow F$ be a convergent sequence of quasicomponents. Assume that $F$ does not reach $\infty$, so in particular it reaches $\tilde{X}$. Then for large enough $n$ each $F_n$ must also reach $\tilde{X}$.
\end{proposition}
\begin{proof} We reason by contradiction, so (after passing to a subsequence) we assume that none of the $F_n$ reaches $\tilde{X}$. Since $(F_n) \longrightarrow F$, there exists a sequence of points $p_n \in F_n$ that converges to $p \in F$. Choose a compact neighbourhood $P$ of $\infty$ in $W^u_{\infty}$ that is disjoint from $F$. Let $C_n$ be the connected component of $F_n$ that contains $p_n$. Notice that each $C_n$ must reach $\infty$ by Proposition \ref{prop:reach} (because it does not reach $\tilde{X}$ by assumption) so $C_n \cap P \neq \emptyset$; also, since $p_n \not\in P$ for big enough because $p \not\in P$, the fact that $C_n$ is connected implies $C_n \cap {\rm fr}(P) \neq \emptyset$. Choosing points $p'_n \in C_n \cap {\rm fr}(P)$, the compactness of ${\rm fr}(P)$ ensures that we may find a subsequence of $(p'_n)$ that converges to some $p' \in {\rm fr}(P)$; hence, $(F_n)$ itself converges to the quasicomponent $F'$ that contains $p'$. Proposition \ref{prop:convergence}.(2) implies that $F = F'$ and so in particular $p' \in F' \cap {\rm fr}(P) = F \cap {\rm fr}(P) \subseteq F \cap P$, which contradicts the choice of $P$.
\end{proof}

\begin{proposition} \label{prop:longer} Let $(F_n) \longrightarrow F$ be a convergent sequence of quasicomponents. Suppose that there exist sequences of points $p_n, q_n \in F_n$ such that (i) $p_n$ and $q_n$ belong to the same component $C_n$ of $U \cap V$ and (ii) $p_n \longrightarrow \infty$ and $q_n \longrightarrow \tilde{X}$ (in the sense that $q_n$ eventually enters any prescribed neighbourhood of $\tilde{X}$). Then $F$ is essential.
\end{proposition}
\begin{proof} Let $P$ and $Q$ be arbitrarily small, disjoint, compact neighbourhoods of $\infty$ and $\tilde{X}$ in $W^u_{\infty}$, respectively. For big enough $n$ we have $p_n \in P$ and $q_n \in Q$ so, since $C_n$ is connected, both intersections $C_n \cap {\rm fr}(P)$ and $C_n \cap {\rm fr}(Q)$ are nonempty. Choose sequences $p'_n \in C_n \cap {\rm fr}(P)$ and $q'_n \in C_n \cap {\rm fr}(Q)$. The compactness of ${\rm fr}(P)$ and ${\rm fr}(Q)$ guarantees that $(p'_n)$ and $(q'_n)$ have subsequences that converge to points $p' \in {\rm fr}(P)$ and $q' \in {\rm fr}(Q')$. The uniqueness of limits established in Proposition \ref{prop:convergence}.(2) implies that $p', q' \in F$. In particular $F$ has a nonempty intersection with $P$ and $Q$ and, since they were arbitrarily small, $F$ reaches both $\infty$ and $\tilde{X}$; that is, it is an essential quasicomponent.
\end{proof}

\begin{remark} \label{rem:subseq} The proof of Proposition \ref{prop:longer} shows that the $C_n$ have a nonempty intersection with the frontier of a suitable small neighbourhood $P$ of $\infty$. Since that frontier is compact, choosing points $p_n \in C_n \cap {\rm fr}(P)$ there is a subsequence $(p_{n_k})$ that converges to some $p \in {\rm fr}(P)$. Letting $F$ be the quasicomponent that contains $p$, we see that $(F_{n_k}) \longrightarrow F$. Thus even if the sequence $(F_n)$ is not required to be convergent, it will have a convergent subsequence $(F_{n_k})$ and the conclusion of the proposition will hold for the subsequence $(F_{n_k})$.
\end{remark}

\subsection{Proof of Lemma \ref{lem:quasi4}} Let us recall the statement of the lemma. Let $O$ be a clopen subset of $W^u_{\infty} \setminus (\tilde{X} \cup \infty)$ that contains all the essential quasicomponents. Define the sets \[G_{\tilde{X}} := \bigcup \{F : F \text{ is a quasicomponent that reaches $\tilde{X}$ and such that } F \cap O = \emptyset\}\] and \[G_{\infty} := \bigcup \{F : F \text{ is a quasicomponent that reaches $\infty$ and such that } F \cap O = \emptyset\}.\] We need to prove that:
\begin{itemize}
	\item[(1)] $W^u_{\infty} \setminus (\tilde{X} \cup \infty) = O \uplus G_{\tilde{X}} \uplus G_{\infty}$.
	\item[(2)] $G_{\tilde{X}}$ is a clopen subset of $W^u_{\infty} \setminus \tilde{X}$ and $G_{\infty}$ is a clopen subset of $W^u_{\infty} \setminus \infty$.
	\item[(3)] $\overline{G}_{\tilde{X}}$ is disjoint from $\infty$ and $\overline{G}_{\infty}$ is disjoint from $\tilde{X}$. Here the closures are taken in $W^u_{\infty}$ (and are therefore compact).
\end{itemize}

Part (1) is easy. Since $O$ is a clopen set of $W^u_{\infty} \setminus (\tilde{X} \cup \infty)$, every quasicomponent $F$ is either contained in $O$ or disjoint from it. By Proposition \ref{prop:reach} every $F$ must reach either $\tilde{X}$ or $\infty$, so clearly $W^u_{\infty} \setminus (\tilde{X} \cup \infty) = O \cup G_{\tilde{X}} \cup G_{\infty}$. By definition $O$ is disjoint from $G_{\tilde{X}}$ and $G_{\infty}$, and these two are also mutually disjoint because a quasicomponent contained in both would be essential and hence contained in $O$, contradicting the definition of $G_{\tilde{X}}$ and $G_{\infty}$.

Now we show that $G_{\infty}$ is closed in $W^u_{\infty} \setminus \infty$. Let $(a_n) \subseteq G_{\infty}$ be a sequence that converges to some $a \in W^u_{\infty} \setminus \infty$. We want to prove that $a \in G_{\infty}$. Denote by $C_n$ and $F_n$ the component and quasicomponent, respectively, that contain $a_n$. By part (1) $F_n$ reaches $\infty$ but not $\tilde{X}$, so the same is true of $C_n$. In particular, we may choose points $p_n \in C_n$ that converge to $\infty$.
\smallskip

{\it Claim.} $a \not\in \tilde{X}$.

{\it Proof of claim.} Assume, on the contrary, that $a \in \tilde{X}$. Then letting $q_n = a_n$ in Proposition \ref{prop:longer} and passing to a subsequence by Remark \ref{rem:subseq} we may assume that $(F_n)$ converges to an essential quasicomponent $F$. In particular there is a sequence of points $b_n \in F_n$ that converges to some point $b \in F$. Since $F \subseteq O$ and $O$ is open in $W^u_{\infty} \setminus (\tilde{X} \cup \infty)$, $b_n \in O$ for big enough $n$ so $F_n \subseteq O$ because $O$ is a clopen set in $W^u_{\infty} \setminus (\tilde{X} \cup \infty)$. Thus in particular $a_n \in O$, which contradicts $a_n \in G_{\infty}$ by part (1).
\smallskip

It follows from the claim that $a \in W^u_{\infty} \setminus (\tilde{X} \cup \infty)$, so it makes sense to consider the quasicomponent $F$ that contains $a$. We then have that $(F_n) \longrightarrow F$ and the contrapositive of of Proposition \ref{prop:quasi2} shows that $F$ must reach $\infty$. It cannot reach $\tilde{X}$ too because then it would be essential and $a \in F \subseteq O$, but $O$ is open in $W^u_{\infty} \setminus (\tilde{X} \cup \infty)$ and so for big enough $n$ we would have $a_n \in O$, which is not the case. Thus $a \in G_{\infty}$, as was to be shown.

An entirely analogous argument shows that $G_{\tilde{X}}$ is closed in $W^u_{\infty} \setminus \tilde{X}$, and by part (1) this implies that $G_{\tilde{X}}$ and $G_{\infty}$ are actually clopen sets of $W^u_{\infty} \setminus (\tilde{X} \cup \infty)$. We use this to prove that $G_{\infty}$ is also open in $W^u_{\infty} \setminus \infty$ by showing that its complement $\tilde{X} \cup O \cup G_{\tilde{X}}$ is closed. Let $(a_n)$ be a sequence in $\tilde{X} \cup O \cup G_{\tilde{X}}$ that converges to a point $a \in W^u_{\infty} \setminus \infty$. We need to show that $a \in \tilde{X} \cup O \cup G_{\tilde{X}}$. If $(a_n)$ has infinitely many terms in $\tilde{X}$ then, since $\tilde{X}$ is compact, $a \in \tilde{X}$ and we are finished. If not, after discarding the first few terms of the sequence we can assume that none of them belongs to $\tilde{X}$ and so the sequence is actually contained in $O \cup G_{\tilde{X}} \subseteq W^u_{\infty} \setminus (\tilde{X} \cup \infty)$. But we know that $O \cup G_{\tilde{X}}$ is closed there, so $a \in O \cup G_{\tilde{X}} \subseteq \tilde{X} \cup O \cup G_{\tilde{X}}$. This finishes the proof of part (2).

Now part (3) is very easy. It will suffice to prove that $\infty \not\in {\overline{G}}_{\tilde{X}}$, since the assertion about ${\overline{G}}_{\infty}$ is entirely analogous. Suppose on the contrary that $\infty \in {\overline{G}}_{\tilde{X}}$. Then clearly $\infty \in \overline{G}^{W^u_{\infty} \setminus \tilde{X}}_{\tilde{X}}$ too, where the closure is now taken in $W^u_{\infty}\setminus \tilde{X}$. However, by part (2) the set $G_{\tilde{X}}$ is closed in $W^u_{\infty} \setminus \tilde{X}$ so $\overline{G}^{W^u_{\infty} \setminus \tilde{X}}_{\tilde{X}} = G_{\tilde{X}}$ and we would have $\infty \in G_{\tilde{X}}$, which is not the case.

\subsection{Proof of Lemma \ref{lem:converge_essential}} Again, let us recall the statement of the lemma: for any quasicomponent $F$, the sequence of forward iterates $(\tilde{f}^nF)_{n \geq 0}$ or the sequence of backward iterates $(\tilde{f}^nF)_{n \leq 0}$ has a subsequence that converges to an essential quasicomponent. To prove this, let $C$ be a component of $F$. By Proposition \ref{prop:reach} $C$ must reach either $\infty$ or $\tilde{X}$. We assume the second to be the case, since the argument in the other case is entirely analogous.

Consider the sequence $C_n := \tilde{f}^n C$ for $n \geq 0$. On the one hand each $C_n$ reaches $\tilde{X}$ so we may certainly construct a sequence $q_n \in C_n$ such that $q_n \longrightarrow \tilde{X}$; on the other, letting $p$ be any point in $C$ and considering its iterates $p_n := \tilde{f}^n(p)$ we see that $p_n$ converges to $\infty$. Thus by Proposition \ref{prop:longer} and Remark \ref{rem:subseq} the sequence $F_n$ has a subsequence that converges to an essential quasicomponent.

\section{\label{app:canonical}}
This section provides a brief discussion on the concept of (compactified) unstable manifold $W_{\infty}$ introduced in Section \ref{sec:background}. More precisely, we prove Propositions \ref{prop:unstablecanonical} and \ref{prop:iteratesappendix}. In words, we show that $W_{\infty}$ does not depend on the choice of index pair employed in its definition and that it is not affected by replacing $f$ by one of its iterates $f^n$. Furthermore, the homeomorphism $\widetilde{f^n}$ induced by $f^n$ in the unstable manifold is conjugate to $(\tilde{f})^n$.

We will first introduce some notation. Throughout this appendix $A, A'$ will be used to denote isolating neighborhoods of an isolated invariant set $X$.
For any positive integer $m$ denote $\Inv_m(f, A)$ the set of points $p \in A$ for which there exists $\{p_i\}_{i=-m}^m$ in $A$ such that $p_0 = p$ and $f(p_{i+1}) = p_i$ for every $-m \le i < m$. In words, $\Inv_m(f, A)$ contains the points $p \in A$ for which we can find an orbit of $f$ of length $2m+1$ contained in $A$ whose middle element is $p$.
If we restrict ourselves to negative semiorbits we obtain $\Inv_m^-(f, A)$, where $p$ belongs to $\Inv_m^-(f, A)$ if there exist $\{p_i\}_{i = 0}^m$ in $A$ such that $f(p_{i+1}) = p_i$. The notation $\Inv^-(f, A)$ is reserved to the set of points $p \in A$ through which there is a full negative semiorbit $\{p_i\}_{i=0}^{\infty}$ contained in $A$. To ease the notation we will typically drop the map $f$ if it is clear from the context.

In the case of homeomorphisms, the previous objects have trivial descriptions in terms of $f$ and $A$ (e.g. $\Inv^-_m(A) = \displaystyle \cap_{k = 0}^m f^{k}(A)$) but if $f$ is non--invertible those descriptions are no longer valid and some caution is needed. The proof of the following lemma (extracted from \cite[Prop 2.2]{franksricheson}) needs some work in this most general situation:

\begin{lemma} \label{lem:invm=inv}
$\Inv^-(A) = \displaystyle \bigcap_{m \ge 0} \Inv^-_m(A)$ \enskip and \enskip$\Inv(A) = \displaystyle \bigcap_{m \ge 0} \Inv_m(A)$.
\end{lemma}

\begin{proof}
Let us focus on the first claim because the second one is a straightforward consequence of the first. One of the inclusions follows easily from the inclusion $\Inv^-_m(A) \subset \Inv^-(A)$ so we just need to prove that if $q \in \Inv_m^-(A)$ for every $m \ge 0$ then $q \in \Inv^-(A)$. We suppose, by hypothesis, that there exist points $\{q^m_i: 1 \le i \le m\}$ in $A$ such that $f(q^m_{i+1}) = q^m_i$ and $f(q^m_1) = q$ for every $1 \le i < m$. Since $A$ is compact, $\{q^m_1\}_{m = 1}^{\infty}$ has at least one accumulation point, say $q_1$. After passing to a subsequence of $m$, we can suppose that $q^m_1$ converges as $m \to \infty$ to $q_1 \in A$. Evidently, $f(q_1) = q$. Next, we consider the set $\{q^m_2: 1 \le m\}$. Again, after passing to a subsequence, we can suppose that $q^m_2$ converges to $q_2 \in A$ and $f(q_2) = q_1$. This procedure can be used to define inductively an infinite backward semiorbit of $q$ in $A$, so $q \in \Inv^-(A)$.
\end{proof}

Now we are ready to prove the following basic fact: if $X$ is an isolated invariant set for $f$ then $X$ is also an isolated invariant set for $f^n$ for any $n \ge 1$ (cf. Proposition \ref{prop:iteratesappendix}):

Let $A$ be an isolating neighborhood of $X$ for $f$ and choose a smaller compact neighbourhood $A'$ of $X$ such that $A' \cup f(A') \cup \ldots \cup f^{n-1}(A') \subseteq A$. By definition, $\Inv(A', f^n) \subseteq \Inv(A', f)$ whereas, by Lemma \ref{lem:invm=inv}, we also have
$$
\Inv(A', f) = \Inv(A, f) = \bigcap_{m \ge 0} \Inv_m(A, f) = \bigcap_{k \ge 0} \Inv_{nk}(A, f) \supset \bigcap_{k \ge 0} \Inv_k(A', f^n) = \Inv(A', f^n),
$$
Thus we conclude that $X = \Inv(A', f^n)$ so $A'$ is an isolating neighborhood of $X$ for $f^n$.

Our next step is to prove that $W_{\infty}$ is a canonical object (Proposition \ref{prop:unstablecanonical}). Recall that the compactified unstable manifold is defined as
$$
W_{\infty} = \varprojlim \left\{ N/L \xleftarrow{f_{\#}} N/L \xleftarrow{f_{\#}} \ldots \right\}
$$
for an index pair $(N, L)$ for $X$ and $f$. Once we remove $\infty$ (recall that it corresponds to $([L], [L], [L], \ldots)$) we are left with a self--similar set that can be recovered from a ``local'' unstable manifold as we shall show below.
If $A$ is an isolating neighborhood of $X$, the limit set of any negative semiorbit contained in $A$ is itself invariant so it must be part of $X$. This suggests that $\Inv^-(A)$ would be a good candidate of ``local'' unstable manifold, should this concept be given any sense.

Let us translate $\Inv^-(A)$ to the language of inverse limits. Write
$$W^A = \varprojlim \left\{\Inv^-(A) \xleftarrow{f} \Inv^-(A) \xleftarrow{f} \ldots \right\}.$$
Observe that we can trivially form a chain of inclusions:
$$W^A \to W^{f(A)} \to W^{f^2(A)} \to \ldots$$
Indeed, if $\xi = (\xi_0, \xi_1, \ldots) \in W^A$ then, by definition, $f(\xi_{i+1})= \xi_i \in f(A)$ for every $i \ge 0$ so $\xi$ belongs to $W^{f(A)}$.

Suppose that $A' \subset A$ is another isolating neighborhood of $X$. Since $X = \displaystyle \cap_{k \ge 0} \mathrm{Inv}_k(A)$, we can find $m \ge 0$ such that $\mathrm{Inv}_m(A) \subseteq A'$.
As a consequence, any point $p$ that belongs to $f^m(\Inv_m(A))$, i.e. there is an orbit of length $2m+1$ inside $A$ that ends with $p$, also belongs to $f^m(A')$.

From this observation we obtain a chain of inclusions that serves to relate the limits $W^A$ for different isolating neighborhoods:
\begin{equation}\label{eq:apC}
W^{A'} \subset W^A \subset W^{f^m(A')}
\end{equation}

Let us explain why we shall think of $W^A$ as the initial part (close to $\tilde{X}$) of $W_{\infty}$.
Consider the union of the increasing sets $\{W^{f^j(A)}\}_{j \ge 0}$, call it $W^X$. As an immediate consequence of Equation (\ref{eq:apC}) we deduce that $W^X$ is independent of the choice of isolating neighborhood $A$.
The topology of $W^X$ comes from those of the sets $W^{f^j(A)}$ which in turn have the topology of the inverse limit. The latter, however,
can be expressed in simpler terms: it is the initial topology with respect to the projection that sends $\xi = (\xi_i)$ to $\xi_0$. Indeed, $\xi, \xi'$ are close elements in the inverse limit $W^A$ iff $\xi_j, \xi'_j$ are close for every $j \ge 0$. Fix an arbitrary neighborhood of $X$, there exist $k$ such that $\xi_j, \xi'_j$ belong to the neighborhood for all $j \ge k$, so we only need to examine the values $j < k$. Therefore, $\xi$ and $\xi'$ are close in $W^A$ iff $\xi_j, \xi'_j$ are close for all $j < k$ which is clearly equivalent to $\xi_0$ being close to $\xi'_0$ in $A$.

In order to prove that $W_{\infty}$ is canonical we will prove that it coincides with the Alexandroff one--point compactification of $W^X$. Equivalently, there is a homeomorphism between $W^X$ and $W_{\infty} \setminus \infty$.

Let $(N, L)$ be the index pair used to define the compactified unstable manifold.
Since, $N \setminus L$ is an isolating neighborhood of $X$ (for the sake of the argument it is irrelevant whether $N\setminus L$ is compact or not), $W^X$ is the union of \{$W^{f^j(N \setminus L)}: j \ge 0\}$.
Take any $\xi = (\xi_0, \xi_1, \ldots) \in W^{f^j(N \setminus L)}$. Clearly, the sequence $(\xi_i)_i$ is eventually contained in $N \setminus L$, say $\xi_i \in N \setminus L$ for every $i \ge m$. We then associate to $\xi$ the element $\mathbf{g}(\xi) = ((f_{\#})^m(\xi_m), \ldots, f_{\#}(\xi_m), \xi_m, \xi_{m+1}, \ldots)$, that does not depend on $m$ provided $m$ is large enough.
It is now an exercise to verify that $\mathbf{g}$ defines a homeomorphism between $W^X = \cup_{j \ge 0} W^{f^j(N \setminus L)}$ and $W_{\infty} \setminus \infty$. This proves that $W_{\infty}$ does not depend on the choice of index pair.

The homeomorphism $\tilde{f}$ is transformed, under conjugation by $\mathbf{g}$, in a map $\overline{f} \colon W^X \to W^X$ that has the simple expression

\centerline{
$\overline{f} \colon (\xi_0, \xi_1, \xi_2, \ldots) \mapsto (f(\xi_0), \xi_0 = f(\xi_1), \xi_1 = f(\xi_2), \ldots).$
}

\medskip

Let us discuss the modifications caused by replacing $f$ by an iterate $f^n$ for a given $n \ge 1$. This substitution makes sense because $X$ is still invariant and isolated for $f^n$, as we showed above.

Consider the inverse limit
$$
W_n^A = \varprojlim \left\{ \Inv^-(f^n, A) \xleftarrow{f^n} \Inv^-(f^n, A) \xleftarrow{f^n} \ldots \right\}.
$$
Denote $A' = A \cup f(A) \cup \ldots \cup f^{n-1}(A)$. There is a natural map $\beta$ that sends
$$
W_n^A \ni \eta = (\eta_i)_{i \ge 0} \mapsto \beta(\eta) = (f^{n-1}(\eta_0), \ldots, f(\eta_0), \eta_0, f^{n-1}(\eta_1), \ldots, \eta_1, \ldots) \in W^{A'}
$$
Intuitively, the map $\beta$ has an inverse that is trivially defined by $$\gamma\colon (\xi_0, \xi_1,\xi_2 \ldots) \mapsto (\xi_{n-1}, \xi_{2n-1}, \xi_{3n-1}, \ldots).$$
Notice however that the image of an element of $W^{A'}$ under $\gamma$ does not necessarily belong to $W_n^A$. An argument similar to the one that lead to Equation (\ref{eq:apC}) yields an integer $m$ such that the image of $\gamma$ is contained in $W_n^{f^m(A)}$. The composition map
$$
W_n^A \xrightarrow{\beta} W^{f^{n-1}(A)} \xrightarrow{\gamma} W_n^{f^m(A)}
$$
is the inclusion. The definitions of $\beta$ and $\gamma$ can be extended to $W_n^{f^j(A)}$ and $W^{f^j(A')}$, respectively, and thus to the union
$$
W_n^X = \bigcup_{j \ge 0} W_n^{f^j(A)} \xrightarrow{\quad\beta\quad} W^X = \bigcup_{j \ge 0} W^{f^j(A)} \xrightarrow{\quad\gamma\quad} W_n^X = \bigcup_{j \ge 0} W_n^{f^j(A)},
$$
It is easy to see that these extended versions $\beta$ and $\gamma$ are one--to--one and the composition $\gamma \circ \beta$ is the identity map. We conclude that $W_n^X$ is homeomorphic to $W^X$.

Let $W_{n, \infty}$ be the compactified unstable manifold for $f^n$. Recall that $W_{n, \infty} \setminus \infty$ is identified to $W_n^X$ and the action of $\widetilde{f^n}$ becomes

\centerline{
$\overline{f^n}(\eta_0, \eta_1, \eta_2, \ldots) = (f^n(\eta_0), f^n(\eta_1) = \eta_0 , f^n(\eta_2) = \eta_1, \ldots).$
}

Then, it follows easily that
$$
\beta \circ \overline{f^n} = (\overline{f})^n \circ \gamma,
$$
and we deduce that the maps $\widetilde{f^n}$ and $(\tilde{f})^n$ are conjugate. This concludes Proposition \ref{prop:iteratesappendix}.

\bibliographystyle{plain}
\bibliography{biblio}

\end{document}